\documentclass[11pt]{amsart}
\usepackage{amscd,amsxtra,amssymb,mathrsfs, bbm}
\usepackage{stmaryrd}
\usepackage{geometry}
        {\begin{quotation}\begin{center}\begin{em}}
        {\par\end{em}\end{center}\end{quotation}}

\newtheorem{theorem}{Theorem}[section]
\newtheorem{corollary}[theorem]{Corollary}
\newtheorem{lemma}[theorem]{Lemma}
\newtheorem{proposition}[theorem]{Proposition}

\theoremstyle{definition}
\newtheorem{definition}[theorem]{Definition}
\newtheorem{remark}[theorem]{Remark}
\newtheorem{example}[theorem]{Example}

\DeclareMathOperator{\Perf}{\mathsf{Perf}}

\DeclareMathOperator{\add}{\mathsf{add}}

\DeclareMathOperator{\wt}{\mathsf{wt}}
\DeclareMathOperator{\rad}{\mathsf{rad}}

\DeclareMathOperator{\SL}{\mathsf{SL}}

\renewcommand{\ker}{\mathsf{ker}}

\DeclareMathOperator{\Coh}{\mathsf{Coh}}
\DeclareMathOperator{\Qcoh}{\mathsf{QCoh}}
\DeclareMathOperator{\Tor}{\mathsf{Tor}}

\DeclareMathOperator{\VB}{\mathsf{VB}}

\DeclareMathOperator{\Hom}{\mathsf{Hom}}
\DeclareMathOperator{\RHom}{\mathsf{RHom}}

\DeclareMathOperator{\Ext}{\mathsf{Ext}}

\DeclareMathOperator{\End}{\mathsf{End}}
\DeclareMathOperator{\Mat}{\mathsf{Mat}}
\DeclareMathOperator{\Spec}{\mathsf{Spec}}

\DeclareMathOperator{\idm}{\mathfrak{m}}

\newcommand{\bul}{\scriptstyle{\bullet}}

\newcommand{\kk}{\mathbbm{k}}
\newcommand{\CC}{\mathbb{C}}
\newcommand{\NN}{\mathbb{N}}
\newcommand{\ZZ}{\mathbb{Z}}

\newcommand{\sm}{\mathsf{m}}

\newcommand{\FF}{\mathbb{F}}

\newcommand{\EE}{\mathbb{E}}

\newcommand{\XX}{\mathbb{X}}
\newcommand{\YY}{\mathbb{Y}}
\renewcommand{\SS}{\mathbb{S}}

\newcommand{\PP}{\mathbbm{P}}

\newcommand{\Ad}{\mathsf{Ad}}

\newcommand{\sMat}{\mbox{\emph{Mat}}}
\newcommand{\sHom}{\mbox{\it{Hom}}}
\newcommand{\sExt}{\mbox{\it{Ext}}}

\newcommand{\kA}{\mathcal{A}}

\newcommand{\kB}{\mathcal{B}}
\newcommand{\kC}{\mathcal{C}}
\newcommand{\kE}{\mathcal{E}}

\newcommand{\kH}{\mathcal{H}}
\newcommand{\kI}{\mathcal{I}}

\newcommand{\kO}{\mathcal{O}}
\newcommand{\kL}{\mathcal{L}}
\newcommand{\kP}{\mathcal{P}}
\newcommand{\kQ}{\mathcal{Q}}

\newcommand{\kT}{\mathcal{T}}
\newcommand{\kS}{\mathcal{S}}
\newcommand{\kX}{\mathcal{X}}

\newcommand{\tX}{\widetilde{X}}
\newcommand{\tZ}{\widetilde{Z}}
\newcommand{\tx}{\tilde{x}}
\newcommand{\tH}{\widetilde{\mathcal{H}}}

\newcommand{\lar}{\longrightarrow}

\newcommand{\llbrace}{(\!(}
\newcommand{\rrbrace}{)\!)}

 \def\sP{\mathsf P}
\def\sD{\mathsf D} 
\def\sE{\mathsf E} \def\sR{\mathsf R}
\def\sF{\mathsf F} \def\sS{\mathsf S}
\def\sG{\mathsf G} \def\sT{\mathsf T}
\def\sH{\mathsf H} 
\def\sI{\mathsf I} 
\def\sJ{\mathsf J} 
\def\sK{\mathsf K} 
\def\sL{\mathsf L}

\def\kron#1#2{\xymatrix@C=2em{{#1}\ar@/^3pt/[r]\ar@/_3pt/[r]&{#2}}}
\def\bu{{\scriptscriptstyle\bullet}}

\usepackage{tikz}
\usetikzlibrary{calc, arrows, positioning, shapes, fit, matrix, decorations}
\usetikzlibrary{decorations.shapes, decorations.pathreplacing}

\tikzset{
  decorate with/.style={decorate,decoration={shape backgrounds,shape=#1,shape size=1.5mm}},
   deco/.style={decorate with=dart},
   ordi/.style={draw,-stealth,  thick},
   conj/.style={dashed, draw, thick},
   ve/.style={circle, draw, thick, fill=blue!20, inner sep=1pt, outer sep=2pt, minimum size=7pt},
    dot/.style={fill=blue!10,circle,draw, inner sep=1pt, minimum size=5pt},
  dv/.style={star,star points=5,
star point ratio=2, draw, thick, fill=green!20, inner sep=1pt,outer sep=2pt,minimum size=7pt}
}

\tikzset{
    tbl5 nodes/.style={
        rectangle,
        execute at begin node=$,
       execute at end node=$,
       fill=blue!5,
        align=center,
        text depth=0.5ex,
        text height=2ex,
        inner xsep=0pt,
        outer sep=0pt,
           },
    tbl5/.style={
        matrix of nodes,
        row sep=-\pgflinewidth,
        column sep=-\pgflinewidth,
        nodes={
            tbl5 nodes
        },
        execute at empty cell={\node[draw=none]{};}
    }
  }

\input xy
\xyoption{all}

\title[non-commutative nodal curves and derived tame algebras]{non-commutative nodal curves and   derived tame algebras}

\author{Igor Burban}
\address{
Universit\"at Paderborn\\
Institut f\"ur Mathematik \\
Warburger Stra\ss{}e 100 \\
33098 Paderborn \\
Germany
}
\email{burban@math.uni-paderborn.de}

\author{Yuriy Drozd}
\address{
 Institute of Mathematics\\
National Academy of Sciences of Ukraine,
Tereschenkivska str. 3,
01004 Kyiv, Ukraine}
\email{drozd@imath.kiev.ua, y.a.drozd@gmail.com}


\begin{document}

\begin{abstract}
In this paper,  we develop a geometric approach to study derived tame finite dimensional associative algebras, based on the theory of non-commutative nodal curves.
\end{abstract}

\maketitle

\tableofcontents

\section{Introduction}

\noindent
Let $\kk$ be an algebraically closed field such that $\mathsf{char}(\kk) \ne 2$. For $\lambda \in \kk \setminus \{0, 1\}$, let
$$Y_\lambda  = V\bigl(zy^2 - x(x-z)(x-\lambda z)\bigr) \subset \PP^2$$
be the corresponding  elliptic curve and   $Y_\lambda  \stackrel{\imath}\lar Y_\lambda$ be the involution, given by the rule
 $(x:y:z) \stackrel{\imath}\mapsto (x:-y:z)$.  Let
 $T_\lambda$ be the tubular canonical algebra of type $((2,2,2,2);\lambda)$ \cite{Ringel}, i.e.~the path algebra of the following quiver
 \begin{equation}\label{E:tubular}
\begin{array}{c}
\xymatrix
{
        &           & \circ \ar[lld]_{a_1}
\ar[ld]^{a_2}  \ar[rd]_{a_3}  \ar[rrd]^{a_4}
      &         &       \\
\circ \ar[rrd]_{b_1}  & \circ \ar[rd]^{b_2}
&        & \circ \ar[ld]_{b_3}
& \circ \ar[lld]^{b_4}\\
        &           &\circ &         &       \\
}
\end{array}
\end{equation}
modulo the  relations $b_1 a_1 - b_2 a_2 = b_3 a_3$ and  $b_1 a_1 - \lambda b_2 a_2 = b_4 a_4$.
According to  Geigle and Lenzing \cite[Example 5.8]{GeigleLenzing}, there exists an exact equivalence
 of triangulated  categories
 \begin{equation}\label{E:standardtilting}
 D^b\bigl(\Coh^G(Y_\lambda)\bigr) \lar D^b(T_\lambda\mathsf{-mod}),
 \end{equation}
 where $G = \langle \imath\rangle \cong \ZZ_2$ and $\Coh^G(Y_\lambda)$ is the category of $G$-equivariant coherent sheaves on $Y_\lambda$.
It is well-known that   $D^b\bigl(\Coh^G(Y_\lambda)\bigr)$ and  $D^b(T_\lambda\mathsf{-mod})$ have  \emph{tame}
 representation type; see \cite{Atiyah, HappelRingel, LenzingMeltzer}. At this place one can ask the following natural

\smallskip
\noindent
\textbf{Question}. Is there any  link between  $ D^b\bigl(\Coh^G(Y_\lambda)\bigr)$ and $D^b(T_\lambda\mathsf{-mod})$ when the parameter $\lambda \in \kk$ takes the ``forbidden'' value $0$?

\medskip
\noindent
Let $E := Y_0 = V\bigl(zy^2 - x^2(x-z)\bigr) \subset \PP^2$ be the  plane nodal cubic and $T := T_0$ be the corresponding degenerate tubular algebra.
Both derived categories $D^b\bigl(\Coh^G(E)\bigr)$ and $D^b(T\mathsf{-mod})$ are known to be representation tame. Moreover, it follows from  our previous papers
\cite{Duke, Toronto} that the  indecomposable objects in both categories can be described by very similar combinatorial patterns. However, since $\mathsf{gl.dim}\bigl(\Coh^G(E)\bigr) = \infty$ and $\mathsf{gl.dim}(T) = 2$,
the derived categories $D^b\bigl(\Coh^G(E)\bigr)$ and $D^b(T\mathsf{-mod})$ can not be equivalent. Nevertheless, it turns out that the following result is true:

\smallskip
\noindent
\textbf{Proposition} (see Remark \ref{R:DegTub}).
There exists a commutative diagram of triangulated categories and  functors
\begin{equation}\label{E:TiltingDegTub}
\begin{array}{c}
\xymatrix{
D^b(T\mathsf{-mod}) \ar@{->>}[rr]^-{\sP} & & D^b\bigl(\Coh^G(E)\bigr) \\
& \Perf^G(E) \ar@{_{(}->}[lu]^-{\sE} \ar@{^{(}->}[ru]_-{\sI} &
}
\end{array}
\end{equation}
where $ \Perf^G(E)$ is the perfect derived category of $\Coh^G(E)$, $\sI$ is the canonical inclusion functor, $\sE$ is a fully faithful functor and $\sP$ is an appropriate Verdier localization functor.

\smallskip
\noindent
The main goal of this work is  to extend the above result to a broader  class of derived tame algebras.  Let
us start with a pair of tuples 
$
\vec{p} = \bigl((p_1^+, p_1^-), \dots, (p_r^+, p_r^-)\bigr) \in \bigl(\NN^2\bigr)^r$ and  $
\vec{q} =  (q_1, \dots, q_s) \in \NN^s,
$
where  $r, s\in  \NN_0$ (either of this tuples is allowed to be empty).  For any $1 \le i \le r$ and $1 \le j \le s$, consider  the following sets:
$
\Xi_i^\pm := \bigl\{x_{i,1}^\pm, \dots, x_{i, p_i^\pm}^\pm\bigr\} \quad \mbox{\rm and} \quad
\Xi_j^\circ:= \bigl\{w_{j,1}, \dots, w_{j,q_j}\bigr\}.
$
Let $\approx$ be a symmetric relation (not necessarily an equivalence) on the set
$
\Xi:= \bigl((\Xi_1^+ \cup \Xi_1^-) \cup \dots \cup (\Xi_r^+ \cup \Xi_r^-)\bigr) \cup \bigl(\Xi_1^\circ \cup \dots \cup  \Xi_s^\circ \bigr)
$
 such that for any $\xi \in \Xi$, there exists at most one $\xi' \in \Xi$ such that $\xi \approx \xi'$.
Then the  datum $(\vec{p}, \vec{q}, \approx)$ defines a derived tame finite dimensional
$\kk$-algebra $\Lambda = \Lambda(\vec{p}, \vec{q}, \approx)$, obtained from the canonical algebras $\Gamma(p_1^+, p_1^-), \dots, \Gamma(p_r^+, p_r^-), \Gamma(2,2, q_1), \dots, \Gamma(2,2, q_s)$ by a certain ``gluing/blowing-up process'' determined by the relation $\approx$.
In the case when $s = 0$ (i.e.~when the tuple $\vec{q}$ is void), the algebra $\Lambda$ is skew-gentle  \cite{GeissDelaPena}. If additionally $\xi \not\approx \xi$ for all $\xi \in \Xi$, the algebra $\Lambda$ is gentle \cite{AssemSkowr}.
Instead of defining  $\Lambda(\vec{p}, \vec{q}, \approx)$ now, we refer to  Definition \ref{D:DerivedTameAlgebras} below and give here  two examples explaining characteristic features  of this class of algebras.

\smallskip
\noindent
\textbf{Example}. Let $\vec{p} = (3, 2)$, $\vec{q}$ be void (i.e.~$r = 1$ and $s = 0$) and $\approx$ be given by the rule
$x_{1,1}^+ \approx x_{1,1}^-$ and $x_{1,3}^+ \approx x_{1,2}^-$.
Then the corresponding gentle algebra $\Lambda(\vec{p}, \approx)$ is the path algebra of the following quiver
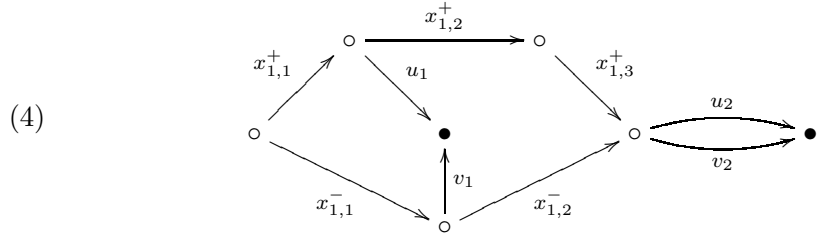
\begin{equation}
\begin{array}{c}
\xymatrix{
              & \circ \ar[rd]^-{u_1} \ar[rr]^-{x_{1,2}^+} &  & \circ \ar[rd]^-{x_{1,3}^+} & & \\
\circ \ar[ru]^-{x_{1, 1}^+} \ar[rrd]_{x_{1,1}^-} &  & \bullet        &          & \circ    \ar@/^/[rr]^-{u_2}  \ar@/_/[rr]_-{v_2}  & &  \bullet   &   \\
 &  &   \circ \ar[u]_-{v_1}  \ar[rru]_-{x_{1,2}^-} & & &
}
\end{array}
\end{equation}
subject to the relations: $u_1 x_{1,1}^+ = 0 = v_1 x_{1,1}^-$ and $u_2 x_{1,3}^+ = 0 = v_2 x_{1,2}^-$.

\smallskip
\noindent
\textbf{Example}. Let $\vec{p} = \bigl((1, 1), (1, 1)\bigr)$, $\vec{q} = (2)$ and $\approx$ be given by the rule:
$x_{1,1}^+ \approx w_{1,1}$, $x_{1,1}^- \approx x_{2,1}^+$ and $w_{1,2} \approx w_{1,2}$. Then the corresponding algebra $\Lambda(\vec{p}, \vec{q}, \approx)$ is the path algebra of the following quiver
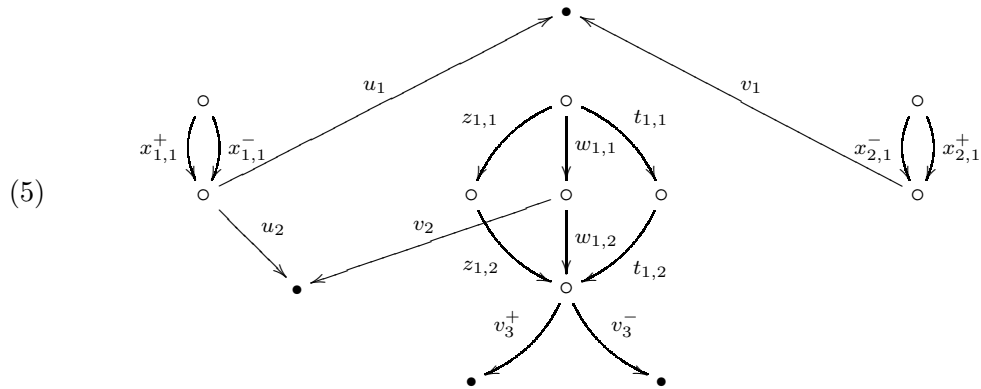
\begin{equation}
\begin{array}{c}
\xymatrix{
& & &  & \bul &   & & &  & \\
\circ \ar@/^/[d]^-{x_{1,1}^-} \ar@/_/[d]_-{x_{1,1}^+} & & & & \circ \ar[d]^-{w_{1,1}} \ar@/^/[rd]^-{t_{1,1}} \ar@/_/[ld]_-{z_{1,1}} & & & & \circ \ar@/^/[d]^-{x_{2,1}^+} \ar@/_/[d]_-{x_{2,1}^-} \\
\circ \ar[rd]^-{u_2} \ar[rrrruu]^-{u_1} & & & \circ \ar@/_/[rd]_-{z_{1,2}} & \circ \ar[d]^-{w_{1,2}} \ar[llld]_-{v_2}& \circ \ar@/^/[ld]^-{t_{1,2}} & & & \circ \ar[lllluu]_-{v_1} \\
 &  \bul & & & \circ \ar@/^/[ld]_-{v_3^+} \ar@/_/[rd]^-{v_3^-} & &  & &  & \\
 & & & \bul & &  \bul & & &  & \\
}
\end{array}
\end{equation}
modulo the relations:
$$
\left\{
\begin{array}{l}
z_{1,2} z_{1,1} + w_{1,2} w_{1,1} + t_{1,2} t_{1,1} = 0 \\
v_3^\pm w_{1,2} = 0 \\
u_1 x_{1,1}^- = 0 = v_1 x_{2,1}^+ \\
u_2 x_{1,1}^+ = 0 = v_2 w_{1,1}.
\end{array}
\right.
$$

\smallskip
\noindent
Let $(\vec{p}, \vec{q}, \approx)$ be a datum as in the definition of $\Lambda(\vec{p}, \vec{q}, \approx)$, additionally satisfying  a certain \emph{admissibility} condition. It turns out that  it  defines (uniquely up to Morita equivalence) a \emph{tame non-commutative projective nodal curve}
$\XX = \XX(\vec{p}, \vec{q}, \approx)$. Converesely, any tame non-commutative projective nodal curve is Morita equivalent to  $\XX(\vec{p}, \vec{q}, \approx)$ for an appropriate admissible datum  $(\vec{p}, \vec{q}, \approx)$,  see \cite{DrozdVoloshyn}. This class of non-commutative nodal curves includes as a  special case stacky chains and cycles of projective lines, considered by Lekili and Polishchuk \cite{LekiliPolishchuk, LekiliPolishchukII} in the context of the homological mirror symmetry for compact surfaces with non-empty boundary; see Example \ref{Ex:StackyCycles}
for a detailed discussion.

\smallskip
\noindent
The main result of this paper is the following.

\smallskip
\noindent
\textbf{Theorem} (see Corollary \ref{C:corollaryMain}). Let $(\vec{p}, \vec{q}, \approx)$  be an admissible datum, 
$\XX$ be the corresponding tame non-commuta\-ti\-ve nodal curve and
$\Lambda$ be the corresponding $\kk$-algebra. Next, let  $\YY$ be the \emph{Auslander curve} of $\XX$ (which is another tame non-commutative projective nodal curve) and $\widetilde\XX
\stackrel{\nu}\lar \XX$ be the  \emph{hereditary cover} of $\XX$.
Then there exists  the following commutative diagram of triangulated categories and exact functors:
\begin{equation}\label{E:Functors}
\begin{array}{c}
\xymatrix{
                                    & D^b\bigl(\Coh(\widetilde{\XX})\bigr) \ar[ld]_{\nu_*}  \ar@{^{(}->}[d]^-{\widetilde\sE} &                          \\
D^b\bigl(\Coh(\XX)\bigr) & D^b\bigl(\Coh(\YY)\bigr) \ar@{->>}[l]_-{\sP}  \ar[r]^-{\sT}                  & D^b(\Lambda\mathsf{-mod}) \\
                         & \Perf(\XX) \ar@{_{(}->}[u]_-{\sE}  \ar@{_{(}->}[lu]^-{\sI} &
}
\end{array}
\end{equation}
where  $\sT$ is an   equivalence of triangulated categories, $\sE$ and $\widetilde\sE$  are fully faithful  functors, $\sI$ is the canonical inclusion,
$\sP$ is an appropriate localization functor
and $\nu_\ast$ is  induced by the forgetful functor $\Coh(\widetilde\XX) \lar \Coh(\XX)$.

\smallskip
\noindent
This theorem generalizes an earlier results of the authors  \cite{bd}, where $\XX$ was a \emph{commutative} tame nodal curve (i.e.~a chain or a cycle of projective lines \cite{DGVB, Duke}). 

\smallskip
\noindent
If $\vec{q}$ is void and $ \approx$ does not contain reflexive elements 
then the algebra $\Lambda(\vec{p}, \approx)$ is gentle. In Lemma
\ref{L:AGinvariant},
we compute the AG-invariant of Avella-Alaminos and Gei\ss{}  
\cite{AAGeiss} of $\Lambda(\vec{p}, \approx)$. Recent results of Lekili and Polishchuk \cite{LekiliPolishchuk, LekiliPolishchukII} allow to deduce from it a version  of the homological mirror symmetry for tame non--commutative nodal curves of gentle type.

\smallskip
\noindent
At this place we want to stress  that the introduced class of algebras $\Lambda(\vec{p}, \vec{q}, \approx)$ does not exhaust (even up to derived equivalence) all  derived tame algebras which are derived equivalent to an appropriate non-commutative tame projective nodal curve. For example, in the paper \cite[Theorem 2.1]{Burban} it was  observed  that on a chain of projective lines there exists a tilting bundle whose
endomorphism algebra  is a  gentle algebra of \emph{infinite} global dimension. In this paper, we have found another  class of gentle algebras which  are derived equivalent to appropriate non-commutative tame projective nodal curves.  For any $n \in \NN$, let $\Upsilon_n$ be  the path algebra  of the following quiver
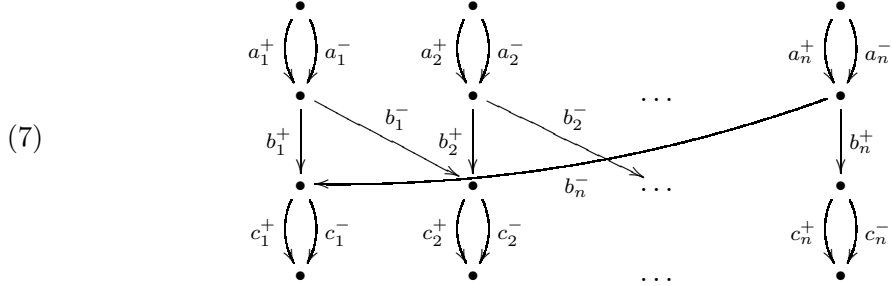
\begin{equation}\label{E:GentleGlobDimThree}
\begin{array}{c}
\xymatrix{
\bul   \ar@/_/[d]_-{a_1^+} \ar@/^/[d]^-{a_1^-} & &\bul   \ar@/_/[d]_-{a_2^+} 
\ar@/^/[d]^-{a_2^-} & &  & & \bul   \ar@/_/[d]_-{a_n^+} \ar@/^/[d]^-{a_n^-}\\ 
\bul  \ar[d]_-{b_1^+} \ar[rrd]^-{b_1^-} & & \bul  \ar[d]_-{b_2^+} \ar[rrd]^-{b_2^-} & & \dots & &  \bul  \ar[d]^-{b_n^+} 
\ar@/^10pt/[dllllll]^-{b_n^-} \\ 
\bul   \ar@/_/[d]_-{c_1^+} \ar@/^/[d]^-{c_1^-} 
& & \bul   \ar@/_/[d]_-{c_2^+} \ar@/^/[d]^-{c_2^-} & & \dots & & \bul   
\ar@/_/[d]_-{c_n^+} \ar@/^/[d]^-{c_n^-}\\ 
\bul & & \bul & & \dots & & \bul
}
\end{array}
\end{equation}
modulo the relations
$$
b_i^\pm a_i^\mp = 0, c_i^- b_i^+ = 0 \quad \mbox{\rm and} \quad c_{i+1}^+ b_i^- = 0 \quad \mbox{\rm for} \quad 
1 \le i \le n.
$$

\smallskip
\noindent
Since  $\mathsf{HH}^3(\Upsilon_n) \ne 0$, the algebra $\Upsilon_n$  can  not be derived equivalent to any gentle algebra of the form $\Lambda(\vec{p},\approx)$. On the other hand, we prove
that $D^b(\Upsilon_n\mathsf{-mod})$ is equivalent to  the derived category of coherent sheaves on the so-called Zhelobenko non-commutative cycle of projective lines; see 
Theorem \ref{T:ZhelobenkoTilting}.

\smallskip
\noindent
\emph{Acknowledgement}.  We are thankful to Christoph Gei\ss{}  for giving us a hint of the  construction of  the derived equivalence from Proposition \ref{P:Geiss} as well as  to Yanki Lekili
for explaining us  connections between gentle algebras and various
versions of the Fukaya category of a compact surface with non-empty boundary. The first-named author is also indebted  to Nicolo Sibilla for communicating him the statement of Proposition \ref{P:Sibilla} as well as for helpful discussions. The work of the first-named author was partially supported by the DFG project Bu--1866/4--1. The results of this paper
were mainly obtained during the stay of the second-named author at Max-Planck-Institut f\"ur Mathematik in Bonn.

\clearpage
\section{Some algebraic prerequisities}

\subsection{Brief review of the theory of minors}\label{SS:Minors}

\noindent
Throughout this subsection, let $R$ be a commutative noetherian ring. For any
$R$-algebra $C$, we denote by $C^\circ$ the opposite $R$-algebra, by $C\mathsf{-mod}$ (respectively, by $\mathsf{mod}-C$) the category of finitely generated left (respectively, right) $C$-modules  and
by $C\mathsf{-mod}$ (respectively, $\mathsf{Mod}-C$) the category of all  left (respectively, right)  $C$-modules. For any $C$-module  $X$ we denote by
$\add(X)$ the additive closure of $X$, i.e.~the full subcategory of $C\mathsf{-mod}$ consisting of all  direct summands of all finite direct sums of $X$.

\smallskip
\noindent
In what follows, $B$ is  an $R$-algebra, which is finitely generated as $R$-module.

\begin{definition}\label{D:Minor}
Let $P$ be a finitely generated projective left $B$-module. Then the $R$-algebra $A:= \bigl(\End_B(P)\bigr)^\circ$ is called a \emph{minor} of $B$; see \cite{Drozd, Minors}.
\end{definition}

\smallskip
\noindent
It is clear that $P$ is a $(B$-$A)$-bimodule and we have an exact functor
$$
\sG = \Hom_B(P, \,-\,): B\mathsf{-mod} \stackrel{\sG}\lar A\mathsf{-mod}.
$$
\begin{remark}
In the case  $P$ is a \emph{projective generator} of the category $B\mathsf{-mod}$ (meaning that for any object $X$ of $B\mathsf{-mod}$ there exists an epimorphism
 $P^n \twoheadrightarrow 	 X$ for some $n \in \NN$), Morita theorem asserts that the functor  $\sG$ is an equivalence of categories. It is well-known that $\sG$ induces an isomorphism of centers $Z(B) \stackrel{\sG_c}\lar Z(A)$ (see e.g. \cite[Section 2.3]{BurbanDrozdMorita} and the references therein). If $Z(B) = R = Z(A)$ and the induced map $R \stackrel{\sG_c}\lar R$ is identity, we say that $\sG$ is a \emph{central} Morita equivalence. 
\end{remark}

\smallskip
\noindent
It is not difficult to prove the following result.
\begin{lemma}\label{L:TechnicalitiesMinors} Consider the dual (right) $B$-module $P^\vee := \Hom_B(P, B)$.
The following statements hold.
\begin{itemize}
\item The canonical morphism  $P \lar P^{\vee\vee} = \Hom_B(P^\vee, B)$ is an isomorphism of left $B$-modules. Moreover, the canonical morphism of $R$-algebras
$$\End_B(P) \lar \bigl(\End_B(P^\vee)\bigr)^{\circ}$$ is an isomorphism too.
\item For any object $X$ of $B\mathsf{-mod}$, the canonical morphism of left $A$-modules
$$
P^\vee \otimes_B X \lar \Hom_B(P, X), \quad l \otimes x \mapsto \bigl(y \mapsto l(y) \cdot x\bigr)
$$
is an isomorphism, i.e.~we have an isomorphism of functors $\sG \cong \,-\, \otimes_{B} P^\vee$. As a consequence,  $P^\vee$ is a flat (actually, even  projective) right  $B$-module.
\item In particular, the canonical morphism
$$P^\vee \otimes_B P \lar \Hom_B(P, P), \quad l \otimes y \mapsto \bigl(x \mapsto l(y) \cdot x\bigr)
$$
is an isomorphism of $\bigl(A$-$A\bigr)$-bimodules.
\end{itemize}
\end{lemma}

\smallskip
\noindent
Using Lemma \ref{L:TechnicalitiesMinors}, one can deduce the following results.

\begin{theorem}\label{T:minorsabelian}
Consider the functors  $\sF := P \otimes_A \,-\,$ and $\sH := \Hom_A(P^\vee, \,-\,)$ from $A\mathsf{-mod}$ to $B\mathsf{-mod}$. Then the following statements hold.
\begin{itemize}
\item The functors $(\sF, \sG, \sH)$ form  an adjoint triple, i.e.~$(\sF, \sG)$ and $(\sG, \sH)$ form  adjoint pairs.
\item The functors $\sF$ and $\sH$ are fully faithful, whereas $\sG$ is essentially surjective.

\item Let $I_P := \mathrm{Im}\bigl(P\otimes_A P^\vee \stackrel{\mathrm{ev}}\lar B\bigr)$. Then $I_P$ is a
$(B$-$B)$-bimodule and
$$
\ker(\sG) := \bigl\{X \in B\mathsf{-mod} \,\big|\, I_P X = 0\bigr\}.
$$
In other words, the kernel
 of the exact functor $\sG$ can be identified with the essential image of the (fully faithful) restriction  functor $\bar{B}\mathsf{-mod} \lar B\mathsf{-mod}$, where
  $\bar{B} = B/I_P$.
Moreover, the functor $\sG$ induces an equivalence of categories $$B\mathsf{-mod}/\ker(\sG) \lar A\mathsf{-mod}.$$
\item The same results remain true, when we consider $\sG$, $\sF$ and $\sH$ as functors between the categories $B\mathsf{-mod}$ and $A\mathsf{-mod}$ of finitely generated modules.
\item The essential image of the functor $\sF: A\mathsf{-mod} \lar B\mathsf{-mod}$ is the category
\begin{align*}
P\mathsf{-mod}:= \Bigl\{X \in B\mathsf{-mod} \,\big|\, &\mbox{\rm there exists an exact sequence} \\
  & P_1 \lar P_0 \lar X \lar 0 \; \mbox{\ \rm with}\; P_0, P_1  \in \mathsf{add}(P)\; \Bigr\}.
\end{align*}
\end{itemize}
\end{theorem}

\smallskip
\noindent
It turns out that the relation between $B\mathsf{-mod}$, $A\mathsf{-mod}$ and $\bar{B}\mathsf{-mod}$ becomes even more transparent, when we pass to the setting  of  derived categories.

\begin{theorem}\label{T:minorsderived} Let $\sD\sG: D(B\mathsf{-mod}) \lar D(A\mathsf{-mod})$ be the derived functor of (the exact) functor $\sG$, $\sL \sF: D(A\mathsf{-mod}) \lar D(B\mathsf{-mod})$ be the left derived functor of (the right exact functor) $\sF$ and $\sR \sH: D(A\mathsf{-mod}) \lar D(B\mathsf{-mod})$ be the right derived functor of (the left exact functor) $\sH$. Then the following results hold.
\begin{itemize}
\item $(\sL \sF, \sD\sG, \sR\sH)$ is an  adjoint triple  of functors.
\item The functors $\sL\sF$ and $\sR\sH$ are fully faithful, whereas  $\sD\sG$ is essentially surjective.
\item The essential image of $\sL\sF$ is equal to the left orthogonal category
$$
^{\perp} \bar{B} := \bigl\{X^\bu \in \mathrm{Ob}\bigl(D(B\mathsf{-mod})\bigr) \;\big| \; \Hom\bigl(X^\bu, \bar{B}[i]\bigr) = 0 \; \mbox{\rm for all} \; i \in \ZZ\bigr\}
$$
of $\bar{B}$ (viewed as a left $B$-module).
Similarly, the essential image of $\sR\sH$ is equal to the right orthogonal category $\bar{B}^\perp$.
\item We have a recollement diagram
\begin{equation}
\xymatrix{D_{\bar{B}}(B\mathsf{-mod}) \ar[rr]|{\,\sI\,} && D(B\mathsf{-mod}) \ar@/^2ex/[ll]^{\sI^{!}} \ar@/_2ex/[ll]_{\sI^*} \ar[rr]|{\,\sD\sG\,}
  && D(A\mathsf{-mod}) \ar@/^2ex/[ll]^{\,\sR\sH\,} \ar@/_2ex/[ll]_{\,\sL\sF\,}},
\end{equation}
where $D_{\bar{B}}(B\mathsf{-mod})$ is the full subcategory of $D(B\mathsf{-mod})$ consisting of those complexes whose cohomologies belong to
$\bar{B}\mathsf{-mod}$.
\item Assume that the $(B$-$B)$-bimodule $I_P$ is flat viewed as a right $B$-module.
\begin{itemize}
\item Then the functor
$
D\bigl(\bar{B}\mathsf{-mod}\bigr) \lar D_{\bar{B}}(B\mathsf{-mod})
$
is an equivalence of triangulated categories.
\item We have: $\mathsf{gl.dim}B \le \mathsf{max}\bigl\{\mathsf{gl.dim}\bar{B} +2, \mathsf{gl.dim}A\bigr\}$.
\item Finally, suppose  that $\mathsf{gl.dim}\bar{B} < \infty$ and $\mathsf{gl.dim}A < \infty$. Then we have a recollement diagram
\begin{equation}
\xymatrix{D^b(\bar{B}\mathsf{-mod}) \ar[rr]|{\,\sI\,} && D^b(B\mathsf{-mod}) \ar@/^2ex/[ll]^{\sI^{!}} \ar@/_2ex/[ll]_{I^*}
 \ar[rr]|{\,\sD\sG\,}
  && D^b(A\mathsf{-mod}) \ar@/^2ex/[ll]^{\,\sR\sH\,} \ar@/_2ex/[ll]_{\,\sL\sF\,}}.
\end{equation}
\end{itemize}
\end{itemize}
\end{theorem}

\begin{remark}
In the case $P = Be$ for an idempotent $e \in B$, most of the results from this subsection are due to Cline, Parshall and Scott \cite{ClineParshallScott}.
The  ``abelian'' theory of minors attached to  an arbitrary finitely generated projective $B$-module $P$  was for the first time suggested in \cite{Drozd}. Detailed proofs of Theorems \ref{T:minorsabelian} and \ref{T:minorsderived} can also be found in  \cite{Minors}.
\end{remark}

\subsection{Generalities on  orders}
From now on, let $R$ be  an excellent  reduced equidimensional commutative  ring of \emph{Krull dimension one}. Let  
\begin{equation}\label{E:splittingQuotRing}
K := \mathsf{Quot}(R) \cong K_1 \times \dots \times K_r
\end{equation}
be  the corresponding  total ring of fractions, where $K_1, \dots, K_r$ are fields.

\begin{definition}
An $R$-algebra $A$ is an $R$-\emph{order} if the following conditions hold.
\begin{itemize}
\item $A$ is a finitely generated torsion free $R$-module.
\item $A_K:= K \otimes_R A$ is a  semisimple $K$-algebra.
\end{itemize}
\end{definition}

\begin{lemma}
Let $R$ be as above, $R' \subseteq R$ be a finite ring extension and $A$ be an $R$-algebra. Then $A$ is an $R$-order if and only if $A$ is an $R'$-order.
Moreover, if $K' := \mathsf{Quot}(R')$ then we have: $A_K \cong A_{K'}$.
\end{lemma}

\begin{proof} It is clear that $A$ is finitely generated and torsion free over $R$ if and only if it is finitely generated and torsion free over $R'$. Next, note that the ring extension
$R' \subseteq R$ induces a finite ring extension $K' \subseteq K$ of the corresponding total rings of fractions. Moreover,
Chinese remainder theorem implies that the multiplication  map $K' \otimes_{R'} R \lar K$ is an isomorphism. Therefore, for any finitely generated
$R$-module $M$, the natural map $K' \otimes_{R'} M \lar K \otimes_R M$ is an isomorphism of $K'$-modules. In particular, we get an isomorphism
of $K'$-algebras $A_{K'} \lar A_K$.
\end{proof}

\begin{definition}\label{D:orders} Let $A$ be a ring.
\begin{itemize}
\item $A$ is  a  \emph{one-dimensional order} (or just an \emph{order}) provided  its center $R = Z(A)$ is a reduced excellent  ring  of Krull dimension one, and $A$ is an $R$-order.
\item Let
$K:= \mathsf{Quot}(R)$. Then $A_K:= K \otimes_R A$ is called the \emph{rational envelope} of $A$.
\item A ring $\widetilde{A}$ is called an \emph{overorder} of $A$ if $A \subseteq \widetilde{A} \subseteq A_K$ and $\widetilde{A}$ is finitely generated as (a left) $A$-module.
\end{itemize}
\end{definition}

\smallskip
\noindent
Note that for any overorder $\widetilde{A}$ of $A$,  the map $K \otimes_R \widetilde{A} \lar A_K$ is automatically an isomorphism. Hence,
$A_K = \widetilde{A}_K$ and $\widetilde{A}$ is an order over $R$.

\begin{theorem}\label{T:HeredOrders} Let $H$ be a hereditary order (i.e.~$\mathsf{gl.dim}H =1$)  and  $R = Z(H)$ be the center of $H$. Then the following results are true.
\begin{itemize}
\item We have:  $R \cong R_1 \times \dots \times R_r$, where  $R_i$ is a   Dedekind domain for all $1\le i \le r$.
\item Let $K_i$ be the quotient field  of $R_i$. Then we have:  $H_K \cong \Upsilon_1 \times \dots \times \Upsilon_r$, where $\Upsilon_i$ is a  finite dimensional \emph{central simple} $K_i$-algebra. Moreover,  we have a decomposition $H = H_1 \times \dots \times H_r$, where each factor $H_i$ is a hereditary order whose rational envelope is $\Upsilon_i$.
\item If $\widetilde{H}$ is an overorder of $H$ then $\widetilde{H}$ is hereditary too.
\item If $H'$ is a minor of $H$ then $H'$ is a hereditary order too.
\item Assume that $R$ is semilocal. Let   $J$  be the radical of $H$ and
$\widehat{H} =
 \varprojlim\limits_{k} \bigl(H/J^k\bigr)$
be the radical completions of $H$. Then $\widehat{H}$ is a hereditary order.
\end{itemize}
\end{theorem}

\smallskip
\noindent
\emph{Proofs} of all these results can be for instance found in \cite{ReinerHO, ReinerMO}.

\begin{theorem}\label{T:BrauerGroupZero} Let $\kk$ be an algebraically closed field and $K$ be either $\kk\llbrace w\rrbrace$ or the function field
of an algebraic curve over $\kk$. Let $\Upsilon$ be a finite dimensional \emph{central} simple algebra over $K$. Then $\Upsilon \cong \Mat_{t}(K)$ for some 
$t \in \NN$. \end{theorem}
\begin{proof}
This is a restatement  of the fact that the Brauer group of the field $K$ is trivial; see for instance \cite[Proposition 6.2.3, Theorem 6.2.8 and Theorem 6.2.11]{GilleSzamuely}.
\end{proof}

\smallskip
\noindent
The following result is well-known; see \cite{ReinerMO}.

\begin{lemma}
Let $R$ be a discrete valuation ring, $\idm$ be its maximal ideal, $\kk:= R/\idm$  the corresponding residue field  and $K$ the  field of fractions of $R$. For any sequence of natural numbers
$\vec{p} = \bigl(p_1, \dots, p_r\bigr)$, consider the $R$-algebra
\begin{equation}\label{E:standardorder}
H(R, \vec{p}) :=
\left(
\begin{array}{ccc|ccc|c|ccc}
R & \dots & R & \idm & \dots & \idm & \dots & \idm & \dots &  \idm  \\
\vdots & \ddots & \vdots & \vdots &  & \vdots &  & \vdots & & \vdots \\
R & \dots & R & \idm & \dots & \idm & \dots & \idm & \dots &  \idm  \\
\hline
R & \dots & R & R & \dots & R & \dots & \idm & \dots &  \idm  \\
\vdots &  & \vdots & \vdots  & \ddots & \vdots &  & \vdots & & \vdots \\
R & \dots & R & R & \dots & R & \dots & \idm & \dots &  \idm  \\
\hline
\vdots &  & \vdots & \vdots & & \vdots & \ddots & \vdots &   & \vdots\\
\hline
R & \dots & R & R & \dots & R & \dots & R & \dots &  R  \\
\vdots &  & \vdots &  \vdots &  & \vdots &  & \vdots & \ddots & \vdots \\
R & \dots & R & R & \dots & R & \dots & R & \dots &  R  \\
\end{array}
\right) \subset \Mat_{p }(K),
\end{equation}
where the size
of the $i$-th diagonal block  is $(p_i \times p_i)$ for each $1 \le i \le r$ and $p := |\vec{p}| = p_1 + \dots + p_r$.
Then $H(R, \vec{p})$ is a hereditary $R$-order.
\end{lemma}

\smallskip
\noindent In what follows, $H(R, \vec{p})$ will be called  \emph{standard} hereditary $R$-order of type $\vec{p}$. For any $M \in H(R, \vec{p})$ and any
pair $1 \le i, j \le r$, we denote by $M^{(i, j)}$ the corresponding block of $M$, which is a matrix of size
$p_i \times p_j$. In particular,
$M^{(i, i)}(0) \in \Mat_{p_i}(\kk)$. In the simplest case when
$\vec{p} = \vec{p}_r := (\underbrace{1, \dots, 1}_{r \; \scriptsize{\rm times}})$, we shall use the notation
$H_r(R) := H\bigl(R, \vec{p}_r\bigr)$.

\begin{theorem}\label{T:Harada}
Let $\kk$ be an algebraically closed field, $R = \kk\llbracket w\rrbracket$ and $K = \kk \llbrace w\rrbrace$. Then the following results are true.
\begin{itemize}
\item Assume that $H$ is a hereditary $R$-order in  the central simple $K$-algebra $\Mat_{p }(K)$. Then there exists $S \in \Mat_{p}(K)$ such that
$H = \mathsf{Ad}_S\bigl(H(R, \vec{p})\bigr) := S \cdot H(R, \vec{p})\cdot S^{-1}$ for some  tuple $\vec{p} = (p_1, \dots, p_r)$ such that
$p = |\vec{p}|$. Moreover, such a tuple  $\vec{p}$ is  uniquely determined  up to a cyclic permutation.
\item Let $H$ be a hereditary $R$-order. Then we have:
$
H \cong H_1 \times \dots \times H_s,
$
where each $H_i$ is some  standard $R$-order for any $1 \le i \le s$.
\item For any vector $\vec{p}$, the orders $H(R, \vec{p})$ and $H_p(R)$ are centrally Morita equivalent.
\end{itemize}
\end{theorem}

\smallskip
\noindent
Proofs of these results can be found in \cite{ReinerMO}.

\section{Nodal orders}

\noindent
Nodal orders are  appropriate non-commutative generalizations  of the commutative nodal ring $D:= \kk\llbracket u, v\rrbracket/(uv)$.

\subsection{Definition and basic properties of nodal orders}
\begin{definition}\label{D:NodalOrders}
An order $A$ is called \emph{nodal} if its center is a semilocal excellent ring and there exists a \emph{hereditary} overorder $H \supseteq A$ such that the following conditions are satisfied.
\begin{itemize}
\item $J:=\rad(A) = \rad(H)$.
\item For any  finitely generated simple left $A$-module $U$ we have: $l_A(H \otimes_A U) \le 2$.
\end{itemize}
\end{definition}

\begin{remark}In what follows, hereditary orders will be  considered as special cases of nodal orders.
 Next, it is clear that an order  $A$ is nodal if and only if its radical completion $\widehat{A}$ is nodal. Moreover, it is not difficult to show that for a nodal order $A$,  the hereditary overorder $H$ from Definition \ref{D:NodalOrders} is in fact 
 \emph{uniquely determined} and admits the following intrinsic description:
 $$
H = \bigl\{x \in A_K \, \big| \, x J \subseteq J\bigr\} \cong \End_A(J),$$ where $J$ is viewed as a right $A$-module and $A_K$ is the rational envelope of $A$.  For a nodal order $A$, the order
$H$ will be called the \emph{hereditary cover} of $A$.
\end{remark}

\begin{remark} The classical commutative nodal ring $D = \kk\llbracket u, v\rrbracket/(uv)$ is a nodal order in the sense
 of Definition \ref{D:NodalOrders}. Indeed, we have an embedding  $D \subseteq \widetilde{D}:= \kk\llbracket u\rrbracket \times \kk\llbracket v\rrbracket$ and $\rad(D) = \rad({\widetilde{D}}) =
(u, v)$. Moreover, $\dim_\kk\bigl({\widetilde{D}} \otimes_D \kk\bigr) = 2$. Thus  $\widetilde{D}$ is the hereditary cover of $D$.

\begin{remark}
Let $A$ be a nodal order. It can be shown that its hereditary cover $H$ can also be described as 
 $
H = \bigl\{x \in A_K \, \big| \, J x \subseteq J\bigr\}$,
i.e.~$H$ can also be identified with $\End_A(J)^\circ$, where $J$ is viewed as a \emph{left} $A$-module; see Theorem \ref{T:DescriptionNodalOrders} and Theorem \ref{T:ClassificationNodalOrders}.
\end{remark}

\smallskip
\noindent
Nodal orders were  introduced by the second-named author in \cite{NodalFirst}. In that work it was shown that the category of finite length modules over a (non-hereditary)  nodal  order is representation tame and conversely  an order of nonwild representation type is automatically nodal. In our   previous joint work \cite{Nodal} we  proved that even the derived category of finite length representations of a nodal order 
has tame representation type. 
\end{remark}

\begin{theorem}\label{T:fundamentalsNodalOrders}
Let $A$ be a nodal order. Then the following statements are true.
\begin{itemize}
\item Any overorder of $A$ is nodal.
\item Any minor of $A$ is  nodal. In particular, if $A'$ is Morita equivalent to $A$ then $A'$ is a nodal order too.
\item Let $G$ be a finite group acting on $A$. If $|G|$ is invertible in $A$ then the skew group product $A*G$ is a nodal order.
\end{itemize}
\end{theorem}

\smallskip
\noindent
\emph{Proofs} of the above statements  can be found or easily deduced  from the results of \cite{NodalFirst}.

\subsection{Combinatorics of nodal orders} Let $\kk$ be an algebraically closed field. It turns out that nodal orders over the discrete valuation ring $\kk\llbracket x\rrbracket$ can be completely classified.

\begin{definition}\label{D:EquivRelation} Let $\Omega$ be a finite set and $\approx$ be a symmetric but not necessarily reflexive relation on $\Omega$ such that  for any $\omega \in \Omega$ there exists \emph{at most one} $\omega' \in \Omega$ (possibly, $\omega' = \omega$) such that $\omega \approx \omega'$ (note that $\approx$ is automatically transitive). We say that $\omega \in \Omega$ is
\begin{itemize}
\item \emph{simple} if $\omega \not\approx \omega'$ for all $\omega' \in \Omega$;
\item \emph{reflexive} if $\omega \approx \omega$;
\item \emph{tied} if there exists $\omega' \ne \omega \in \Omega$ such that $\omega \approx \omega'$.
\end{itemize}
It is clear that any element of $\Omega$ is either simple, or reflexive, or tied with respect to the relation $\approx$.

\smallskip
\noindent
Given $(\Omega, \approx)$ as above, we define two new sets $\Omega^\ddagger$  and $\widetilde\Omega^\ddagger$ by the following procedures.
\begin{itemize}
\item We get $\Omega^\ddagger$ from  $\Omega$ by replacing each reflexive element $\omega \in \Omega$ by a pair  of  new simple elements
$\omega_+$ and $\omega_-$. The tied elements of $\Omega^\ddagger$ are the same as for $\Omega$.
\item The set $\widetilde\Omega^\ddagger$ is obtained from   $\Omega^\ddagger$ by replacing each pair of tied elements $\bigl\{\omega', \omega''\bigr\}$ by a single element $\overline{\bigl\{\omega', \omega''\bigr\}}$.
\end{itemize}
A map $\Omega^\ddagger \stackrel{\mathsf{wt}^\ddagger}\lar \NN$ is called a \emph{weight function} provided $\mathsf{wt}^\ddagger(\omega') = \mathsf{wt}^\ddagger(\omega'')$ for all  $\omega' \approx \omega''$ in $\Omega^\ddagger$. It is clear that $\mathsf{wt}^\ddagger$ descends to a map
    $\widetilde\Omega^\ddagger \stackrel{\mathsf{wt}^\ddagger}\lar \NN$.
\smallskip
A  weight function $\Omega^\ddagger \stackrel{\mathsf{wt}^\ddagger}\lar \NN$ determines a map (also called weight function)  $\Omega \stackrel{\mathsf{wt}}\lar \NN$ given by the rule $\mathsf{wt}(\omega) := \mathsf{wt}^\ddagger(\omega_+) + \mathsf{wt}^\ddagger(\omega_-)$ for any reflexive point $\omega \in \Omega$.
Abusing the notation, we shall drop the symbol $\ddagger$ in the notation of $\mathsf{wt}^\ddagger$ and write $\mathsf{wt}$ for all weight functions introduced above.
\end{definition}

\begin{remark}
The simplest weight function $\Omega^\ddagger \stackrel{\mathsf{wt}_{\circ}}\lar \NN$ is given by the rule $\mathsf{wt}_{\circ}(\omega) = 1$ for all $\omega \in \Omega^\ddagger$.
\end{remark}

\smallskip
\noindent
Let $(\Omega, \approx)$ be as in Definition \ref{D:EquivRelation} and 
$\Omega \stackrel{\sigma}\lar \Omega$ be a permutation. Then we have a decomposition
\begin{equation}\label{E:Decomposition}
\Omega = \Omega_1 \sqcup \dots \sqcup  \Omega_t,
\end{equation}
where $\sigma(\Omega_i) = \Omega_i$ and the restricted permutation $\sigma\big|_{\Omega_i}$ is cyclic for any $1 \le i \le t$. In a similar way, we get
a decomposition $\Omega^\ddagger = \Omega^\ddagger_1 \sqcup \dots \sqcup  \Omega^\ddagger_t$.

\begin{definition}\label{D:DatumEnhanced} Let $\Omega$ be a finite set and $\Omega \stackrel{\sigma}\lar \Omega$ a permutation.  A \emph{marking} $\sm$ of $(\Omega, \sigma)$ is a choice of an element
$\omega_i \in \Omega_i$ for any $1\le i \le t$, where $\Omega_i$ are given by \eqref{E:Decomposition}.

\smallskip
\noindent
Note that a choice of marking $\sm$  makes each set  $\Omega_i$  totally ordered:
\begin{equation}
\omega_i < \sigma(\omega_i) < \dots < \sigma^{l_i-1}(\omega_i),
\end{equation}
where  $l_i:= |\Omega_i|$. Let $\approx$ be a relation on $\Omega$ as in Definition \ref{D:EquivRelation} and $\Omega^\ddagger \stackrel{\mathsf{wt}}\lar \NN$ be a  weight function. Then for  any $1 \le i \le t$, we get a vector
\begin{equation}
\vec{p}_i := \bigl(\mathsf{wt}(\omega_i), \mathsf{wt}(\sigma(\omega_i)),\dots, \mathsf{wt}(\sigma^{l_i-1}(\omega_i))\bigr) \in \NN^{l_i}.
\end{equation}
\end{definition}

\begin{definition}\label{D:DatumHeredOrder} Let $(\Omega, \approx, \sigma)$ be a datum as in Definition \ref{D:DatumEnhanced},
$\sm$ be a marking of $(\Omega, \sigma)$
and $\Omega^\ddagger \stackrel{\mathsf{wt}}\lar \NN$ be a  weight function. For any $1 \le i \le t$, let $H_i := H(R, \vec{p}_i)$ be the corresponding standard hereditary order \eqref{E:standardorder}.
Then  we put:
\begin{equation}\label{E:OrderHereditary}
H = H\bigl(R, (\Omega, \sigma, \approx, \sm, \wt)\bigr) := H_1 \times \dots \times H_t.
\end{equation}
It is clear that  $H$ is a hereditary order, whose rational envelope is the semisimple algebra
\begin{equation}\label{E:Hull}
\Lambda := \Mat_{s_1}(K) \times \dots \times \Mat_{s_t}(K),
\end{equation}
where $K$ is the fraction field of $R$ and $s_i = \big|\vec{p}_i\big|$ for $1 \le i \le t$.
\end{definition}

\begin{remark} According to the definition \eqref{E:standardorder} of a standard hereditary order, any  matrix belonging to   $H_i$ is endowed with a division  into vertical and horizontal stripes labelled by the elements of the set $\Omega_i$. Moreover, for any reflexive element $\omega \in \Omega_i$, the corresponding vertical and horizontal stripes have
further subdivisions labelled by the elements $\omega_\pm \in \Omega^\ddagger_i$.
\end{remark}

\begin{definition}\label{D:nodalrings} Let $R$ be a discrete valuation ring and $\bigl(\Omega, \sigma, \approx, \sm, \mathsf{wt}\bigr)$ be a datum as above. Then  we have a ring $A = A\bigl(R, (\Omega, \sigma, \approx, \sm, \mathsf{wt})\bigr) \subseteq H$ defined as follows:
\begin{equation*}
 A:= \left\{(X_1, \dots X_t) \in H \left|  \begin{array}{l}
 X_{i'}^{(\omega',\, \omega')}(0) = X_{i''}^{(\omega'',\, \omega'')}(0) \mbox{\rm \, for all\, }
 \begin{array}{c}
1 \le i', i'' \le t \\
\omega' \in \Omega_{i'},\, \omega'' \in \Omega_{i''} \\
\omega' \approx \omega'', \; \omega' \ne \omega''
\end{array}
\\
\hline
\\
X_{i}^{(\omega_+,\, \omega_-)}(0) = 0 \; \mbox{\rm and}\;  X_{i}^{(\omega_-,\, \omega_+)}(0) = 0  \; \mbox{\; \rm for all\;}\;
\begin{array}{l}
1 \le i  \le t \\
\omega \in \Omega_i \\
\omega \approx \omega
\end{array}
\end{array}
\right.
\right\}.
\end{equation*}
\end{definition}

\smallskip
\noindent
The proof of the following results is straightforward.
\begin{theorem}\label{T:DescriptionNodalOrders}
Let $R$ be a discrete valuation ring, $\bigl(\Omega, \sigma, \approx, \sm, \mathsf{wt}\bigr)$ a  datum as in Definition
 \ref{D:DatumEnhanced}
 and $A = A\bigl(R, (\Omega, \sigma, \approx, \sm, \mathsf{wt})\bigr)$ the corresponding ring from Definition \ref{D:nodalrings}.  Then the following statements hold.
\begin{itemize}
\item The ring $A$ is connected if and only if there exists a surjection  
$$\{1, \dots, m\} \stackrel{\tau}\lar \{1, \dots, t\}$$ for some $m$ as well as  elements $\upsilon_i \in \Omega_{\tau(i)}$ for  $1 \le i \le m-1$ such that each $\upsilon_i \approx \upsilon'_i$ for some 
    $\upsilon'_i \in \Omega_{\tau(i+1)}$. If $A$ is connected and  $R = \kk\llbracket w\rrbracket$ then the center of $A$ is isomorphic to
    $\kk\llbracket w_1, \dots, w_t\rrbracket/(w_i w_j, 1 \le i \ne j \le t)$.
\item The ring $A$ is an order whose rational envelope is the semisimple algebra $\Lambda$ given by \eqref{E:Hull}. The order $H$ defined
by \eqref{E:OrderHereditary} is an overorder of $A$.
\item Let $\widetilde\sm$ be any other marking of $(\Omega,  \sigma)$ and $\widetilde{A} = A\bigl(R, (\Omega, \sigma, \approx, \widetilde\sm, \mathsf{wt})\bigr)$ be the corresponding order. Then there exists $S \in \Lambda$ such that $\widetilde{A} = \Ad_S(A)$, i.e.~the orders
$\widetilde{A}$ and $A$ are conjugate in $\Lambda$.  It means that the order   $A\bigl(R, (\Omega, \sigma, \approx, \sm, \mathsf{wt})\bigr)$ does not depend (up to a conjugation) on the choice of marking of $(\Omega, \sigma)$, so in what follows it will be denoted by $A\bigl(R, (\Omega, \sigma, \approx, \mathsf{wt})\bigr)$.
\item $A\bigl(R, (\Omega, \sigma, \approx, \mathsf{wt})\bigr)$ and $A\bigl(R, (\Omega, \sigma, \approx, \mathsf{wt}_\circ)\bigr)$ are centrally Morita equivalent.
    \item Let
\begin{equation}\label{E:NodalRadical}
    J:= \left\{(X_1, \dots X_t) \in A\; \left| \; X_i^{(\omega,\, \omega)}(0) = 0 \; \mbox{\rm for all}\; \begin{array}{l}
1 \le i \le t \\
\omega \in \Omega_i
\end{array}
\right.
\right\}.
\end{equation}
Then we have: $J = \rad(A) = \rad(H)$. Moreover, the natural map
$$H \lar  \End_A(J), \quad X \mapsto (Y \mapsto XY)
$$
is an isomorphism, where $J$ is viewed as a right $A$-module.
\item  For any $\omega \in \Omega$, let $\bar{H}_\omega := \Mat_{m_\omega}(\kk)$, where $m_\omega = \mathsf{wt}(\omega).$
Similarly, for any
$\gamma \in \widetilde\Omega^\ddagger$, let
$
\bar{A}_\gamma := \Mat_{m_\gamma}(\kk)$, where $m_\gamma = \mathsf{wt}(\gamma)$.
Then we have a commutative diagram:
\begin{equation}\label{E:gentle}
\begin{array}{c}
\xymatrix{
A/J \ar[rr]^-{\cong} \ar@{_{(}->}[d] & & \prod\limits_{\gamma \in \widetilde\Omega^\ddagger} \bar{A}_{\gamma}
\ar@{^{(}->}[d]^-{\imath}\\
H/J \ar[rr]^-{\cong}  & & \prod\limits_{\omega \in \Omega} \bar{H}_{\omega}
}
\end{array}
\end{equation}
where  the components of the embedding $\imath$ are defined as follows.
\begin{itemize}
\item Let  $\omega \in \Omega$ be a simple element and $\gamma$ be the corresponding element of $\widetilde\Omega^\ddagger$. Then
the corresponding component $\bar{A}_\gamma \lar \bar{H}_\omega$ of  $\imath$ is the identity map.
\item Let $\omega \in \Omega$ be reflexive. Then the corresponding component of the map $\imath$ is given by the rule:
$$
\bar{A}_{\omega_+} \times \bar{A}_{\omega_-} \lar \bar{H}_\omega,  \quad
(X, Y) \mapsto
\left(
\begin{array}{cc}
X & 0 \\
0 & Y
\end{array}
\right).
$$
\item Finally, let  $\omega', \omega'' \in \Omega$ be a pair of tied elements and $\gamma := \overline{\left\{\omega', \omega''\right\}}$ be the corresponding element of $\widetilde\Omega^\ddagger$. Then the corresponding component of $\imath$ is the diagonal embedding
    $$
    \bar{A}_\gamma \lar \bar{H}_{\omega'} \times \bar{H}_{\omega''}, \quad X \mapsto  (X, X).
    $$
\end{itemize}
\item The order  $A\bigl(R, (\Omega, \sigma, \approx, \mathsf{wt})\bigr)$ is nodal and $H\bigl(R, (\Omega, \sigma, \approx, \mathsf{wt})\bigr)$
is its hereditary cover.
\item Let $(\Omega, \sigma, \approx, \mathsf{wt}), (\Omega', \sigma', \approx', \mathsf{wt}')$
be two data as in Definition \ref{D:DatumEnhanced}. If $R$ is complete then we have:
$$
A\bigl(R, (\Omega, \sigma, \approx, \mathsf{wt})\bigr) \cong A\bigl(R, (\Omega', \sigma', \approx', \mathsf{wt}')\bigr)
$$
if and only if there exists a bijection $\Omega \stackrel{\varphi}\lar \Omega'$ such that
the diagram
$$
\xymatrix{
\Omega \ar[rr]^-{\sigma} \ar[dd]_-{\varphi} &  & \Omega \ar[rd]^-{\wt}  \ar[dd]^-{\varphi} &  \\
                      &    &                      & \NN \\
\Omega' \ar[rr]^-{\sigma'}  & &  \Omega \ar[ru]_-{\wt'}         &  \\
}
$$
is commutative and $\omega_1 \approx \omega_2$ in  $\Omega$  if and only if  $\varphi(\omega_1) \approx \varphi(\omega_2)$ in  $\Omega'$.
\end{itemize}
\end{theorem}

\begin{remark}
For the  weight function $\Omega^\ddagger \stackrel{\mathsf{wt}_{\circ}}\lar \NN$,   given by the rule $\mathsf{wt}_{\circ}(\omega) = 1$ for all $\omega \in \Omega^\ddagger$, we shall write $A\bigl(R, (\Omega, \sigma, \approx)\bigr):= A(R, \Omega, \sigma, \approx, \mathsf{wt}_{\circ})$ or simply  
$A(\Omega, \sigma, \approx)$ in the case when it is clear which ring $R$ is meant.
\end{remark}

\smallskip
\noindent
In the examples below we take  $R = \kk\llbracket w\rrbracket$.

\begin{example}\label{Ex:Dyad}
Let $\Omega = \{1, 2\}$, $\sigma$ is identity
 and $1 \approx 2$. Then we have:
$$
A(\Omega, \sigma, \approx) \cong D:= \kk\llbracket x, y\rrbracket/(xy).
$$
As already mentioned, the hereditary cover of $D$  is $\widetilde{D}:= \kk\llbracket x\rrbracket \times \kk\llbracket y\rrbracket$.
\end{example}

\begin{example}\label{Ex:ZhelobenkoOrder}
Let $\Omega = \{1, 2, 3\}$, $\sigma = \left(
\begin{smallmatrix}
1 & 2 & 3\\
1 & 3 & 2
\end{smallmatrix}
\right)$ and $1 \approx 2$. Then we have:
$$
H = H(\Omega, \sigma, \approx) = \kk\llbracket x\rrbracket \times
\left(
\begin{array}{cc}
\kk\llbracket y\rrbracket & (y) \\
\kk\llbracket y\rrbracket & \kk\llbracket y\rrbracket
\end{array}
\right)
$$
and
\begin{equation}\label{E:ZhelobenkoOrder}
A= A(\Omega, \sigma, \approx) = \left\{
(X, Y) \in H \; \left|\;  X(0) = Y^{(11)}(0)\right.\right\}.
\end{equation}
 Note that $A \cong \End_D(\kk\llbracket x\rrbracket \oplus D)$, where $D = \kk\llbracket x,y\rrbracket/(xy)$. Alternatively,
one can identify   $A$ with  the arrow completion of the path algebra of the following quiver with relations:
\begin{equation}\label{E:ZhelobQuiver}
\xymatrix
{
\bul  \ar@(ul, dl)_{a} \ar@/^/[rr]^{b}  & & \bul \ar@/^/[ll]^{c}
} \qquad ba = 0, \,ac = 0.
\end{equation}
The order given by   \eqref{E:ZhelobenkoOrder} will be called \emph{Zhelobenko order},
since it appeared for the first time in a work of Zhelobenko \cite{Zhelobenko} dedicated to  the study of admissible finite length representations of the Lie group $\SL_2(\CC)$.
\end{example}

\begin{example}\label{Ex:AuslanderOrder}
Let $\Omega = \{1, 2, 3, 4\}$,  $\sigma =
\left(
\begin{smallmatrix}
1 & 2 & 3 & 4\\
2 & 1 & 4 & 3\\
\end{smallmatrix}
\right)$ and $2 \approx 3$. Then we have:
\begin{equation}\label{E:AuslanderOrderNodal}
A = A(\Omega, \sigma, \approx) \cong
\left(
\begin{array}{cc}
D & \widetilde{D} \\
I & \widetilde{D}
\end{array}
\right) \cong \End_D(D \oplus I),
\end{equation}
where $D = \kk\llbracket x,y\rrbracket/(xy)$, $I = (x, y)$ and $\widetilde{D} = \kk\llbracket x\rrbracket \times \kk\llbracket y\rrbracket$. The order $A$ is isomorphic to the arrow completion of the following quiver with relations:
\begin{equation}\label{E:AuslanderQuiver}
\xymatrix
{
- \ar@/^/[rr]^{u_{-}}  & &  \star \ar@/^/[ll]^{v_{-}}
 \ar@/_/[rr]_{v_{+}}
 & &
\ar@/_/[ll]_{u_{+}} +}  \qquad v_{\pm} u_{\mp} = 0.
\end{equation}
Following the terminology of our previous work \cite{bd}, $A$ will be called the Auslander order of $D$, or just \emph{Auslander order}.
\end{example}

\begin{example}\label{Ex:NodalOrbifolded}
Let $\Omega = \{1\}$  with $1 \approx 1$ (of course, $\sigma = \mathrm{id}$ in this case). Then we have:
\begin{equation}\label{E:NodalOrbifolded}
A = A(\Omega, \sigma, \approx) \cong
\left(
\begin{array}{cc}
\kk\llbracket w\rrbracket &  (w)\\
(w)  & \kk\llbracket w\rrbracket
\end{array}
\right).
\end{equation}
The hereditary cover $H$ of $A$ ist just the matrix algebra $\Mat_2\bigl(\kk\llbracket w\rrbracket\bigr)$.
\end{example}

\begin{example}
Let $\Omega = \{1, 2\}$, $\sigma =
\left(
\begin{smallmatrix}
1 & 2 \\
2 & 1 \\
\end{smallmatrix}
\right)$ and $2 \approx 2$. Then we have:
\begin{equation}\label{E:GelfandOrder}
A = A(\Omega, \sigma, \approx) \cong
\left(
\begin{array}{ccc}
\kk\llbracket w\rrbracket &  (w) & (w) \\
\kk\llbracket w\rrbracket &  \kk\llbracket w\rrbracket & (w) \\
\kk\llbracket w\rrbracket &  (w) & \kk\llbracket w\rrbracket \\
\end{array}
\right).
\end{equation}
Note that  $A$ is isomorphic to the arrow completion of the following quiver with relations:
\begin{equation}\label{E:GelfandQuiver}
\xymatrix
{
- \ar@/^/[rr]^{a_{-}}  & &  \star \ar@/^/[ll]^{b_{-}}
 \ar@/_/[rr]_{b_{+}}
 & &
\ar@/_/[ll]_{a_{+}} +}  \qquad a_+ b_+ = a_- b_-.
\end{equation}
The order \eqref{E:GelfandOrder} appeared for the first time in the 1970 ICM talk of I.~Gelfand \cite{Gelfand} in the context of  the study of admissible finite length representations of the Lie group $\SL_2(\mathbb{R})$. In what follows, it will be called \emph{Gelfand order}. The  hereditary cover of $A$ is
$$
H = H(\Omega, \sigma, \approx) \cong
\left(
\begin{array}{ccc}
\kk\llbracket w\rrbracket &  (w) & (w) \\
\kk\llbracket w\rrbracket &  \kk\llbracket w\rrbracket & \kk\llbracket w\rrbracket \\
\kk\llbracket w\rrbracket &  \kk\llbracket w\rrbracket & \kk\llbracket w\rrbracket \\
\end{array}
\right).
$$
\end{example}

\begin{theorem}\label{T:ClassificationNodalOrders} Let $\kk$ be an algebraically closed field, $R = \kk\llbracket w\rrbracket, K = \kk\llbrace w\rrbrace$ and
\begin{equation*}
\Lambda = \Mat_{s_1}(K) \times \dots \times \Mat_{s_t}(K)
\end{equation*}
for some $s_1, \dots s_t \in \NN$. If $A$ is a nodal order whose rational envelope  is $\Lambda$ then there exists  a datum $(\Omega, \sigma, \approx, \wt)$  and $S \in \Lambda$ such that $A = \Ad_S\bigl(A(\Omega, \sigma, \approx, \wt)\bigr)$, where $A(\Omega, \sigma, \approx, \wt)$ is the nodal order from Definition \ref{D:nodalrings}.
\end{theorem}

\smallskip
\noindent
\begin{proof} Let $A \subset \Lambda$ be a nodal order, whose
rational envelope is $\Lambda$ and  $A^\diamond$ be a basic order
Morita equivalent to $A$. Then $A^\diamond$ is also nodal, see Theorem \ref{T:fundamentalsNodalOrders}. Let $H^\diamond$ be the hereditary cover of $A_\diamond$.
Then 
$H^\diamond \cong H(R, \vec{p}_1) \times \dots \times H(R, \vec{p}_r),
$
where each component $H(R, \vec{p}_i)$ is a standard hereditary order given by \eqref{E:standardorder}. Let $J^\diamond = \rad(A^\diamond) = \rad(H^\diamond)$ be the common
radical of $A^\diamond$ and $H^\diamond$, $\bar{A}^\diamond = A^\diamond/J^\diamond$ and 
$\bar{H}^\diamond = H^\diamond/J^\diamond$. Since $A^\diamond$ is basic and $\bar{A}^\diamond$ is semisimple, $\bar{A}^\diamond_\gamma$ is isomorphic to a product of several copies of the field $\kk$. By the definition of nodal orders, the embedding  of semisimple algebras $\bar{A}^\diamond \stackrel{\imath}\hookrightarrow \bar{H}^\diamond$ has the following property: for any simple $\bar{A}^\diamond$-module $U$ we have: 
$l_{\bar{A}^\diamond}\bigl(\bar{H}^\diamond \otimes_{\bar{A}^\diamond} U\bigr) \le 2$. From this property one can easily deduce that
\begin{itemize}
\item Each simple component of $\bar{H}^\diamond$ is either $\kk$ or $\Mat_2(\kk)$. In other words, for any $1 \le i \le r$, each entry of the vector $\vec{p}_i$ is either
$1$ or $2$.
\item The embedding $\bar{A}^\diamond \stackrel{\imath}\hookrightarrow \bar{H}^\diamond$ splits into the product of the following components:
$$
\kk \stackrel{\mathsf{id}}\lar \kk, \quad \kk \stackrel{\mathsf{diag}}\lar \kk \times \kk \quad \mbox{\rm or} \quad (\kk \times \kk) \stackrel{\mathsf{diag}}\lar \Mat_2(\kk).
$$
\end{itemize}
Let $1 \le i \le r$ and $\vec{p}_i = \bigl(p_{i,1}, \dots,p_{i,k_i}\bigr)$. Then we put:
$$
\Omega_i := \bigl\{(i, 1), \dots, (i, k_i)\bigr\}, \quad
\Omega:= \Omega_1 \sqcup \dots \sqcup \Omega_r \quad \mbox{\rm and} \quad 
\sigma(i, j):= (i, j+1 \; \mathsf{mod}\; k_i).
$$ 
Note that the set $\Omega$ parameterizes simple components of the algebra $\bar{H}^\diamond$. If $p_{i, j} = 2$ for $(i, j) \in \Omega_i$, we put: $(i, j) \approx
(i, j)$. If $(i, j), (i', j') \in \Omega$ are such that the corresponding component
of the embedding $\imath$ is $\kk \stackrel{\mathsf{diag}}\lar \kk \times \kk$, we put:
$(i, j) \approx (i', j')$.

\smallskip
\noindent
Note that the following diagram of $\kk$-algebras
\begin{equation*}
\begin{array}{c}
\xymatrix{
A^\diamond \ar[rr] \ar@{_{(}->}[d] & & \bar{A}^\diamond
\ar@{^{(}->}[d]^-{\imath}\\
H^\diamond \ar@{->>}[rr]^-{\pi}  & & \bar{H}^\diamond
}
\end{array}
\end{equation*}
is a pull-back diagram, i.e.~$A^\diamond \cong 
\pi^{-1}\bigl(\mathrm{Im}(\imath)\bigr)$. From this it is not difficult to deduce that
$$
A^\diamond = A(\Omega, \sigma, \approx, \wt_\circ),
$$
where $\wt_\circ$ is the trivial weight function.

\smallskip
\noindent
Since the order $A$ is Morita equivalent to $A^\diamond$, there exists
a projective left $A^\diamond$-module $P$ such that $A \cong 
\bigl(\End_{A^\diamond}(P)\bigr)^\circ$. Recall that the isomorphism classes of indecomposable $A^\diamond$-modules are  parameterized by the elements of the set $\widetilde{\Omega}^\ddagger$.
Let
$
P \cong \bigoplus\limits_{\gamma \in \widetilde{\Omega}^\ddagger} P_\gamma^{\oplus m_\gamma}
$
be a  decomposition of $P$ into a direct sum of indecomposable modules. Then we get
a weight function
$
\widetilde{\Omega}^\ddagger \stackrel{\mathsf{wt}}\lar \NN, \;  \gamma \mapsto m_\gamma.
$
It is easy to see that $A \cong A(\Omega, \sigma, \approx, \wt)$. 
\end{proof}

\subsection{Skew group products of $\kk\llbracket u, v\rrbracket/(uv)$ with a finite group}
Let $\kk$ be an algebraically closed  field of characteristic zero, $\zeta \in \kk$ be a primitive $n$-th root of $1$,   $G:= \left\langle \rho \,\big|\, \rho^n = e\right\rangle$ be a cyclic group of order $n \in \NN_{\ge 2}$ and $\kk[G]$ be the corresponding group algebra. The following results is well-known.

\begin{lemma}
For $ 1 \le k \le n$, let   $\xi_k := \zeta^k$ 
and 
$
\varepsilon_k := \frac{1}{n} \sum\limits_{j=0}^{n-1} \xi_k^j \rho^j \in \kk[G].
$
Then we have:
$$
\left\{
\begin{array}{ccc}
1 & = & \varepsilon_1 + \dots + \varepsilon_n \\
\varepsilon_k \cdot \varepsilon_l & = & \delta_{kl} \varepsilon_k
\end{array}
\right.
$$
In other words, $\left\{\varepsilon_1,\dots,\varepsilon_n\right\}$ is a complete set
of primitive idempotents of $\kk[G]$.
\end{lemma}

\begin{proposition}
Consider the action of the cyclic group $G$ on the polynomial algebra $\kk[u]$ given by the rule: $\rho \circ u := \zeta u$. Then the skew product $\kk[u] \ast G$ is isomorphic to the path algebra of the cyclic quiver $\vec{C}_n$
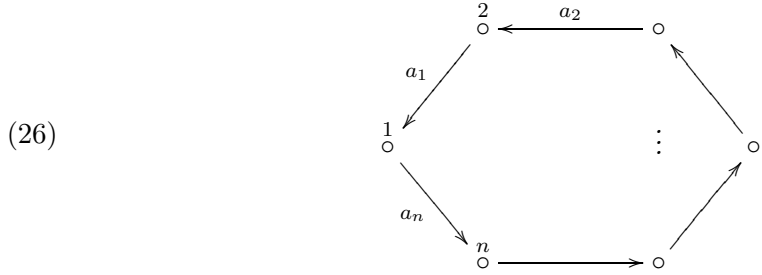
\begin{equation}\label{E:cyclicquiver}
\begin{array}{c}
\xymatrix{ & \stackrel{2}{\circ} \ar[ld]_-{a_1} & &\ar[ll]_-{a_2} \circ & \\
\stackrel{1}\circ  \ar[rd]_-{a_n} &         &   & \vdots  & \ar[lu] \circ\\
& \stackrel{n}\circ  \ar[rr] & & \circ \ar[ru] & \\
}
\end{array}
\end{equation}
\end{proposition}
\begin{proof}
An isomorphism $\kk[u]\ast G \lar \kk\bigl[\vec{C}_n\bigr]$ is given by the  rule:
$$
\left\{
\begin{array}{ccc}
\varepsilon_k & \mapsto & e_k \\
\varepsilon_k u \varepsilon_{k+1}  & \mapsto & a_k,
\end{array}
\right.
$$
where $e_k \in \kk\bigl[\vec{C}_n\bigr]$ is the trivial path  corresponding to the vertex $k$. 
\end{proof}

\begin{corollary}
Let $R := \kk\llbracket u\rrbracket$. Then the skew group product $R \ast G$
is isomorphic to the arrow completion of the path algebra of \eqref{E:cyclicquiver}.
Note that the latter algebra can in its turn be identified with the algebra of matrices
\begin{equation}\label{E:standardtransposed}
T_n(R):= \left(
\begin{array}{cccc}
R    & R & \dots & R \\
\idm & R & \dots & R \\
\vdots & \vdots& \ddots & \vdots \\
\idm &\idm & \dots & R
\end{array}
\right),
\end{equation}
where the primitive idempotent $e_k$ corresponding to the vertex $k$ of $\vec{C}_n$ is sent to the $k$-th diagonal matrix unit of $T_n(R)$.
\end{corollary}

\begin{remark} Let us notice that, strictly speaking, $R\ast G$ depends on the choice of an $n$-th primitive root of unity $\zeta$.
 On the other hand, the
primitive idempotent
$$
\varepsilon = \varepsilon_n := \frac{1}{n}(e + \rho + \dots + \rho^{n-1}) \in 
\kk[G] \cong R\ast G/\rad(R\ast G)
$$
does not depend on the choice of $\zeta$. Therefore,
identifying the skew group product $R \ast G$ with the completed path algebra of a cyclic quiver, we shall always choose a labeling of vertices of $\vec{C}_n$ such that
the idempotent $\varepsilon$ is identified with the trivial path corresponding to the vertex labeled by $n$.

\smallskip
\noindent
Note also that the orders $T_n(R)$ and $H_n(R)$ are isomorphic.
\end{remark}

\smallskip
\noindent
Let $0 < c < n$ be such that $\mathsf{gcd}(n, c) = 1$. Then we have a permutation
$$
\bigl\{\bar{1}, \dots, \bar{n}\bigr\}
\stackrel{\tau_c}\lar 
\bigl\{\bar{1}, \dots, \bar{n}\bigr\} \quad \bar{k} \mapsto \overline{c\cdot k},
$$
where $\bar{k}$ denotes the remainder of $k$ modulo $n$.

\begin{proposition}\label{P:cyclicAction}
For any $0 < c < n$ such that $\mathsf{gcd}(n, c) = 1$, consider the action
of the cyclic group $G = \left\langle \rho \,\big|\, \rho^n = e\right\rangle$
on the nodal algebra $D = \kk\llbracket u, v\rrbracket/(uv)$, given by the rule
\begin{equation}\label{E:cyclicgroupaction}
\left\{
\begin{array}{ccc}
\rho \circ u & = & \zeta \; u \\
\rho \circ v & = & \zeta^{c} v.
\end{array}
\right.
\end{equation}
Then the nodal order $A = A_{(n, d)} := D \ast G$ has the following description:
\begin{equation}\label{E:stackynode1}
 A_{(n, c)} \cong  \Bigl\{(U, V) \in T_n(R) \times T_n(R) \, \big| \, U^{\tau_c(k)\tau_c(k)}(0) = 
V^{kk}(0) \; \mbox{\rm for} \; 1 \le k \le n\Bigr\}.
\end{equation}
\end{proposition}
\begin{proof}
Let $\widetilde{D} = \kk\llbracket u\rrbracket \times \kk\llbracket v\rrbracket$ be
the hereditary cover of $D$.  Then
$$
H:= \widetilde{D} \ast G \cong \bigl(\kk\llbracket u\rrbracket \ast G\bigr) \times
\bigl(\kk\llbracket v\rrbracket \ast G\bigr) \cong T_n(R) \times T_n(R)
$$
is the hereditary cover of $A$. For any $1 \le k \le n$, let 
\[
 \widetilde{\varepsilon}_k := \varepsilon_{\tau_ck}= \dfrac{1}{n}\bigl(1 + \zeta^{ck} \rho + \zeta^{2ck} \rho^2 + \dots +
\zeta^{(n-1)ck} \rho^{n-1}\bigr) \in\kk[G].
\]
 We consider the elements $\varepsilon_k$ and $\widetilde\varepsilon_k$ as elements, respectively, of 
 $\kk\llbracket u\rrbracket\ast G$ and $\kk\llbracket v\rrbracket\ast G$. It is convenient, taking into account that  the actions of $G$ on $\kk\llbracket u\rrbracket$ and on
$\kk\llbracket v\rrbracket$ are actually different.

\smallskip
\noindent
We have the following commutative diagram of algebras and algebra homomorphisms:
$$
\xymatrix{
 A_{(n, c)} \ar@{_{(}->}[d]  \ar@{->>}[rr] & &  A_{(n, c)}/\rad( A_{(n, c)}) \ar@{_{(}->}[d] \ar[rr]^-\cong & & \kk[G]  \ar@{^{(}->}[d]^-{\mathsf{diag}} \\
H \ar@{->>}[rr] & & H/\rad(H) \ar[rr]^-\cong & & \kk[G] \times \kk[G]
}
$$
Viewing $\varepsilon_k$ and  $\widetilde{\varepsilon}_k$ as elements of
$\bigl(\kk\llbracket u\rrbracket \ast G\bigr)/\bigl((u) \ast G\bigr) \cong \kk[G] \cong \bigl(\kk\llbracket v\rrbracket \ast G\bigr)/\bigl((v) \ast G\bigr)$, we get: 
$\mathsf{diag}\bigl(\widetilde{\varepsilon}_k\bigr) = \bigl(\widetilde{\varepsilon}_k,
\widetilde{\varepsilon}_k\bigr) = \bigl(\varepsilon_{\sigma_c(k)}, \widetilde{\varepsilon}_{k}\bigr).$ Taking into account the rules for the isomorphisms
$$
\kk\llbracket u\rrbracket \ast G \stackrel{\cong}\lar T_n(R) \stackrel{\cong}\longleftarrow \kk\llbracket v\rrbracket \ast G,
$$
we get the description \eqref{E:stackynode1} of the nodal order $ A_{(n, c)}$.
\end{proof}

\begin{remark}\label{R:nodalcyclic} Note that the center of the order
$A_{(n, c)}$ is equal to $\kk\llbracket a, b\rrbracket/(ab)$, where $a = u^n$ and
$b = v^n$. Observe that  $A_{(n, c)}$ has precisely $n$ pairwise not-isomorphic finitely generated simple modules. It is not difficult top see that two such orders $A_{(n, c)}$
and $A_{(n', c')}$ are Morita equivalent if and only if 
$n = n'$ and $c = c'$ or $d = c'$, where $cd \equiv 1 \, \mathsf{mod}\, n$.
If $c \ne d$ then any Morita equivalence  between $A_{(n, c)}$ and $A_{(n, d)}$
permutes the irreducible branches of their common  center $\kk\llbracket a, b\rrbracket/(ab)$.

\smallskip
\noindent
In the terms of Theorem \ref{T:DescriptionNodalOrders}, the order $ A_{(n, c)}$  has the following description.
\begin{itemize}
\item Let $\Omega = \bigl\{1, \dots, n, \tilde{1}, \dots, \tilde{n}\bigr\}$.
\item The relation $\approx$ is given by the rule: $\tilde k \approx \tau_c(k)$ for $1 \le k \le n$.
\item The permutation $\Omega \stackrel{\sigma}\lar \Omega$ is given by the formula
$$
\sigma = \left(
\begin{array}{cccc|cccc}
1 & 2 & \dots & n & \tilde{1} & \tilde{2}& \dots & \tilde{n}\\
n & 1 & \dots & n-1 & \tilde{n} & \tilde{1} & \dots & \widetilde{n-1}
\end{array}
\right).
$$
\end{itemize}
Then we have: $ A_{(n, c)} \cong A(\Omega, \sigma, \approx, \wt_\circ)$. 
For any other weight function $\wt$, the orders $A_{(n, c)}$ and $A_{(n, c)}(\wt):= A(\Omega, \sigma, \approx, \wt)$ are centrally Morita equivalent.

\smallskip
\noindent
It is not difficult
to derive the quiver description of the order $A_{(n, c)}$. Of the major interest is the case
$c = n-1$. Then $ A_{(n, c)}$ is isomorphic to the arrow completion of the path algebra of the following quiver
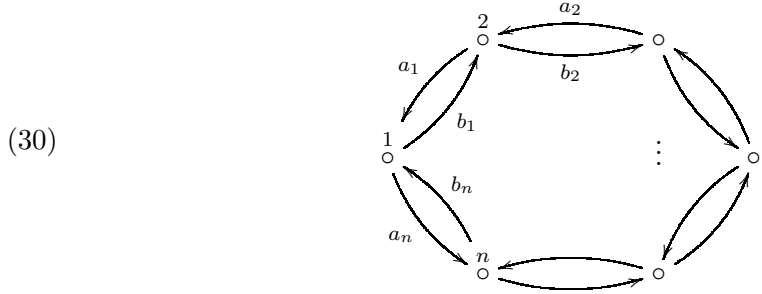
\begin{equation}\label{E:nodalquiver}
\begin{array}{c}
\xymatrix{ & \stackrel{2}{\circ} \ar@/_/[ld]_-{a_1} \ar@/_/[rr]_-{b_2}& &\ar@/_/[ll]_-{a_2} \circ \ar@/_/[rd] & \\
\stackrel{1}\circ  \ar@/_/[rd]_-{a_n} \ar@/_/[ru]_-{b_1}&         &   & \vdots  & \ar@/_/[lu] \circ \ar@/_/[ld]\\
& \stackrel{n}\circ  \ar@/_/[rr] \ar@/_/[lu]_-{b_n}& & \circ \ar@/_/[ru] \ar@/_/[ll]& \\
}
\end{array}
\end{equation}
modulo the relations $a_k b_k = 0 = b_k a_k$ for all $1 \le k \le n$.
\end{remark}

\begin{lemma}\label{L:BasicNodalSkew}
Consider the action
of the cyclic group $G = \left\langle \tau \,\big|\, \tau^2 = e\right\rangle$
on the nodal algebra $D = \kk\llbracket u, v\rrbracket/(uv)$, given by the rule
$\tau(u) = v$.
Then the nodal order $A := D \ast G$ has the following description:
\begin{equation}\label{E:stackynode2}
A \cong  
\left(
\begin{array}{cc}
\kk\llbracket w\rrbracket & (w) \\
(w)  & \kk\llbracket w\rrbracket
\end{array}
\right).
\end{equation}
\end{lemma}

\begin{proof}
Consider the elements $e_\pm := \dfrac{1 \pm \tau}{2} \in A$. Then the following statements are true.
\begin{itemize}
\item $e_\pm^2 = e_\pm$, $e_\pm e_\mp = 0$ and $1 = e_+ + e_-$. Moreover, $\tau \cdot e_\pm = \pm e_\pm$.
\item For any $s, t \in \left\{+, -\right\}$ we have: $e_s A e_t = e_s D e_t$.
\end{itemize}
Therefore we have the Peirce  decomposition:
\begin{equation}\label{E:Peirce}
A \cong 
\left(
\begin{array}{cc}
e_+ A e_+ & e_+ A e_- \\
e_- A e_+ & e_- A e_-
\end{array}
\right)
= 
\left(
\begin{array}{cc}
e_+ D e_+ & e_+ D e_- \\
e_- D e_+ & e_- D e_-
\end{array}
\right).
\end{equation}
For any $m \in \NN_0$, set
$
w^{(m)}_\pm := \dfrac{1}{2^m}\bigl(u^m \pm v^m\bigr)
$
(within this notation, $w^{(0)}_- = 0$).
It is easy to see that
$
w^{(m_1)}_{s_1} \cdot w^{(m_2)}_{s_2} = w^{(m_1+m_2)}_{s_1 \cdot s_2}
$
for any $m_1, m_2 \in \NN_0$ and $s_1, s_2 \in \bigl\{+, -\bigr\}$. Moreover,
one can check that
$$
e_\pm D e_\pm  =   
\left\langle w_+^{(m)} e_\pm \,\big|\, m \in \NN_0\right\rangle_\kk
\quad \mbox{\rm and}\quad e_\pm D e_\mp  =   \left\langle w_-^{(m)}  e_\mp\,\big|\, m \in \NN \; \right\rangle_\kk.
$$
One can check that the linear map
$$
A \cong \left(
\begin{array}{cc}
e_+ D e_+ & e_+ D e_- \\
e_- D e_+ & e_- D e_-
\end{array}
\right)
\lar \left(
\begin{array}{cc}
\kk\llbracket w\rrbracket & (w) \\
(w)  & \kk\llbracket w\rrbracket
\end{array}
\right)
$$
given by the following rules:
$$
\left\{
\begin{array}{ccc}
w^{(m)}_+ e_+ & \mapsto &w^m e_{11} \\
w^{(m)}_+ e_- & \mapsto &w^m e_{22}
\end{array}
\right.
\quad m \in \NN_0
$$
and 
$$
\left\{
\begin{array}{ccc}
w^{(m)}_- e_- & \mapsto &w^m e_{12} \\
w^{(m)}_- e_+ & \mapsto &w^m e_{21}
\end{array}
\right.
\quad m \in \NN
$$
is an algebra isomorphism we are looking for.
\end{proof}

\begin{proposition} For any $n \in \NN$, let 
$
G:= \bigl\langle \rho, \tau \, \big|\, \rho^n = e = \tau^2, \tau \rho 
\tau = \rho^{-1} \bigr\rangle
$
be the dihedral group and $\zeta \in \kk$ be a primitive $n$-th root of $1$. Consider
the action of $G$ on the nodal ring $D = \kk\llbracket u, v\rrbracket/(uv)$ given by the rules:
$$
\left\{
\begin{array}{ccc}
\rho \circ u & = & \zeta \; u \\
\rho \circ v & = & \zeta^{-1} v \\
\tau \circ u & = & v \\
\tau \circ v & = & u.
\end{array}
\right.
$$
Then the nodal order $A:= D \ast G$ has the following description. 
\begin{itemize}
\item If $n = 2l +1$ for $l \in \NN_0$ then $A \cong A(\Omega, \sigma, \approx, \wt_\circ)$, where
$$\Omega:= \bigl\{1, 2, \dots, n \bigr\}, \quad 
\sigma = 
\left(
\begin{array}{cccc}
1 & 2 & \dots & n \\
n & 1 & \dots & n-1
\end{array}
\right)
$$
and 
$(2k-1) \approx 2k$ for $1 \le k \le l$, whereas $(2l+1) \approx (2l+1)$.
\item If $n = 2l +2$ for $l \in \NN_0$ then $A \cong A(\Omega, \sigma, \approx, \wt_\circ)$, where
$$\Omega:= \bigl\{0, 1, 2, \dots, 2l+1 \bigr\}, \quad 
\sigma = 
\left(
\begin{array}{cccc}
1 & 2 & \dots & n \\
n & 1 & \dots & n-1
\end{array}
\right)
$$
and 
$(2k-1) \approx 2k$ for $1 \le k \le l$, whereas $0 \approx 0$ and $(2l +1) \approx
(2l+1)$.
\end{itemize}
\end{proposition}

\smallskip
\noindent
\emph{Sketch of the proof}. The cyclic group $K := \langle \rho\rangle$ is a normal subgroup of $G$ of index two. Let $L = \langle \tau\rangle$. Then we have a commutative diagram of $\kk$-algebras and algebra homomorphisms
$$
\xymatrix{
A \ast G \ar@{_{(}->}[d]  \ar[rr]^-\cong & & (A \ast K) \ast L \ar@{^{(}->}[d]\\
H \ast G  \ar[rr]^-\cong & & (H \ast K) \ast L,
}
$$
where the action of $L$ on $H \ast K = \bigl(\kk\llbracket u \rrbracket \ast K\bigr)
\times \bigl(\kk\llbracket v \rrbracket \ast K\bigr)$ is given by the rule
$$
\tau \circ \bigl(u^{k_1} \rho^{l_1}, v^{k_2} \rho^{l_2}\bigr) = 
\bigl(u^{k_2} \rho^{-l_2}, v^{k_1} \rho^{-l_1}\bigr)
$$
for any $k_1, k_2, l_1, l_2 \in \NN_0$. The nodal ring $A \ast K$ is described
in Remark \ref{R:nodalcyclic}. In the terms of the quiver presentation \eqref{E:nodalquiver} we have:
$$
\left\{
\begin{array}{ccc}
\tau \circ a_k  & = & b_{\bar{k}-1} \\
\tau \circ b_k  & = & a_{\bar{k}-1} \\
\tau \circ e_k  & = & e_{\bar{k}}
\end{array}
\right.
$$
 where $\bar{k} = (n-k)$ for $1 \le k \le n$ and  all indices are taken  modulo $n$. The remaining part  is a lengthy computation analogous to the one made in the course of 
 the proof of Lemma \ref{L:BasicNodalSkew} which we leave for the interested reader. 
 
\subsection{Auslander order of a nodal order}

\begin{definition}
Let $A$ be a nodal order  and $H$ be its hereditary cover. Then
\begin{equation}
C:= \left\{a \in A \, \big|\, a h \in A \; \mbox{\rm for all}\; h \in H \right\}
\end{equation}
is called \emph{conductor ideal} of $A$.
\end{definition}

\begin{remark}
It follows from the definition   that $C = A C H$ and the canonical morphism
\begin{equation}\label{E:conductorlocal}
C \lar \Hom_A(H, A), \quad c \mapsto \bigl(h \mapsto ch\bigr)
\end{equation}
is a bijection  (here, we view both $H$ and $A$ as left $A$-modules).
\end{remark}

\begin{proposition}\label{P:ConductorNodalOrder} Let $R$ be a discrete valuation ring and
$A = A\bigl(R, (\Omega, \sigma, \approx, \mathsf{wt})\bigr)$ be the nodal order from the Definition \ref{D:nodalrings}. Then  we have:
\begin{equation*}
 C =  \left\{(X_1, \dots X_t) \in A\; \left| \; X_{i'}^{(\omega',\, \omega')}(0) =  0 = X_{i''}^{(\omega'',\, \omega'')}(0) \; \mbox{\rm for all}\; \begin{array}{l}
1 \le i', i'' \le t \\
\omega' \in \Omega_{i'}, \; \omega'' \in \Omega_{i''} \\
\omega' \approx \omega''
\end{array}
\right.
\right\}
\end{equation*}
and $C = \left\{a \in A \, \big|\, h a \in A \; \mbox{\rm for all}\; h \in H \right\}$. In particular, $C$ is a \emph{two-sided} ideal both in $H$ and $A$ containing the common radical $J = \rad(A) = \rad(H)$ of $A$ and $H$.
\end{proposition}

\begin{corollary}\label{C:ConductorNodalOrder}
Let $\kk$ be an algebraically closed field, $R$ be the local ring of an affine curve over $\kk$ at a smooth point,
$H$ be a hereditary $R$-order, $A$ be a nodal order whose hereditary cover is $H$ and $C$ be the corresponding conductor ideal.
Then $C$ is a two sided ideal in $H$.
\end{corollary}

\begin{proof}
We have to show that the canonical map of $R$-modules $C \lar HCH$ is surjective. For this, it is sufficient to prove the corresponding statement for the  radical completions of $A$ and $H$. However,  the structure of nodal orders over $\widehat{R} \cong \kk\llbracket w\rrbracket$
is known; see Theorem
\ref{T:ClassificationNodalOrders}. Hence, the statement follows from Proposition \ref{P:ConductorNodalOrder}.
\end{proof}

\begin{lemma}
Let $A = A\bigl(R, (\Omega, \sigma, \approx, \mathsf{wt})\bigr)$ be a nodal order, $H$ be its hereditary cover, $C$ be the conductor ideal,
$\bar{A}:= A/C$ and $\bar{H} := H/C$. Let $\Omega_\circ$ be the subset of $\Omega$ whose elements are reflexive or tied elements of $\Omega$ and $\widetilde{\Omega}^\ddagger_\circ$ be the subset of $\widetilde{\Omega}^\ddagger$ defined in a similar way.
Then the following  diagram
\begin{equation}\label{E:NodalModConductor}
\begin{array}{c}
\xymatrix{
\bar{A} \ar[rr]^-{\cong} \ar@{_{(}->}[d] & & \prod\limits_{\gamma \in \widetilde\Omega_\circ^\ddagger} \bar{A}_{\gamma}
\ar@{^{(}->}[d]^-{\imath}\\
\bar{H} \ar[rr]^-{\cong}  & & \prod\limits_{\omega \in \Omega_\circ} \bar{H}_{\omega}
}
\end{array}
\end{equation}
is commutative, where the components of the embedding $\imath$ are described in the same way as in  diagram \eqref{E:gentle}.
\end{lemma}

\begin{definition} Let $A$ be a nodal order, $H$ be its hereditary cover and $C$ be the corresponding conductor ideal. The order
\begin{equation}\label{E:AuslanderOrder}
B  :=
\left(
\begin{array}{cc}
A & H \\
C & H
\end{array}
\right)
\end{equation}
is called \emph{Auslander order} of $A$.
\end{definition}

\begin{example}\label{Ex:AuslanderOrderI} The Auslander order of the commutative nodal ring  $\kk\llbracket u,v\rrbracket/(uv)$
is
$$
\left(
\begin{array}{cc}
\kk\llbracket u,v\rrbracket/(uv) & \kk\llbracket u\rrbracket \times 
\kk\llbracket v \rrbracket \\
(u, v) & \kk\llbracket u\rrbracket \times \kk\llbracket v\rrbracket
\end{array}
\right),
$$
i.e.~the order from Example \ref{Ex:AuslanderOrder}.
\end{example}

\begin{example}\label{Ex:AuslanderOrderII}
Let $R = \kk\llbracket w\rrbracket$, $\idm = (w)$ and 
$
A = 
\left(
\begin{array}{cc}
R  &  \idm \\
\idm   & R
\end{array}
\right)
$
be the nodal order from Example \ref{Ex:NodalOrbifolded}.
The hereditary cover $H$ of $A$ ist just the matrix algebra 
$\Mat_2(R)$, whereas the corresponding conductor ideal $C = \Mat_2(\idm)$. Therefore, the Auslander order of $A$ is
$$
B = 
\left(
\begin{array}{cccc}
R  &  \idm  & R & R \\
\idm   & R  & R &  R \\
\idm  &  \idm  & R & R \\
\idm   & \idm  & R &  R \\
\end{array}
\right).
$$
It is easy to see that $B$ is Morita equivalent to the Gelfand order \eqref{E:GelfandOrder}.
\end{example}

\smallskip
\noindent
Let $A$ be an arbitrary  nodal order and $B$ be the Auslander order of $A$.
For the idempotents
$
e := \left(
\begin{array}{cc}
1 & 0 \\
0 & 0
\end{array}
\right), \,
f := \left(
\begin{array}{cc}
0 & 0 \\
0 & 1
\end{array}
\right)
 \in B$, consider the corresponding  projective left $B$-modules $
P := B e =
\left(
\begin{array}{c}
A \\
C
\end{array}
\right)
$  and $
Q := B f =
\left(
\begin{array}{c}
H \\
H
\end{array}
\right).
$
Note that $$A = e B e \cong \bigl(\End_B(P)\bigr)^\circ \quad \mbox{\rm and} \quad
H = f B f \cong  \bigl(\End_B(Q)\bigr)^\circ.$$
The action of $B$ on the projective left $B$-modules $\left(
\begin{array}{c}
A \\
C
\end{array}
\right)$ and $\left(
\begin{array}{c}
H \\
H
\end{array}
\right)$  is given by the matrix multiplication, whereas the isomorphisms $A \cong \bigl(\End_B(P)\bigr)^\circ$, respectively
$H \cong \bigl(\End_B(Q)\bigr)^\circ$,  are compatible with the canonical right actions on $P$, respectively $Q$.
 The nodal order  $A$ as well as its hereditary cover $H$ are \emph{minors} of the Auslander order $B$ in the sense of Definition
 \ref{D:Minor}.
Let
$$
P^\vee := \Hom_B(P, B) \cong eB \quad \mbox{\rm and} \quad Q^\vee := \Hom_B(Q, B) \cong  fB.
$$
\smallskip
\noindent
In the terms of Subsection \ref{SS:Minors}, we have   the following functors.
\begin{itemize}
\item Since $P$ is a projective left $A$-module, we  get an exact functor $$\sG := \Hom_B(P, \,-\,) \simeq P^\vee \otimes_B \,-\,$$ from  $B\mathsf{-mod}$ to $A\mathsf{-mod}$.
Of course, it restricts to an exact functor $$B\mathsf{-mod} \stackrel{\sG}\lar A\mathsf{-mod}$$ between the corresponding categories of finitely generated modules.
\item  Similarly, we have an exact functor  $\widetilde\sG = \Hom_B(Q, \,-\,) \simeq Q^\vee \otimes_B \,-\,$ from
$B\mathsf{-mod}$ to $H\mathsf{-mod}$, as well as its restriction on the full subcategories of the corresponding finitely generated modules.
\item We have functors $\sF := P \otimes_A \,-\,$ and $\sH:= \Hom_A(P^\vee, \,-\,)$ from $A\mathsf{-mod}$ to  $B\mathsf{-mod}$.
\item Similarly, we have functors  $\widetilde\sF := Q \otimes_H \,-\,$ and $\widetilde\sH:= \Hom_H(Q^\vee, \,-\,)$ from $H\mathsf{-mod}$ to  $B\mathsf{-mod}$.
\end{itemize}

\smallskip
\noindent
Additionally to Theorem \ref{T:minorsabelian}, the following result is true.
\begin{proposition} The functor $\widetilde\sF$ is exact, maps projective modules to projective modules  and has the following explicit description: if $\widetilde{N}$ is a left $H$-module, then
$$
\widetilde\sF(\widetilde{N}) =
\left(
\begin{array}{c}
\widetilde{N} \\
\widetilde{N}
\end{array}
\right) \cong \widetilde{N} \oplus \widetilde{N},
$$
where for  $b = \left(\begin{array}{cc}
b_1 & b_2 \\
b_3 & b_4
\end{array}
\right) \in B$
  and   $z = \left(\begin{array}{c}
z_1 \\
z_2
\end{array}
\right) \in
\left(
\begin{array}{c}
\widetilde{N} \\
\widetilde{N}
\end{array}
\right),
$  the element $bz$ is given by the matrix multiplication.
\end{proposition}

\begin{remark} Since the order $H$ is hereditary,
the conductor ideal $C$ is projective viewed as a left $H$-module. The functor $\widetilde\sF$ transforms projective modules into projective modules, hence the left $B$-module $\left(
\begin{array}{c}
C \\
C
\end{array}
\right)
$ is projective, too. Consider   the left $B$-modules $S$ and $T$, given by the projective resolutions
\begin{equation*}
0 \lar
\left(
\begin{array}{c}
C \\
C
\end{array}
\right)
\lar \left(
\begin{array}{c}
A \\
C
\end{array}
\right)
\lar S \lar 0 \;\;\mbox{\rm and}\;\;
0 \lar
\left(
\begin{array}{c}
C \\
C
\end{array}
\right)
\lar \left(
\begin{array}{c}
H \\
H
\end{array}
\right)
\lar T\lar 0.
\end{equation*}
Obviously, both $S$ and $T$ have finite length viewed as $B$-modules. Moreover, $S$  is isomorphic to $\bar{A}$ viewed as an $A$-module and
$T$ is isomorphic to $\bar{H}$ viewed as an $H$-module.
\end{remark}

\begin{theorem} The following results are true.
\begin{itemize}
\item $(\sL \sF, \sD\sG, \sR\sH)$ and $(\sL\widetilde\sF, \sD\widetilde\sG, \sR\widetilde\sH)$ are triples  pairs of functors.
\item The functors $\sL\sF$, $\sL\widetilde\sF$, $\sR\sH$ and $\sR\widetilde\sH$ are fully faithful, whereas the functors $\sD\sG$ and
$\sD\widetilde\sG$
are essentially surjective.
\item The essential image of $\sL\widetilde\sF$ is equal to the left orthogonal category
$$
^{\perp} S := \bigl\{X^\bu \in \mathrm{Ob}\bigl(D(B\mathsf{-mod})\bigr) \;\big| \; \Hom\bigl(X^\bu, S[i]\bigr) = 0 \; \mbox{\rm for all} \; i \in \ZZ\bigr\}
$$
of $S$, whereas the essential image of $\sR\widetilde\sH$ is equal to the right orthogonal category $S^\perp$. Similarly, the essential image of $\sL\sF$ is equal to $^{\perp} T$ and the essential image of $\sR\sH$ is equal to $T^\perp$.
\item We have a recollement diagram
\begin{equation}\label{E:recollement1}
\xymatrix{D(\bar{A}\mathsf{-mod}) \ar[rr]|{\,\sI\,} && D(B\mathsf{-mod}) \ar@/^2ex/[ll]^{\sI^{!}} \ar@/_2ex/[ll]_{\sI^*} \ar[rr]|{\,\sD\widetilde\sG\,}
  && D(H\mathsf{-mod}) \ar@/^2ex/[ll]^{\,\sR\widetilde\sH\,} \ar@/_2ex/[ll]_{\,\sL\widetilde\sF\,}},
\end{equation}
where $\sI(\bar{A}) := S$, the functor $\sI^*$ is left adjoint to $\sI$ and $\sI^!$ is right adjoint to $\sI$.
\item Similarly, we  have another recollement diagram
\begin{equation}\label{E:recollement2}
\xymatrix{D_T(B\mathsf{-mod}) \ar[rr]|{\,\sJ\,} && D(B\mathsf{-mod}) \ar@/^2ex/[ll]^{\sJ^{!}} \ar@/_2ex/[ll]_{\sJ^*} \ar[rr]|{\,\sD\sG\,}
  && D(A\mathsf{-mod}) \ar@/^2ex/[ll]^{\,\sR\sH\,} \ar@/_2ex/[ll]_{\,\sL\sF\,}},
\end{equation}
where $D_T(B\mathsf{-mod})$ is the full subcategory of the derived category $D(B\mathsf{-mod})$ consisting of those complexes whose cohomologies belong to $\mathsf{Add}(T)$ and  $\sJ$ is the canonical inclusion functor.
\item We have: $\mathsf{gl.dim}B = 2$.
\end{itemize}
\end{theorem}

\begin{proof} These results are specializations of Theorem \ref{T:minorsderived}.
The statements about both recollement diagrams \eqref{E:recollement1} and \eqref{E:recollement2} follow from the description of the kernels of the functors $\sD\widetilde\sG$ and $\sD\sG$. Namely, consider the two-sided ideal
$$
I_Q:= \mathrm{Im}\bigl(Q \otimes_A Q^\vee \stackrel{\mathrm{ev}}\lar B\bigr) =
\left(
\begin{array}{cc}
C & H \\
C & H
\end{array}
\right)
$$
in the algebra $B$. As one can easily see, $I_Q$ is projective viewed as a right $B$-module. Moreover, $B/I_Q \cong A/C =: \bar{A}$ is semisimple. Hence, Theorem  \ref{T:minorsderived} gives the first recollement diagram \eqref{E:recollement1}. Analogously, for
$$
I_P:= \mathrm{Im}\bigl(P \otimes_H P^\vee \stackrel{\mathrm{ev}}\lar B\bigr) =
\left(
\begin{array}{cc}
A & H \\
C & C
\end{array}
\right)
$$
the algebra $B/I_P \cong H/C =: \bar{H}$ is again semisimple. However, this time $I_P$ is not projective viewed as a right $B$-module.

\smallskip
\noindent
To show the last statement, note that
according to Theorem \ref{T:minorsderived} we have: $\mathsf{gl.dim}B \le 2$. Since $A$ is a non-hereditary minor of $B$,
the order $B$ itself can not be hereditary; see Theorem \ref{T:HeredOrders}. Hence $\mathsf{gl.dim}B =  2$, as claimed.
\end{proof}

\begin{corollary}\label{C:minorsderived}
We have a recollement diagram
\begin{equation}\label{E:recollement3}
\xymatrix{D^b(\bar{A}\mathsf{-mod}) \ar[rr]|{\,\sI\,} && D^b(B\mathsf{-mod}) \ar@/^2ex/[ll]^{\sI^{!}} \ar@/_2ex/[ll]_{\sI^*}
 \ar[rr]|{\,\sD\widetilde\sG\,}
  && D^b(H\mathsf{-mod}) \ar@/^2ex/[ll]^{\,\sR\widetilde\sH\,} \ar@/_2ex/[ll]_{\,\sL\widetilde\sF\,}
  }.
\end{equation}
As a consequence, we have a semiorthogonal decomposition
\begin{equation}
D^b(B\mathsf{-mod}) = \left\langle \mathsf{Im}(\sI), \, \mathsf{Im}(\sL\widetilde\sF)\right\rangle = \left\langle
D^b\bigl(\bar{A}\mathsf{-mod}), \,
D^b(H\mathsf{-mod})
\right\rangle.
\end{equation}
Moreover, we have the following commutative diagram of categories and functors:
\begin{equation}
\begin{array}{c}
\xymatrix{
D^b(H\mathsf{-mod}) \ar[rrd]_{\sP} \ar@{^{(}->}[rr]^-{\sL\widetilde\sF} & & D^b(B\mathsf{-mod}) \ar@{->>}[d]_-{\sD\sG} & & \ar@{_{(}->}[ll]_-{\sL\sF} \Perf(A) \ar@{_{(}->}[lld]^-{\sE} \\
 & & D^b(A\mathsf{-mod}) & &
 }
 \end{array}
\end{equation}
where $\Perf(A)$ is the perfect derived category of $A$, $\sE$ is the canonical inclusion functor and $\sP$ is the derived functor of the
restriction functor $H\mathsf{-mod} \lar A\mathsf{-mod}$.
\end{corollary}

\begin{proof}
The recollement diagram \eqref{E:recollement3} is just the restriction of the recollement diagram \eqref{E:recollement1} on the corresponding full subcategories of compact objects. The isomorphism $\sE \simeq \sD\sG \circ \sL\sF$ follows from the fact that the adjunction unit
$\mathsf{Id}_{D(A\mathsf{-mod})} \lar \sD\sG \circ \sL\sF$ is an isomorphism of functors (already on the level of unbounded derived categories). Next, since the functor $\widetilde{\sF}$ is exact, we have: $\sD\sG \circ \sL\widetilde\sF \simeq \sD (\sG \circ \widetilde\sF)$. For any $H$-module $\widetilde{N}$ we have:
$$
(\sG \circ \widetilde\sF)(\widetilde{N}) = \Hom_B\bigl(P, \widetilde\sF(\widetilde{N})\bigr) =
\Hom_B\left(Be, \left(
\begin{array}{c}
\widetilde{N} \\
\widetilde{N}
\end{array}
\right)\right) \cong e \cdot \left(
\begin{array}{c}
\widetilde{N} \\
\widetilde{N}
\end{array}
\right) \cong \widetilde{N}.
$$
Hence, $\sG \circ \widetilde\sF$ is isomorphic to the restriction functor $H\mathsf{-mod} \lar A\mathsf{-mod}$, what finishes a proof of the second statement.
\end{proof}

\begin{proposition} Let $A$ be a nodal order. Then the corresponding Auslander order $B$ is nodal too.
\end{proposition}
\begin{proof}
As usual, let $H$ be the hereditary cover of $A$ and $J = \rad(A) = \rad(H)$ be the common radical of $A$ and $H$. Consider the following
orders:
$$
\widetilde{B}:=
\left(
\begin{array}{cc}
A & H \\
J & H
\end{array}
\right)
\quad \mbox{\rm and} \quad
\widetilde{H}:=
\left(
\begin{array}{cc}
H & H \\
J & H
\end{array}
\right).
$$ 
It is not difficult to show that
$$
\widetilde{J}:= \rad(\widetilde{H}) = \left(
\begin{array}{cc}
J & H \\
J & J
\end{array}
\right) = \rad(\widetilde{B}).
$$
 Since $\widetilde J$ is projecive as $H$-module, $H$ is hereditary.
Then the commutative diagram
$$
\begin{array}{c}
\xymatrix{
 \widetilde{B}/\widetilde{J} \ar[rr]^-{\cong} \ar@{_{(}->}[d] & &  A/J \times H/J \ar@{^{(}->}[d]\\
\widetilde{H}/\widetilde{J} \ar[rr]^-{\cong}  & & H/J \times H/J
}
\end{array}
$$
implies
that $\widetilde{B}$ is a nodal order and $\widetilde{H}$ is its hereditary cover. Since the conductor ideal $C$ contains
the radical $J$, the Auslander order $B$ is an overorder of $\widetilde{B}$.  It follows from Theorem \ref{T:fundamentalsNodalOrders} that the order $B$ is nodal, too.
\end{proof}

\smallskip
\noindent
In what follows, we shall need the following result about the finite length $B$-module $S$. Assume that $R = \kk\llbracket w\rrbracket$ and
$A = A\bigl(R, (\Omega, \sigma, \approx, \mathsf{wt})\bigr)$.
Recall that
$$
\bar{A} := A/C \cong
\prod\limits_{\gamma \in \widetilde\Omega_\circ^\ddagger} \bar{A}_\gamma \cong \prod\limits_{\gamma \in \widetilde\Omega_\circ^\ddagger} \Mat_{\wt(\gamma)}(\kk).
$$
It is clear that the set $\widetilde\Omega_\circ^\ddagger$ also parameterizes  the isomorphism classes of the simple $\bar{A}$-modules. For any $\gamma \in \widetilde\Omega_\circ^\ddagger$, let $S_\gamma$ be the simple left $B$-module which corresponds to the (unique, up to an isomorphism) simple
$\bar{A}_\gamma$-module  and $P_\gamma$ be its projective cover.  Then we have:
$
S  \cong \bigoplus\limits_{\gamma \in \widetilde\Omega_\circ^\ddagger} S_\gamma^{\oplus \wt(\widetilde\Omega_\circ^\ddagger)}.
$
Our next goal is  to describe a minimal projective resolution of $S_\gamma$. For any $\omega \in \Omega$, let $\widetilde{Q}_{\omega}$ be the corresponding indecomposable projective left $H$-module and
$$
Q_\omega := \widetilde\sF\bigl(\widetilde{Q}_{\omega}\bigr) = \left(
\begin{array}{c}
\widetilde{Q}_{\omega} \\
\widetilde{Q}_{\omega}
\end{array}
\right)
$$
be the corresponding indecomposable projective left $B$-module.

\begin{lemma}\label{L:ExtLOcalComp} The following statements hold.
\begin{itemize}
\item Let $\omega \in \Omega$ be a reflexive element and $\gamma = \omega_\pm$ be one of the corresponding elements of $\widetilde\Omega_\circ^\ddagger$. Then  
\begin{equation}\label{E:ResolutionCase1}
0 \lar Q_{\sigma(\omega)}  \lar P_\gamma \lar S_\gamma \lar 0
\end{equation}
is a  minimal projective resolution of the simple $B$-module $S_\gamma$.
In particular, for any $\delta \in \Omega$ we have:
\begin{equation}\label{E:ExtCase1}
\Ext^1_B\bigl(S_\gamma, Q_\delta\bigr) \cong
\left\{
\begin{array}{cl}
\kk & \mbox{\rm if} \;  \delta = \sigma(\omega) \;  \\
0 & \mbox{\rm otherwise}.
\end{array}
\right.
\end{equation}
\item Let $\omega', \omega'' \in \Omega$ be a pair of tied elements and $\gamma = \overline{\{\omega', \omega''\}}$ be the corresponding element of $\widetilde\Omega_\circ^\ddagger$. Then  a minimal projective resolution of the simple $B$-module $S_\gamma$ has the following form:
\begin{equation}\label{E:ResolutionCase2}
0 \lar Q_{\sigma(\omega')} \oplus Q_{\sigma(\omega'')} \lar P_\gamma \lar S_\gamma \lar 0.
\end{equation}
In particular, for any $\delta \in \Omega$ we have:
\begin{equation}\label{E:ExtCase2}
\Ext^1_B\bigl(S_\gamma, Q_\delta\bigr) \cong
\left\{
\begin{array}{cl}
\kk & \mbox{\rm if} \;  \delta = \sigma(\omega) \; \mbox{\rm for} \; \omega   \in \{\omega', \omega''\} \\
0 & \mbox{\rm otherwise}.
\end{array}
\right.
\end{equation}
\end{itemize}
\end{lemma}

\begin{proof}
 It is not difficult to show that $\rad(P_\gamma) = Q_{\sigma(\omega)}$ for $\gamma = \omega_\pm$ if $\omega \in \Omega$ is reflexive and $\rad(P_\gamma) = Q_{\sigma(\omega')} \oplus Q_{\sigma(\omega'')}$ if  $\gamma = \overline{\{\omega', \omega''\}}$ for $\omega', \omega'' \in \Omega$ tied. The formulae \eqref{E:ExtCase1} and \eqref{E:ExtCase2} follow from the fact that  $\rad(Q_\omega) = Q_{\sigma(\omega)}$ for any $\omega \in \Omega$.
\end{proof}

\section{Non-commutative nodal curves: global theory}  

\noindent
In this section, we are going to explain the construction as well as main  properties  of non-commutative nodal curves of tame representation type.

\subsection{The idea of a non-commutative nodal curve} Before going to technicalities and details, let us  consider the following  example.
Let $\kk$ be a field, $S = \kk[x]$, $J = (x^2-1)$ and $K = \kk(x)$. Consider the hereditary order
$$
H = 
\left(
\begin{array}{ccc}
S & J &  J \\
S & S & J \\
S & S  & S
\end{array}
\right)
\subset \Upsilon  := \Mat_{3}(K).
$$
Next, consider the order
$$
 A:= \left\{X  \in H \left|  \begin{array}{l}
 X_{11}(1) = X_{22}(1) \\
 X_{33}(1) = X_{33}(-1) \\
 X_{21}(-1) = 0
 \end{array}
 \right.
 \right\} \subset H.
$$
 Let $Z = Z(A)$ be the center of $A$. Then we have:
$$
Z = \left\{p \in S \; \big| \; p(1) = p(-1)\right\} = \kk\bigl[x^2-1, x(x^2-1)\bigr] \cong \kk[u, v]/(v^2 - u^3 - u^2).
$$
The multiplication maps $K \otimes_Z A \lar \Upsilon$ and $K \otimes_Z H \lar \Upsilon$ are isomorphisms. In other words, $A$ and $H$ are $Z$-\emph{orders} in the \emph{central simple} $K$-algebra $\Upsilon$.

\smallskip
\noindent
Let $E_\circ = V(v^2-u^3-u^2) \subset \mathbbm{A}^2$ be an affine plane nodal cubic and  
$s = (0, 0) \in E_\circ$ be its unique singular point. For any $x \in E_\circ \setminus \{s\}$ we have:
$A_x = \Mat_3\bigl(O_x\bigr)$, where $O_x$ is the local ring of $E_\circ$  at the point $x$.
On the other hand,
$$
\widehat{A}_s = A\bigl(R, (\Omega, \sigma, \approx, \mathsf{wt}_\circ)\bigr)
$$
is a nodal order, where $R=\kk\llbracket t\rrbracket$,
$$\Omega:= \bigl\{1, 2, 3, 4, 5\bigr\}, \quad 
\sigma = 
\left(
\begin{array}{ccccc}
1 & 2 & 3 & 4 & 5 \\
2 & 3 & 1 & 5 & 4
\end{array}
\right)
$$
and 
$1 \approx 2$, $3 \approx 5$ and $4 \approx 4$.

\smallskip
\noindent
Let $E = \overline{V(v^2-u^3-u^2)} \subset \PP^2$ be the projective closure of $E_\circ$. Then the $Z$-order $A$ can be extended 
 to a sheaf  of orders
$\kA$   on the projective curve $E$  in such a way that the stalk of $\kA$ at  the ``infinite point'' $(0:1:0)$ of $E$  is  a maximal order (see for instance \cite{BL}).
The ringed space $\mathbbm{E} = (E, \kA)$ is a typical  example of a projective non-commutative nodal curve of tame representation type. 

\smallskip
\noindent
Let 
$
H' = 
\left(
\begin{array}{ccc}
S & I &  J \\
S & S & J \\
S & S  & S
\end{array}
\right),
$
where $J \subset I := (x-1) \subset S$. Then $H'$ is a hereditary order, too.
Moreover, $H'$ is the hereditary cover of the order $A$ (the notion of the hereditary cover can be defined locally). As above, we can extend  $H'$ to a sheaf of orders $\kH$ on the 
 projective curve $E$ in such a way that the stalk of $\kH$ at the infinite point is the maximal order. Thus we get a non-commutative curve $\widetilde{\mathbbm{E}} =
 (E, \kH)$. Note that the center of $H$ is $S$. Therefore, one can actually 
 construct a sheaf of hereditary orders $\widetilde\kH$ on $\PP^1$ such that
 $\kH = \nu_\ast(\widetilde\kH)$, where $\PP^1 \stackrel{\nu}\lar E$ is the normalization map. The functor $\nu_\ast$ provides an equivalence between the categories of coherent sheaves on $\widetilde{\mathbbm{E}}$ and $(\PP^1, \widetilde\kH)$. In what follows, we shall consider $\widetilde{\mathbbm{E}}$ as the hereditary cover of the non-commutative nodal curve $\mathbb{E}$ what can be viewed as an appropriate non-commutative generalization of the normalization of a singular commutative nodal curve.
 
\begin{definition}
Let $X$ be a reduced quasi-projective curve over a field $\kk$ and $\kA$ be a sheaf 
of orders on $X$. Then the ringed space 
$\XX = (X, \kA)$ is called a  \emph{non-commutative curve}. We say that $\XX$ is  projective if 
the commutative  curve $X$ is projective. 
If for any point $x \in X$ the corresponding stalk $\kA_x$ is a nodal order then $\XX$ is a \emph{non-commutative nodal curve}.
\end{definition}

\begin{remark}
A non-commutative curve $\YY = (Y, \kB)$ is called \emph{central} if 
$\widehat{O}_y = Z(\widehat{\kB}_y)$ for any $y \in Y$. For any non-commutative curve $\XX$ there exists a central non-commutative curve $\YY$ such that the categories $\Coh(\XX)$ and $\Coh(\YY)$ are equivalent (we say that $\XX$ and $\YY$ are Morita equivalent); see \cite[Section 2.4]{BurbanDrozdMorita}). For a central non-commutative curve
$\YY = (Y, \kB)$, we denote by
\begin{equation}
\mathfrak{S}(\YY) := \left\{y \in Y \, \big| \, \widehat{\kB}_y \quad \mbox{\rm is not maximal} \right\}
\end{equation}
the set of \emph{non-regular points} of $\YY$. It is not difficult to prove that
$\mathfrak{S}(\YY)$ is in fact finite; see for instance 
\cite[Lemma 7.1]{BurbanDrozdMorita}.
\end{remark}

\begin{theorem}\label{T:MoritaUniqueness}
Let $\kk$ be an algebraically closed field, $X$ be a quasi-projective curve over $\kk$,
$\kA, \kB$ be two central sheaves of orders on $X$ and 
$\XX = (X, \kA), \YY = (X, \kB)$ be the corresponding non-commutative curves. 
Assume that $\mathfrak{S}(\XX) = \mathfrak{S}(\YY)$ and for 
any $x \in \mathfrak{S}(\XX)$, the corresponding orders $\widehat{\kA}_x$
and $\widehat{\kB}_x$ are \emph{centrally} Morita equivalent. Then  the categories 
$\Coh(\XX)$ and $\Coh(\YY)$ are equivalent.
\end{theorem}

\smallskip
\noindent
For a proof of this result, see \cite[Proposition 7.7]{BurbanDrozdMorita}.

\subsection{Construction of non-commutative nodal curves}
Let $\kk$ be an algebraically closed field  and $(\tX, \kO_{\tX})$ be a smooth quasi-projective curve over $\kk$.

\begin{itemize}
\item Let $\tX \stackrel{l}\lar \NN$ be a function such that $l(\tx) = 1$ for all but finitely many points $\tx \in \tX$ (such function will be called a \emph{length function}).
\item For any $\tx \in \tX$ we put: 
\begin{equation}\label{E:SetPi}
\Pi_{\tx} := \bigl\{(\tx, 1), \dots, (\tx, l(\tx))\bigr\} \quad \mbox{\rm and}\quad
\Pi:= \bigcup\limits_{\tx \in \tX} \Pi_{\tx}.
\end{equation}
\item For any function $\Pi \stackrel{\mathsf{wt}}\lar \NN$ and any point $\tx \in \tX$ we denote:
$$
\vec{p}(\tx):= \bigl(\mathsf{wt}(\tx, 1), \dots, \mathsf{wt}(\tx, l(\tx))\bigr) \quad \mbox{\rm and} \quad
m(\tx):= \big|\vec{p}(\tx)\big| :=  \sum\limits_{i = 1}^{l(\tx)} \mathsf{wt}(\tx, i).
$$
\item We say that $\mathsf{wt}$ is a \emph{weight function} compatible with the given length function $l$ if $m(\tx') = m(\tx'')$ for any pair of points  $\tx', \tx'' \in \tX$  belonging to the same irreducible component of $\tX$.
\item  For $\tx \in \tX$, let  $O_{\tx}$ be the stalk of the structure sheaf $\kO_{\tX}$ at the point $\tx$. Let
$$\widetilde{H}_{\tx}:= H\bigl(O_{\tx},  \vec{p}(\tx)\bigr) \subseteq \Mat_{m(\tx)}(O_{\tx})$$ be the standard hereditary order defined
by \eqref{E:standardorder}.
\end{itemize}

\begin{definition}
Assume (for simplicity of notation) that $\tX$ is connected, $\tX \stackrel{l}\lar \NN$ be a length function and $\mathsf{wt}$ be a weight function. Let $m = m(\tx)$ for some (hence for any) point $\tx \in \tX$.  Then we define the sheaf of hereditary orders $\tH = \tH(l, \mathsf{wt})$ on the curve $\tX$ as the subsheaf of the sheaf of maximal orders $\sMat_{m}(\kO_{\tX})$ such that
$\tH_{\tx} = \widetilde{H}_{\tx}$ for all $\tx \in \tX$. The corresponding ringed space $\widetilde\XX:= (\tX, \tH)$  is a \emph{non-commutative hereditary curve} over the field $\kk$ defined by the datum $(\tX, l, \mathsf{wt})$.
\end{definition}

\begin{theorem} Let $\tX$ be a smooth quasi-projective curve over $\kk$ and $\tX \stackrel{l}\lar \NN$ be  a length function. Let
$\mathsf{wt}, \mathsf{wt}': \Pi \lar \NN$ be two  weight functions compatible with $l$ and $\widetilde\XX$ and $\widetilde\XX'$
be the corresponding non-commutative hereditary
curves. Then the categories $\Qcoh(\widetilde\XX)$ and $\Qcoh(\widetilde\XX')$ (respectively, 
$\Coh(\widetilde\XX)$ and $\Coh(\widetilde\XX')$)
are equivalent. Let $\widetilde{Y}$ be a   smooth quasi-projective curve over $\kk$, $\widetilde{Y} 
\stackrel{t}\lar \NN$ be a length function and $\widetilde\YY$ be the corresponding non-commutative hereditary curve.
Then $\widetilde\XX$ and $\widetilde\YY$ are Morita equivalent if and only if there exists an isomorphism (of commutative curves)
$\widetilde{X} \stackrel{f}\lar \widetilde{Y}$ such that the following diagram
$$
\xymatrix{
\widetilde{X} \ar[rr]^-f  \ar[rd]_-{l} & & \widetilde{Y} \ar[ld]^-{t}\\
& \NN & 
}
$$
is commutative. In other words, the Morita type of a non-commutative hereditary curve does not depend on the choice of a weight function $\mathsf{wt}$ and is  determined by the underlying commutative curve 
$\widetilde{X}$ and length function $l$.
\end{theorem}

\smallskip
\noindent
\emph{Comment to the proof}. This result is due to Spie\ss{} \cite{Spiess}, see also \cite[Section 4.3]{BurbanDrozdMorita}.

\begin{remark}
Let $\tX = \PP^1$, $\tX \stackrel{l}\lar \NN$ be a length function,
$\Pi \stackrel{\wt}\lar \NN$ be a weight function compatible with $l$
 and $\XX$ be the corresponding hereditary curve.  Then 
$\XX$  can be identified with an appropriate \emph{weighted projective line}
of Geigle and Lenzing \cite{GeigleLenzing} 
in the sense that the categories of (quasi-)coherent sheaves on both objects are equivalent (see, for instance, 
the paper \cite{ReitenvandenBergh}).

\smallskip
\noindent
 Let us choose homogeneous coordinates on $\PP^1$ and put: 
$\tilde{o}^+ := (0:1)$,  $\tilde{o}^- := (1:0)$ and $\tilde{o}:= (1:1)$. In what follows, we shall use the following notation.
\begin{itemize}
\item $\PP^1(n_+, n_-)$ is the weighted projective line corresponding 
to the length function given by the rule:
$$
l(\tilde{x}) = 
\left\{
\begin{array}{cl}
n_\pm & \mbox{\rm if} \;\;  \tilde{x} =  \tilde{o}^\pm \\
1 & \mbox{\rm otherwise}.
\end{array}
\right.
$$
\item $\PP^1(n_+, n_-, n)$ is the weighted projective line corresponding 
to the length function given by the rule:
$$
l(\tilde{x}) = 
\left\{
\begin{array}{cl}
n_\pm & \mbox{\rm if} \;\;  \tilde{x} =  \tilde{o}^\pm \\
n & \mbox{\rm if} \;\;  \tilde{x} =  \tilde{o} \\
1 & \mbox{\rm otherwise},
\end{array}
\right.
$$
where we additionally assume that $n_\pm \ge 2$.
\end{itemize}
\end{remark}

\begin{definition}\label{D:RelationNodalCurves}
Let $\tX$ be a smooth quasi-projective curve over $\kk$ and  $\tX \stackrel{l}\lar \NN$ be a length function. Let $\approx$ be a relation on the set $\Pi$ defined by \eqref{E:SetPi} such that 
\begin{itemize}
\item For any $\omega \in \Pi$ there exists at most one $\omega' \in \Pi$  such that
$\omega \approx \omega'$ (such elements $\omega, \omega'$  will be called \emph{special}).
\item There are only finitely many special elements in $\Pi$.
\end{itemize}
Non-special elements of $\Pi$ will be called \emph{simple}. The set of special elements of $\Pi$ will be denoted by $\Pi_\circ$. An element 
$\omega \in \Pi_\circ$ is called \emph{reflexive} if $\omega \approx \omega$ and \emph{tied} if
 $\omega \approx \omega'$ for some $\omega \ne \omega'$.
 
\smallskip
\noindent
Similarly to Definition \ref{D:EquivRelation} we define the set $\Pi^\ddagger$ by  replacing each reflexive element $\omega \in \Pi$ by two new simple elements
$\omega_+$ and $\omega_-$. The pairs of tied elements of $\Pi^\ddagger$ are the same as for $\Pi$.

\smallskip
\noindent
Let $\Pi^\ddagger \stackrel{\mathsf{wt}^\ddagger}\lar \NN$ be a function such that  $\mathsf{wt}^\ddagger(\omega') = \mathsf{wt}^\ddagger(\omega'')$ for all  $\omega' \approx \omega''$ in $\Pi^\ddagger$. Then we define the map  $\Pi \stackrel{\mathsf{wt}}\lar \NN$ by the following rule:
$$\mathsf{wt}(\omega) := 
\left\{
\begin{array}{cc}
\mathsf{wt}^\ddagger(\omega_+) + \mathsf{wt}^\ddagger(\omega_-) & \mbox{\rm if} \; 
\omega \in \Pi \;  \mbox{\rm is reflexive}\\
\mathsf{wt}^\ddagger(\omega) & \mbox{\rm otherwise}.
\end{array}
\right.$$ 

\smallskip
\noindent
We call such a relation $\approx$ on the set $\Pi$
\emph{admissible} if there exists
a  function $\Pi^\ddagger \stackrel{\mathsf{wt}^\ddagger}\lar \NN$ for which 
the corresponding function  $\Pi \stackrel{\mathsf{wt}}\lar \NN$
is a weight function 
compatible with  the length function  $l$.
Abusing the notation, we shall drop the symbol $\ddagger$ in the notation of $\mathsf{wt}^\ddagger$ and write $\mathsf{wt}$ for all weight functions introduced above.
\end{definition}

\smallskip
\noindent
 We say that two points $\tx'  \ne  \tx'' \in \tX$  are \emph{tied}  if there are $\omega' \in \Omega_{\tx'}$ and $\omega'' \in \Omega_{\tx''}$ such that $\omega' \approx \omega''$. Let 
\begin{equation}\label{E:tiedpoints}
 \tZ:= \left\{\tx \in \tX \,\big| \, \mbox{\rm there exists}\; \tilde{y} \in \tX \setminus \{\tx\} \; \mbox{\rm such that}\; \tx \; \mbox{\rm and}\; \tilde{y}
 \; \mbox{\rm are tied}\;\right\}
\end{equation}
 be the set of tied points of $\tX$.
 Taking the transitive closure, we get an equivalence relation  $\sim$  on  
 $\widetilde{Z}$. We put: 
 $ Z:= \tZ/\sim.
 $
In what follows, we shall also consider $\tZ$ as a reduced subscheme of $\tX$, $Z$ as a reduced scheme over $\kk$ and the projection map $\tZ \stackrel{\tilde\nu}\lar Z$ as a morphism of schemes.

 \smallskip
 \noindent
 Given an admissible  datum $(\tX, l, \approx)$, we define a quasi-projective curve $X$ requiring  the following diagram of algebraic schemes
\begin{equation}\label{E:gluing}
\begin{array}{c}
\xymatrix
{\widetilde{Z} \ar@{^{(}->}[r]^-{\tilde{\eta}} \ar[d]_-{\tilde{\nu}}
& \widetilde{X} \ar[d]^-\nu \\
Z \ar@{^{(}->}[r]^-\eta & X.
}
\end{array}
\end{equation}
to be  cartesian. In other words, the curve $X$ is obtained from $\tX$ by gluing transversally the equivalent points. It is  clear that $X$ is singular provided $\tZ$ is non-empty and that $\tX \stackrel{\nu}\lar X$ is the normalization map. It always exists, as follows from \cite{Serre}.

\smallskip
\noindent
We put: $\kH:= \nu_*\bigl(\tH\bigr)$. For any $x \in X$ (respectively, $\tx \in \tX$)  let $\widehat{H}_x$ (respectively, $\widehat{H}_{\tx}$) be the radical completion of $\kH_x$ (respectively, $\tH_{\tx}$).
Note  that in the notation of \eqref{E:standardorder} we have: $$\widehat{H}_{\tx}:= H\bigl(\widehat{O}_{\tx},  \vec{p}(\tx)\bigr),$$
where $\widehat{O}_{\tx}$ is the completion of the local ring $O_{\tx}$. It is clear that $\widehat{H}_{\tx}$ is also an order over the
local ring $\widehat{O}_x$, which is the completion of the local ring of the structure sheaf of $X$ at the point $x$.

\smallskip
\noindent
Assume now that
$x \in Z$ and  $\nu^{-1}(x) = \left\{\tx_1, \dots, \tx_r\right\}$.
Then we have a canonical isomorphism
$$
\widehat{H}_x \stackrel{\cong}\lar \widehat{H}_{\tx_1} \times \dots \times  \widehat{H}_{\tx_r}.
$$
Next, we put:
$
\Omega_x := \Omega_{\tx_1} \sqcup \dots \sqcup \Omega_{\tx_r}
$
Then we have a permutation $\Omega_x \stackrel{\sigma_x}\lar \Omega_x$ given by the rule $\sigma_x(\tx, i):= \bigl(\tx, i+1 \;\mathsf{mod}\; l(\tx)\bigr)$ for any $\tx \in \left\{\tx_1, \dots, \tx_r\right\}$. In the terms of Definition \ref{D:nodalrings} we put:
\begin{equation}
\widehat{A}_x:= A(\Omega_x, \sigma_x, \approx, \mathsf{wt}) \subset \widehat{H}_x.
\end{equation}
Then  $\widehat{A}_x$ is a nodal order and  $\widehat{H}_x$ is its hereditary cover. Moreover, the center
of $\widehat{A}_x$ contains the local ring $\widehat{O}_x$.
\begin{definition}\label{D:NonCommNodalCurve}
We define the sheaf of orders $\kA$ on the curve $X$ to be the subsheaf
of $\kH$ satisfying the following conditions on the stalks:
\begin{equation}
\widehat{\kA}_x :=
\left\{
\begin{array}{ccl}
\widehat{\kH}_x & \mbox{\rm if} & x \notin Z \\
\widehat{A}_x & \mbox{\rm if} & x \in  Z.
\end{array}
\right.
\end{equation}
We call the ringed space $\XX = (X, \kA)$ the \emph{non-commutative nodal curve} attached to the datum $(\tX, l, \approx, \mathsf{wt})$.
The ringed space $\widetilde{\XX} = (X, \kH)$  will be called the \emph{hereditary cover} of $\XX$.
\end{definition}

\noindent
 Note that for
$\widetilde{\XX}':= (\tX, \tH)$ we have a natural morphism of ringed spaces $\widetilde{\XX}' \stackrel{\nu}\lar \XX$, which induces
an equivalence of categories $\Coh(\widetilde{\XX}') \lar \Coh(\widetilde\XX)$.

\begin{theorem}\label{T:NodalMorita}
Let $(\tX, l, \approx)$ be an admissible datum, $\Pi^\ddagger \stackrel{\mathsf{wt}}\lar \NN$ be any compatible weight function and $\XX$ be the corresponding non-commutative nodal curve. Then the following results hold.
\begin{itemize}
\item Let
 $\Pi^\ddagger \stackrel{\mathsf{wt}'}\lar \NN$ be any other compatible weight function and
$\XX'$ be the corresponding non-commutative nodal curve. Then the categories of quasi-coherent sheaves $\Qcoh(\XX)$ and $\Qcoh(\XX')$ are equivalent.
 \emph{That is why we often do not mention the weight $\wt$ and say that $\XX$ is attached to the admissible datum $(\tX, l, \approx)$.}
 
\item Let $\approx'$ be another equivalence relation on $\Pi$ and
$\Pi \stackrel{\mathsf{wt}'}\lar \NN$ be a weight function compatible with $\approx'$.
Suppose  that for any $\tx \in \tX$ there exists a cyclic permutation $\Pi_{\tx} \stackrel{f_x}\lar \Pi_{\tx}$ such that the diagram
$$
\xymatrix{
\Pi_{\tx} \ar[rr]^-{f_x} \ar[rd]_-{\mathsf{wt}} & &  \Pi_{\tx} \ar[ld]^-{\mathsf{wt}'}\\
& \NN &
}
$$
is commutative. Then the categories  $\Qcoh(\XX)$ and $\Qcoh(\XX')$ are equivalent.
\end{itemize}
\end{theorem}

\smallskip
\noindent
\emph{Comment to the proof}. This result is a  consequence of  Theorem \ref{T:MoritaUniqueness} (proven in \cite{BurbanDrozdMorita}) and Theorem \ref{T:DescriptionNodalOrders}.

\begin{example}\label{Ex:AdmissibleDatum} Let $(\widetilde{X}, l, \approx)$ be such that for any $\tx \in \tX$ with $l(\tx) \ge 2$, the set $\Pi_{\tx}$ contains a non-tied element. Then the datum $(\tX, l, \approx)$ is admissible.
\end{example}

\begin{example} Let $\tX$ be any curve and $\tx_1 \ne \tx_2 \in \tX$ be two distinct points.
Define a length  function $\tX \stackrel{l}\lar \NN$ by the rule: 
$$
l(\tilde{x}) = 
\left\{
\begin{array}{cl}
2 & \mbox{\rm if} \;\;  \tilde{x} =  \tx_1 \\
1 & \mbox{\rm otherwise}.
\end{array}
\right.
$$
 Let $\approx$ be given by the rule: $(\tx_1, 1) \approx (\tx_2, 1)$. Then the datum $(\tX, l, \approx)$ is not admissible.
\end{example}

\subsection{Non-commutative nodal curves of tame representation type}

\smallskip
\noindent
In this subsection we recall, following the paper \cite{DrozdVoloshyn}, the description 
of those non-commutative projective nodal curves $\XX =(X, \kA)$ for which the category $\VB(\XX)$ of vector bundles (i.e.~of locally projective coherent $\kA$-modules) has tame representation type. 
Let
$\tX = \tX_1 \sqcup \dots \sqcup \tX_r$ be a disjoint union of $r$  projective lines. We  choose homogeneous coordinates
on each component $\tX_k$ and define points $\tilde{o}_k , \tilde{o}_k^\pm,\in \tX_k$ setting: $\tilde{o}_k:= (1:1), \tilde{o}_k^+ := (0:1)$ and $\tilde{o}_k^- := (1:0)$.
Assume that $(\tX, l, \approx)$ is an admissible datum defining a non-commutative nodal curve $\XX$.
For each $1 \le k \le r$, let $\Sigma_k \subset \tX_k$ be the corresponding set of special points.
\begin{itemize}
\item If $\big|\Sigma_k\big| = 2$, we may without loss of generality  assume that 
$\Sigma_k = \left\{\tilde{o}_k^+, \tilde{o}_k^-\right\}$.
\item Similarly, if $\big|\Sigma_k\big| = 3$, we assume that 
$\Sigma_k = \left\{\tilde{o}_k,\tilde{o}_k^+, \tilde{o}_k^-\right\}$.
\end{itemize}

\smallskip
\noindent
The following result is proved in  \cite{DrozdVoloshyn}.

\begin{theorem}\label{T:DrozdVoloshyn}
Let $\XX$ be a   non-commutative projective  curve.  Then $\VB(\XX)$ has tame  representation type if and only if the following conditions are satisfied.
\begin{itemize}
\item $\XX$ is Morita equivalent to a commutative elliptic curve, i.e.~$\tX$ is an elliptic curve, while $l$ and $ \approx$ are trivial.
\item $\XX$ is the rational non-commutative nodal curve attached to an admissible datum   $(\tX, l, \approx)$ such that $\tX = \tX_1 \sqcup \dots \sqcup \tX_r$ is  a disjoint union of $r$  projective lines, whereas  $(l, \approx)$ satisfies the following conditions:
\begin{itemize}
\item For any $1 \le k \le r$ we have: $\big|\Sigma_k\big| \le 3$.
\item If $\big|\Sigma_k\big| =  3$ then we additionally have:
the  set $\Pi_{\tx}$ contains special elements for precisely one point $\tx \in \Sigma_k$ (say, for $\tx = \tilde{o}$), whereas for the remaining two points of $\Pi_{\tx}$ (say, for $\tilde{o}^\pm$) we have:
$l(\tilde{o}^\pm) = 2)$.
\end{itemize}
\end{itemize}
\end{theorem}

\begin{definition}\label{D:TameNodal} Consider the pair $\bigl(\vec{p}, \vec{q}\bigr)$, where 
$$
\vec{p} = \bigl((p_1^+, p_1^-), \dots, (p_t^+, p_t^-)\bigr) \in \bigl(\NN^2\bigr)^t \quad \mbox{\rm and} \quad
\vec{q} =  (q_1, \dots, q_s) \in \NN^s
$$
for some $t, s\in  \NN_0$ (either of this tuples may be empty). Let 
$$
\tX := \tX_1 \sqcup \dots \sqcup  \tX_t \sqcup \tX_{t+1} \sqcup \dots \sqcup \tX_{t+s}
$$
be a disjoint union of $t+s$ projective lines. We define the weight function
$\tX \stackrel{l}\lar \NN$ by the following rules
\begin{itemize}
\item For each $1 \le k \le t$ we put: $l(\tilde{o}_k^\pm) = p_k^\pm$.
\item For each $1 \le k \le s$ we put: $l(\tilde{o}_{t+k}^\pm) = 2$ and 
$l(\tilde{o}_{t+k}) = q_k$.
\end{itemize}
Let $\approx$ be a relation on the set 
$
\bigl(\Pi_{\tilde{o}_1^+}\sqcup \Pi_{\tilde{o}_1^-}\bigr) \sqcup \dots
\sqcup \bigl(\Pi_{\tilde{o}_t^+}\sqcup \Pi_{\tilde{o}_t^-}\bigr) \sqcup
\bigl(\Pi_{\tilde{o}_{t+1}}\sqcup \dots \sqcup \Pi_{\tilde{o}_{t+s}}\bigr).
$
satisfying the conditions of Definition \ref{D:RelationNodalCurves}. If 
$\bigl(\vec{p}, \vec{q}, \approx\bigr)$ is admissible and $\wt$ is a compatible weight, we denote
by $\XX\bigl(\vec{p}, \vec{q}, \approx,\wt\bigr)$ the corresponding non-commutative
nodal rational projective curve. Since the weight $\wt$ does not imply the derived category, we often omit it and
write $\XX\bigl(\vec{p}, \vec{q}, \approx\bigr)$.
\end{definition}

\smallskip
\noindent
One can rephrase Theorem \ref{T:DrozdVoloshyn} in the following way.

\begin{theorem}\label{T:VBTame}
The category $\VB(\XX)$  of vector bundles on a non-commutative projective curve $\XX$ is representation tame if and only if $\XX$ is either a commutative elliptic curve
or a non-commutative nodal curve $\XX\bigl(\vec{p}, \vec{q}, \approx\bigr)$, where
$(\vec{p}, \vec{q}, \approx)$ is an admissible datum as in Definition \ref{D:TameNodal}.
\end{theorem}

\begin{remark} Let $\vec{p} = \bigl((2,2), (2,2)\bigr)$, $\vec{q}$ be void  and 
$\approx$ be given by the following rule: $(\tilde{o}_k^+, 1) \approx (\tilde{o}_k^-, 1)$ for $k = 1, 2$ and $(\tilde{o}_1^\pm, 2) \approx (\tilde{o}_2^\pm, 2)$. 
Then the central curve $X$ of the corresponding non-commutative nodal curve 
$\XX(\vec{p}, \approx)$ is given by the following Cartesian diagram:
$$
\xymatrix{
\Spec(\kk) \sqcup \Spec(\kk) \ar[rr]^-{\imath} \ar[d] & & E \sqcup E \ar[d]\\
\Spec(\kk) \ar[rr] & & X
}
$$
where $E \cong \overline{V(v^2-u^3-u^2)} \subset \PP^2$ is a plane nodal cubic and
the image of $\imath$ of the singular set of $E \sqcup E$. Note that the arithmetic genus of $X$ is \emph{two}.
In fact, the central curve of a tame non-commutative nodal curve can have arbitrary high arithmetic genus.
\end{remark}

\subsection{Remarks on stacky cycles of projective lines} In this subsection, let  $\kk$ be an algebraically closed field of characteristic zero. 

\begin{example}\label{Ex:FirstCyclicAction}
Let $E$ be a plane nodal cubic and
$\PP^1 \stackrel{\nu}\lar E$ be its normalization. Let us choose homogeneous 
coordinates $(z: w)$ on $\PP^1$ in such a way that $\nu^{-1}(s) = \bigl\{\tilde{o}^+,
 \tilde{o}^-\bigr\}$, where  $s$ is the singular point of $E$ and  $\tilde{o}^+ = (0: 1)$ and  $\tilde{o}^- = (1: 0)$. Consider the action of the cyclic group
 $G:= \left\langle \rho \,\big|\, \rho^n = e\right\rangle$ on $\PP^1$ given by the formula $(z:w) \stackrel{\rho}\mapsto (\zeta z: w)$, where $\zeta$ is a primitive 
 $n$-th root of unitity. It is clear that the action of $G$ on $\PP^1$ descends to an action of $G$ on $E$ such that $E':= E/G \cong E$. Let $\kA := \kO_{E'} \ast G$ be the sheaf on $E'$ defined by the following rule:
 $$
 U \mapsto \Gamma\bigl(\pi^{-1}(U), \kO_E\bigr) \ast G \quad \mbox{\rm for any open } \; U \subseteq E' ,
 $$
 where $E \stackrel{\pi}\lar E'$ is the projection map.
 Then $\kA$ is a sheaf of nodal orders on the projective curve   $E'$. For any
 $x \in E' \setminus \{s\}$, the order $\kA_x$ is maximal, whereas $\widehat{\kA}_s$ is the nodal order given by \eqref{E:nodalquiver}. 
\end{example}

\smallskip
\noindent 
The following result is obvious.
 
\begin{lemma}
The category $\Coh^G(E)$ of $G$-equivariant sheaves on $E$ is equivalent to the category $\Coh(\mathbbm{E})$, where $\mathbbm{E} = (E', \kA)$.
\end{lemma}

\begin{remark}
The non-commutative nodal curve $\mathbbm{E}$ admits the following description. 
Consider the length function $\PP^1 \stackrel{l}\lar \NN$ given by the rule:
$$
l(\tilde{x}) = 
\left\{
\begin{array}{cl}
n & \mbox{\rm if} \;\;  \tilde{x} \in \{\tilde{o}^+, \tilde{o}^-\} \\
1 & \mbox{\rm otherwise}.
\end{array}
\right.
$$
We then define the relation $\approx$ on the set $\Pi$ setting $(\tilde{o}^+, k) \approx
(\tilde{o}^-, n-k)$ for $1 \le k \le n$, where we replace $0$ by $ n$. It is easy to see that the datum $(\PP^1, l, \approx)$ is admissible. Using Theorem \ref{T:MoritaUniqueness} one can conclude that $\EE$ and the non-commutative nodal curve corresponding to the datum $(\PP^1, l, \approx)$ are Morita equivalent.
\end{remark}

\begin{example}\label{Ex:NodalCubicOrbifold}
Again, let  $E = \overline{V(v^2-u^3-u^2)} \subset \PP^2$ be a plane nodal cubic.
Consider the involution $E \stackrel{\tau}\lar E$ given by the rule  $(u, v) \mapsto (u, -v)$. Let
$
E \stackrel{\pi}\lar E/G \cong \PP^1
$
be the projection map, where   $G = \langle \tau \rangle \cong \ZZ_2$. Next, let  $\kA := \kO_{E} \ast G$ and $\EE = (\PP^1, \kA)$
be the corresponding non-commutative nodal curve. Again, the categories
$\Coh^G(E)$ and $\Coh(\EE)$ are equivalent. Let us choose homogeneous coordinates
on $\PP^1$ in such a way that $\pi^{-1}(\tilde{o}^+)$ is the singular point of $E$ and 
$\pi^{-1}(\tilde{o}^-)$ is its infinite point. Consider the length function 
$\PP^1 \stackrel{l}\lar \NN$ given by the rule:
$$
l(\tilde{x}) = 
\left\{
\begin{array}{cl}
2 & \mbox{\rm if} \;\;  \tilde{x} =  \tilde{o}^- \\
1 & \mbox{\rm otherwise}.
\end{array}
\right.
$$
We define the relation $\approx$ on the set $\Pi$ by setting 
$(\tilde{o}^+, 1) \approx (\tilde{o}^+, 1)$. Obviously, the datum $(\PP^1, l, \approx)$ is admissible.   According to Lemma \ref{L:BasicNodalSkew}
and Theorem \ref{T:MoritaUniqueness}, the datum $(\PP^1, l, \approx)$ defines  a non-commutative nodal curve, which is Morita equivalent to $\EE$.
\end{example}

\begin{example}\label{Ex:StackyCycles} In this example, we give a description of stacky cycles of projective lines used in the paper of Lekili and Polishchuk \cite{LekiliPolishchuk} in the language of non-commutative nodal curves. Let $r \in \NN$, $\vec{n} = (n_1, \dots, n_r) \in \NN^r$ and $\vec{c} = (c_1, \dots, c_r) \in \NN^r$ be such that $\mathsf{gcd}(n_k, c_k) = 1$ for any $1 \le k \le r$. Let $E_r$ be 
a cycle of $r$ projective lines and $\tX \stackrel{\pi}\lar E_r$ be its normalization. Then
$\tX = \tX_1 \sqcup \dots \sqcup \tX_r$ is  a disjoint union of $r$  
projective lines. Let $\bigl\{o_1, \dots, o_r\bigr\}$ be the set of singular points 
of $E_r$, where we choose their labeling in such a way that $\pi^{-1}(o_k) = 
\bigl\{\tilde{o}_k^-, \tilde{o}_{k+1}^+\bigr\}$, where $\tilde{o}_k^- = (1: 0) \in \tX_k$ and
$\tilde{o}_{k+1}^+ = (0: 1) \in \tX_{k+1}$. The completion of the local ring of $E_r$ at each  point $o_k$ is isomorphic to the commutative nodal ring $\kk\llbracket u, v\rrbracket/(uv)$. For any $1 \le k \le n$, consider the action of the cyclic group
$G_{n_k} = \bigl\langle \rho \,\big|\, \rho^{n_k} = e \bigr\rangle$ on $D = \kk\llbracket u, v\rrbracket/(uv)$ given by the rule
\begin{equation}\label{E:action}
\left\{
\begin{array}{ccc}
\rho \circ u & = & \zeta_k u \\
\rho \circ v & = & \zeta_k^{c_k} v,
\end{array}
\right.
\end{equation}
where $\zeta_k$ is some primitive $n_k$-th root of $1$. \emph{Heuristically}, the  category
of coherent sheaves $\Coh(\EE)$ on a stacky cycle
of projective lines $\EE:= \EE_r(\vec{n}, \vec{c})$ is an abelian category satisfying the following property: the category $\Tor(\EE)$  of finite length objects of 
$\Coh(\EE)$ 
splits into a direct sum of blocks:
$$
\Tor(\EE) \cong \biguplus\limits_{x \in E_r} \Tor_x(\EE),
$$
where $\Tor_x(\EE)$ is equivalent to the category of finite length $\kk\llbracket w\rrbracket$-modules if $x \in E_r$ smooth and to the category 
of finite length $D \ast G_{n_k}$-modules if $x = o_k$, where the action of 
$G_{n_k}$ is given by the rule \eqref{E:action}.

\smallskip
\noindent
Informally speaking,  a  stacky cycle of projective lines $\EE_r(\vec{n}, \vec{c})$ can be thought   as an appropriate cyclic gluing of weighted projective lines $\PP^1(n_1, n_2), \dots, \PP^1(n_{r-1}, n_r),  \PP^1(n_r, n_1)$.
Now let us proceed with  a formal definition of  $\EE_r(\vec{n}, \vec{c})$ viewed as a non-commutative nodal curve. As above, let $\tX$ be a disjoint union of $r$ projective lines. 
Consider the length function 
$\tX \stackrel{l}\lar \NN$ given by the rule:
$$
l(\tilde{x}) = 
\left\{
\begin{array}{cl}
n_k & \mbox{\rm if} \;\;  \tilde{x} \in  \left\{\tilde{o}_k^-, \tilde{o}_{k+1}^+ \right\} \\
1 & \mbox{\rm otherwise}.
\end{array}
\right.
$$
Let $\Pi_k^\pm = \Pi_{\tilde{o}_k^\pm}$. It is convenient to  use the identification
$$
\Pi_k^- = \left\{\bar{1}, \dots, \bar{n}_k\right\} = \Pi_{k+1}^+,
$$
given by the rule: $(\tilde{o}_k^-, j) = \bar{j} = (\tilde{o}_{k+1}^+, j)$ for 
$1 \le j \le n_k$. We have a bijection
$$
\Pi_k^- \stackrel{\tau_{k}}\lar \Pi_{k+1}^+, \; \bar{j} \mapsto \overline{c_k\cdot j}.
$$
Let $\approx$ be a relation on the set $\Pi$, given by the rule: 
$(\tilde{o}_k^-, j) \approx \bigl(\tilde{o}_{k+1}^+, \tau_k(j)\bigr)$.
Note that the set $\Pi$ does not contain reflexive elements, hence $\Pi^\ddagger = \Pi$ in this case. 

\smallskip
\noindent
Next, we claim that 
 the datum
$(\tX, l, \approx)$ is admissible. Indeed, let $n  := 
\mathsf{lcm}(n_1,  \dots, n_r)$ be the least common multiple of $n_1, \dots, n_r$. Then
we can define a compatible weight function $\Pi \stackrel{\wt}\lar \NN$ by the following rules:
$$
\wt(\tx, j) = 
\begin{cases}
 \dfrac{n}{n_{k-1}} & \text{if }\,  \tilde{x} =  \tilde{o}_k^+\, \text{ and }\,  1 \le j \le n_{k-1}, \\[10pt] 
 \dfrac{n}{n_{k}} & \text{if }\, \tilde{x} =  \tilde{o}_k^-\, \text{ and }\,  1 \le j \le n_{k},\\
 n & \text{otherwise}.
\end{cases}
$$
Then the stacky cycle of projective lines $\EE_r(\vec{n}, \vec{c})$   can be defined as the non-commutative nodal curve corresponding to the datum $(\tX, l, \approx, \mathsf{wt})$. 
The  central curve of $\EE_r(\vec{n}, \vec{c})$ is the usual cycle  of $r$ projective lines $E_r$, i.e. $\EE_r(\vec{n}, \vec{c}) = (E_r, \kA)$, where $\kA$ is an appropriate sheaf of nodal orders. 

\smallskip
\noindent
Let $\widetilde{X} \stackrel{\nu}\lar E_r$ be the normalization morphism. 
It is clear that the locus $\mathfrak{S}\bigl(\EE_r(\vec{n}, \vec{c})\bigr)$ of non-regular points of $\EE_r(\vec{n}, \vec{c})$ is just the set $\bigl\{o_1, \dots, o_r\bigr\}$ of the singular points
of $E_r$, where $\nu^{-1}(o_k) := 
\left\{\tilde{o}_k^{-}, \tilde{o}_{k+1}^{+} \right\}$. Moreover, for any $1 \le k \le r$ the order  
$
\widehat{A}_{o_k} 
$
is Morita equivalent to the basic nodal order $A_{(n_k, c_k)}$; see Remark \ref{R:nodalcyclic} for the corresponding notation. 

\smallskip
\noindent
As a ringed space, $\EE_r(\vec{n}, \vec{c})$ depends on the choice of the weight
function $\wt$. However, 
Theorem \ref{T:NodalMorita} assures that the corresponding category of coherent sheaves $\Coh\bigl(\EE_r(\vec{n}, \vec{c})\bigr)$ does not depend on this choice.

\smallskip
\noindent
One can define \emph{stacky chains} of projective lines (which appeared in \cite{LekiliPolishchuk}) in a similar way.
\end{example}

\begin{remark}
The following  example shows that a heuristic treatment of the notion of a stacky cycle of projective lines has to be performed  with a special care. Let $n \in \NN$ and
$1 \le c, c', d, d < n$ are such that $c, c', d, d'$ are all pairwise distinct, mutually prime with $n$, $c c' \equiv 1 \, \mod \, n$ and 
$d d' \equiv 1 \, \mod \, n$. Consider the corresponding
non-commutative nodal curves $\EE := \EE\bigl((n, n), (c, d)\bigr)$ and $\EE' := \EE\bigl((n, n), (c', d)\bigr)$. Then $\EE = (E, \kA)$ and $\EE' = (E, \kA')$ are tame non-commutative nodal curves, whose underlying central curve $E$ is a cycle of two projective lines. Next, $\mathfrak{S}(\EE) = \mathfrak{S}(\EE') = \left\{o_1, o_2\right\}$, where $o_1, o_2$ are the singular points of $E$. We have:
$$
\widehat{\kA}_{o_1} \sim A_{(n, c)}, \, \widehat{\kA}'_{o_1} \sim A_{(n, c')} \quad
\mbox{\rm and} \quad \widehat{\kA}_{o_2} \sim A_{(n, d)} \sim \widehat{\kA}'_{o_2},
$$
where $\sim$ denotes central Morita equivalence of the corresponding module categories. Note that $A_{(n, c)}$ and $A_{(n, c')}$ are isomorphic as rings.
Summing up, $\EE$ and $\EE'$ are two non-commutative nodal curves with the same central curve $E$ and such that for any $x \in E$, the corresponding categories
$\Tor_x(\EE)$ and $\Tor_x(\EE')$ are equivalent. Nonetheless, we claim that the categories 
$\Coh(\EE)$ and $\Coh(\EE')$ are \emph{not} equivalent. Indeed, any equivalence
$$\Coh(\EE) \stackrel{\Phi}\lar \Coh(\EE')$$ induces an automorphism of the central curve  $E \stackrel{\Phi_c}\lar E$; see \cite[Theorem 4.4]{BurbanDrozdMorita}. It follows from the assumptions on $d, d', c, c'$  that the orders $\widehat{\kA}_{o_1}$
and $\widehat{\kA}'_{o_2}$ are not Morita equivalent. Hence, 
$\Phi_c(o_1) = o_1$ and $\Phi_c(o_2) = o_2$. In particular, we get (as restrictions of $\Phi$) equivalences of categories
$$
\widehat{\kA}_{o_1}-\mathsf{fdmod} \stackrel{\Phi_1}\lar \widehat{\kA}'_{o_1}-\mathsf{fdmod} \quad \mbox{\rm and} \quad \widehat{\kA}_{o_2}-\mathsf{fdmod} \stackrel{\Phi_2}\lar \widehat{\kA}'_{o_2}-\mathsf{fdmod}.
$$
For $k = 1, 2$, let  $\widehat{\kO}_{o_k} \stackrel{\varphi_k}\lar \widehat{\kO}_{o_k}$ be the  induced automorphisms of the corresponding centers. It follows that 
$\varphi_1$ swaps the irreducible components of $\Spec\bigl(\widehat{\kO}_{o_1}\bigr)$, whereas $\varphi_2$ preserves the irreducible components of 
$\Spec\bigl(\widehat{\kO}_{o_2}\bigr)$. Hence, $\varphi_1$ and $\varphi_2$ can not be automorphisms, induced by a common automorphism $\Phi_c$ of the curve $E$. We get a contradiction, showing that the categories $\Coh(\EE)$ and $\Coh(\EE')$ are not equivalent.  \qed
\end{remark}

\subsection{Auslander curve of a non-commutative nodal curve}
\begin{definition}
Let $\XX = (X, \kA)$ be a non-commutative nodal curve and $\kH$ be the hereditary cover of $\kA$. The \emph{conductor ideal sheaf} $\kC$ is defined
as follows: for any open subset $U \subseteq X$ we put:
\begin{equation}\label{E:conductorsheaf}
\Gamma(U, \kC):= \Bigl\{f \in \Gamma(U, \kA) \,\big|\; fg \in \Gamma(U, \kA)\; \mbox{\rm for all}\; g  \in \Gamma(U, \kH)\Bigr\}.
\end{equation}
\end{definition}

\begin{lemma}
The following results are true.
\begin{itemize}
\item The canonical morphism  of $\kO_X$-modules $\kC \to \sHom_{\kA}(\kH, \kA)$ given on the level of local sections by the rule \eqref{E:conductorlocal} is an isomorphism.
    \item $\kC$ is a sheaf of two-sided ideals both in $\kA$ and $\kH$.
    \item The $\kk$-algebras $\bar{A}:= \Gamma(X, \kA/\kC)$ and $\bar{H}:= \Gamma(X, \kH/\kC)$ are finite dimensional and semisimple.
\end{itemize}
\end{lemma}
\begin{proof}
All statements follow directly from the corresponding local statements; see Proposition \ref{P:ConductorNodalOrder} and Corollary \ref{C:ConductorNodalOrder}.
\end{proof}

\begin{definition}
For a nodal curve $\XX = (X, \kA)$ as above, we call the sheaf of orders
\begin{equation}\label{E:AuslanderSheaf}
\kB  :=
\left(
\begin{array}{cc}
\kA & \kH \\
\kC & \kH
\end{array}
\right)
\end{equation}
the \emph{Auslander sheaf} of $\kA$. The corresponding  non-commutative nodal curve  $\YY:= (X, \kB)$ is called the  \emph{Auslander curve} of $\XX$.
\end{definition}

\smallskip
\noindent
Consider the following idempotent sections
$e:=
\left(
\begin{array}{cc}
1 & 0 \\
0 & 0
\end{array}
\right) \; \mbox{\rm and} \; f:=
\left(
\begin{array}{cc}
0 & 0 \\
0 & 1
\end{array}
\right)
\in \Gamma(X, \kB)$. Then  we get the following sheaves of locally projective left $\kB$-modules:
\begin{equation}
\kP := \kB e \cong
\left(
\begin{array}{c}
\kA \\
\kC
\end{array}
\right) \; \mbox{\rm and} \; \kQ := \kB f \cong
\left(
\begin{array}{c}
\kH \\
\kH
\end{array}
\right).
\end{equation}
 We denote
$
\kP^\vee := \sHom_\kB(\kP, \kB) \cong e\kB$ and $\kQ^\vee := \sHom_\kB(\kQ, \kB) \cong  f\kB$.

\smallskip
\noindent
Similarly to  the  local case, we introduce  the following functors.
\begin{itemize}
\item An exact functor $\sG := \Hom_\kB(\kP, \,-\,) \simeq \kP^\vee \otimes_\kB \,-\,$ from  $\Qcoh(\YY)$ to $\Qcoh(\XX)$.
\item  Similarly, we have an exact functor  $\widetilde\sG = \Hom_\kB(\kQ, \,-\,) \simeq \kQ^\vee \otimes_\kB \,-\,$ from
$\Qcoh(\YY)$ to $\Qcoh(\widetilde\XX)$.
\item We have functors $\sF := \kP \otimes_\kA \,-\,$ and $\sH:= \sHom_\kA(\kP^\vee, \,-\,)$ from $\Qcoh(\XX)$ to $\Qcoh(\YY)$.
\item Similarly, we have functors  $\widetilde\sF := \kQ \otimes_\kH \,-\,$ and $\widetilde\sH:= \sHom_\kH(\kQ^\vee, \,-\,)$ from
$\Qcoh(\widetilde\XX)$ to $\Qcoh(\YY)$.
\item Let $\sD\sG, \sD\widetilde\sG$, $\sL\sF$, $\sL\widetilde\sF$, $\sR\sH$ and $\sR\widetilde\sH$ be the corresponding derived functors.
\end{itemize}

\begin{remark}
It is clear that all functors $\sG, \widetilde\sG$, $\sF$, $\widetilde\sF$, $\sH$ and $\widetilde\sH$ can be restricted to the corresponding subcategories of coherent sheaves. The functor $\widetilde\sF$ is exact and transforms locally projective $\kH$-modules to locally projective $\kB$-modules.
\end{remark}

\noindent
All statements and  proofs of
Theorem \ref{T:minorsderived} and Corollary \ref{C:minorsderived} can be generalized to the global setting in a straightforward way.
In particular, we have the following results.

\begin{theorem}\label{T:NonCommNodalCurves} Let $\XX = (X, \kA)$ be a non-commutative nodal curve as in Definition \ref{D:NonCommNodalCurve}, $\widetilde\XX = (X, \kH)$ be its hereditary cover and $\YY = (X, \kB)$ be the corresponding Auslander curve. Then the following results are true.
\begin{itemize}
\item We have: $\mathsf{gl.dim}\bigl(\Coh(\YY)\bigr) = 2$.
\item We have a recollement diagram
\begin{equation}\label{E:RecollementAuslanderCurve}
\xymatrix{D^b\bigl(\bar{A}\mathsf{-mod})\bigr) \ar[rr]|{\,\sI\,} && D^b\bigl(\Coh(\YY)\bigr) \ar@/^2ex/[ll]^{\sI^{!}} \ar@/_2ex/[ll]_{\sI^*}
 \ar[rr]|{\,\sD\widetilde\sG\,}
  && D^b\bigl(\Coh(\widetilde\XX)\bigr). \ar@/^2ex/[ll]^{\,\sR\widetilde\sH\,} \ar@/_2ex/[ll]_{\,\sL\widetilde\sF\,}
  }
\end{equation}
Here, the exact functor $\sI$ is determined by the rule $\sI(\bar{A}) = \kS$, where $\kS$ is given by the locally projective resolution
\begin{equation}\label{E:SheafS}
0 \lar
\left(
\begin{array}{c}
\kC \\
\kC
\end{array}
\right)
\lar \left(
\begin{array}{c}
\kA \\
\kC
\end{array}
\right)
\lar \kS \lar 0.
\end{equation}
In particular, we have a semi-orthogonal decomposition
\begin{equation}
D^b\bigl(\Coh(\YY)\bigr) = \left\langle \mathsf{Im}(\sI), \, \mathsf{Im}(\sL\widetilde\sF)\right\rangle = \left\langle
D^b\bigl(\bar{A}\mathsf{-mod}), \,
D^b\bigl(\Coh(\widetilde\XX)\bigr)
\right\rangle.
\end{equation}
\item
Moreover, we have the following commutative diagram of categories and functors:
\begin{equation}
\begin{array}{c}
\xymatrix{
D^b\bigl(\Coh(\widetilde\XX)\bigr)  \ar[rrd]_{\nu_*} \ar@{^{(}->}[rr]^-{\sL\widetilde\sF} & & D^b\bigl(\Coh(\YY)\bigr) \ar@{->>}[d]_-{\sD\sG} & & \ar@{_{(}->}[ll]_-{\sL\sF} \Perf(\XX) \ar@{_{(}->}[lld]^-{\sE} \\
 & & D^b\bigl(\Coh(\XX)\bigr) & &
 }
 \end{array}
\end{equation}
where $\Perf(\XX)$ is the perfect derived category of coherent sheaves on $\XX$, $\sE$ is the canonical inclusion functor,
$\sL\sF$ and $\sL\widetilde\sF$ are fully faithful, $\sD\sG$ is an appropriate localization functor
 and $\nu_*$
is the functor induced by the ``normalization map'' $\widetilde{\XX} \stackrel{\nu}\lar \XX$.
\end{itemize}
\end{theorem}

\section{Tilting on rational  non-commutative nodal projective curves}

\subsection{Tilting on projective hereditary curves}
We begin with a brief description of the standard tilting bundle on  a weighted projective line due to Geigle and Lenzing \cite{GeigleLenzing} reexpressed in the language of non-commutative hereditary curves. Let $\tX = \PP^1$ and  $\tX \stackrel{l}\lar \NN$ be any length function. Let us fix any weight function $\Pi \stackrel{\mathsf{wt}}\lar \NN$  compatible with $l$. As usual, we put $$\vec{p}(\tx):=
\bigl(\wt(\tx, 1),\dots, \wt(\tx, l(\tx))\bigr) \quad \mbox{\rm for any}\quad \tx \in \tX.$$
Let $m:= \big|\vec{p}(\tx)\big|$ for some (hence for any) point $\tx \in \tX$ and
$\widetilde{H}_{\tx} := H\bigl(O_{\tx}, \vec{p}(\tx)\bigr) \subseteq \Mat_{m}(O_{\tx})$
be the standard hereditary order, defined by
the vector $\vec{p}(\tx)$.  Next, let $\widetilde{Q}_{(\tx, 1)}, \dots, \widetilde{Q}_{(\tx, l(\tx))}$ be the standard indecomposable projective left
$\widetilde{H}_{\tx}$-modules, i.e.~ we have a direct sum decomposition
$$
\widetilde{H}_{\tx} \cong \widetilde{Q}_{(\tx, 1)}^{\oplus \wt(\tx, 1)} \oplus \dots \oplus \widetilde{Q}_{(\tx, l(\tx))}^{\oplus \wt(\tx, l(\tx))}.
$$
Here, we have:
$$
\widetilde{Q}_{(\tx, 1)} :=
\left(
\begin{array}{c}
O_{\tx}\\
\vdots \\
O_{\tx}
\end{array}
\right) \cong
\left(
\begin{array}{c}
\idm_{\tx}\\
\vdots \\
\idm_{\tx}
\end{array}
\right) =: \widetilde{Q}'_{(\tx, 1)}.
$$
Note that there is a chain of   embeddings of $H_{\tx}$-modules
$
\widetilde{Q}'_{(\tx, 1)} \subset \widetilde{Q}_{(\tx, l(\tx))} \subset \dots \subset \widetilde{Q}_{(\tx, 1)}.
$

\smallskip
\noindent
Let $\tH := \tH(l, \mathsf{wt})$ be the  sheaf of hereditary orders on $\tX$ defined by $(l, \wt)$ and $\widetilde\XX = (\tX, \tH)$ the corresponding
non-commutative hereditary curve. Recall that $\tH$ is a subsheaf of the sheaf of maximal orders $\sMat_{m}(\kO_{\tX})$  such that
$\tH_{\tx} = \widetilde{H}_{\tx}$ for any $\tx \in \tX$.
First of all, note that we have an exact fully faithful functor
$$
\Coh(\tX) \stackrel{\bar\FF}\lar \Coh(\widetilde\XX), \quad
\kE \mapsto
\left(
\begin{array}{c}
\kE\\
\vdots \\
\kE
\end{array}
\right),
$$
which transforms locally free sheaves on $\tX$ into locally projective $\tH$-modules. We put:
$$
\widetilde\kL := \bar{\FF}\bigl(\kO_{\tX}\bigr) =
\left(
\begin{array}{c}
\kO_{\tX}\\
\vdots \\
\kO_{\tX}
\end{array}
\right).
$$
Next, for any $\tx \in \tX$ such that $l(\tx) \ge 2$ and  $2 \le i \le l(\tx)$, we have  a locally  projective $\kH$-module
$\widetilde\kL_{(\tx, i)}$ uniquely determined  by the following properties:
\begin{itemize}
\item The sheaf $\widetilde\kL_{(\tx, i)}$ is a subsheaf of $\widetilde\kL$.
\item For any $\tilde{y} \in \tX$ we have:
$$
\left(\widetilde\kL_{(\tx, i)}\right)_{\tilde{y}} = \left\{
\begin{array}{ccl}
\widetilde{\kL}_{\tilde{y}} & \mbox{\rm if} & \tilde{y} \ne \tx \\
\widetilde{Q}_{(\tx, i)} & \mbox{\rm if} & \tilde{y} = \tx.
\end{array}
\right.
$$
\end{itemize}
For a convenience of notation, we also define for any point $\tx \in \tX$ the locally projective subsheaf $\widetilde\kL_{(\tx, 1)}$ of $\widetilde\kL$ defined by the following condition on the stalks:
$$
\left(\widetilde\kL_{(\tx, 1)}\right)_{\tilde{y}} = \left\{
\begin{array}{ccl}
\widetilde{\kL}_{\tilde{y}} & \mbox{\rm if} & \tilde{y} \ne \tx \\
\widetilde{Q}'_{(\tx, 1)} & \mbox{\rm if} & \tilde{y} = \tx.
\end{array}
\right.
$$
It is clear that
$$
\widetilde\kL_{(\tx, 1)} := \bar{\FF}\bigl(\kO_{\tX}(-\tx)\bigr) =
\left(
\begin{array}{c}
\kO_{\tX}(-\tx)\\
\vdots \\
\kO_{\tX}(-\tx)
\end{array}
\right).
$$
Note that $\widetilde\kL_{(\tx', 1)} \cong \widetilde\kL_{(\tx'', 1)}$ for any $\tx', \tx'' \in \tX$. Abusing the notation, we shall denote this sheaf by $\widetilde\kL(-1)$.  Next, let
\begin{equation}\label{E:SetPhi}
\Phi := \bigl\{\tx \in \tX \,\big|\, l(\tx) \ge 2 \bigr\} =:\bigl\{\tx_1, \dots, \tx_r\bigr\}
\end{equation}
be the set of weighted points of $\tX$.
Let us choose some homogeneous coordinates $(z: w)$ on $\tX$. Then we have
a pair of distinguished sections $z, w \in \Hom_{\tX}\bigl(\kO_{\tX}(-1), \kO_{\tX}\bigr)$ vanishing, respectively at the points
$\tilde{o}^+:= (0:1)$ and $\tilde{o}^- = (1:0)$.
Let  $\tx_i = (\lambda_i: \mu_i)$ and $l_i = l(\tx_i)$ for $1 \le i \le r$.

\smallskip
\noindent
The following theorem is a restatement  of the classical result of  Geigle and Lenzing; see \cite[Proposition 4.1]{GeigleLenzing}.

\begin{theorem} The locally projective $\kH$-module
\begin{equation}\label{E:standrdtilting}
\widetilde\kT:= \bigl(\widetilde\kL \oplus \widetilde\kL(-1)\bigr) \oplus
 \bigoplus\limits_{\tx \in \Phi} \bigoplus\limits_{i = 2}^{l(\tx)} \widetilde\kL_{(\tx, i)}
\end{equation}
is a tilting object in the derived category $D^b\bigl(\Coh(\widetilde\XX)\bigr)$ (called, in what follows, the \emph{standard tilting bundle} on
$\widetilde\XX$) and the corresponding algebra $\Gamma:= \bigl(\End_{\widetilde\XX}(\widetilde\kT)\bigr)^\circ$ is isomorphic to the 
Ringel canonical algebra 
$\Gamma=\Gamma\bigl((\tx_1, l_1), \dots, (\tx_r, l_r)\bigr)$ which is the path algebra of the following quiver%
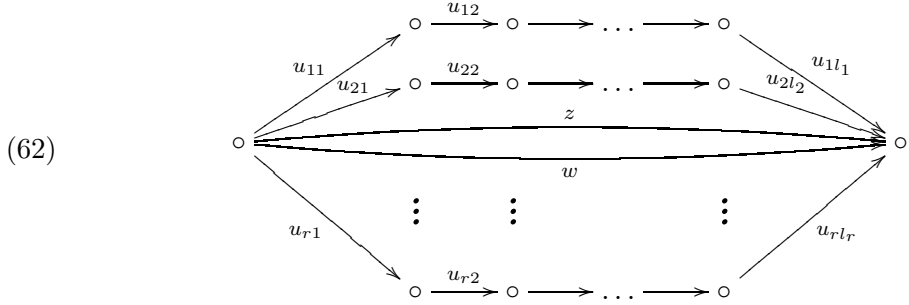
\footnote{\, In this picture the leftmost vertex corresponds to the sheaf $\kL$, the rightmost one corresponds to the sheaf $\kL(-1)$, while the
internal vertices of the $k$-th branch correspond to the sheaves $\widetilde\kL_{(\tx_k, i)}$ for a special point $\tx_k$ and $l_k=l(\tx_k)$.}
\def\bvd{\mbox{\huge$\boldsymbol{\vdots}$}}
\begin{equation}\label{E:Canonical}
\begin{array}{c}
 \xymatrix@R=1em{ && \circ \ar[r]^{u_{12}} & \circ\ar[r] &\dots \ar[r] & \circ \ar[rrdd]^{u_{1l_1}} \\
 				  &&  \circ \ar[r]^{u_{22}} & \circ\ar[r] &\dots \ar[r] & \circ \ar[rrd]^(.3){u_{2l_2}} \\
 				  \circ \ar[uurr]^{u_{11}} \ar[urr]^(.7){u_{21}} \ar[ddrr]_{u_{r1}} \ar@/^/[rrrrrrr]^z \ar@/_/[rrrrrrr]_w &&&&&&& \circ \\
 					&& \bvd & \bvd & & \bvd \\
 				  && \circ \ar[r]^{u_{r2}} & \circ\ar[r] &\dots \ar[r] & \circ \ar[rruu]_{u_{rl_r}}
 				  }
 				  \end{array}
\end{equation}
modulo the relations 
\begin{equation}\label{E:CanonicalRel}
u_{i l_i} \dots u_{i 1}=\lambda_{i} w - \mu_{i} z \quad  \mbox{\rm for}\quad 
 1 \le i \le r.
\end{equation} 
In other words, the derived functor
$$
\RHom_{\widetilde\XX}(\widetilde\kT,\,-\,): D^b\bigl(\Coh(\widetilde\XX)\bigr) \lar D^b(\Gamma\mathsf{-mod})
$$
is an equivalence of triangulated categories.

\end{theorem}

\begin{remark} In what follows, we shall call the
arrows $u_{ij}$ in the quiver \eqref{E:Canonical} for $1 \le i \le r, \,1 \le j \le l_i$ \emph{essential}, 
whereas the arrows $z$ and $w$ will be called \emph{redundant}. 
Note that the relations \eqref{E:Canonical} defining the canonical algebra $\Gamma\bigl((\tx_1, l_1), \dots, (\tx_r, l_r)\bigr)$ generate an admissible ideal if and only if the set $\Phi$ is empty (in this case, the canonical algebra is the path algebra of the Kronecker quiver). If $r \ge 2$ then both  redundant arrow can be excluded and if $r \ge 1$ then one of the redundant arrows can be excluded.

\smallskip
\noindent
Note also that we can formally add to the set $\Phi$ any point $\tx \in \tX$ of \emph{length one}, which correspond to a formal addition of another redundant arrow and does not change  the corresponding canonical algebra. We will use this procedure in the study of nodal curves.
\end{remark}

\begin{example}
Let $\lambda \in \kk \setminus \{0, 1\}$, $\Phi := \bigl\{(0:1), (1:0), (1:1), (\lambda:1)\bigr\}$ and $\l(\tx) = 2$ for any $\tx \in \Phi$. Then the corresponding 
canonical algebra \eqref{E:Canonical} is the tubular algebra \eqref{E:tubular} from  the introduction. 
\end{example}

\subsection{Tilting on non-commutative rational projective nodal curves}
We begin with an admissible datum  $(\tX, l, \approx)$, where  $\tX = \tX_1 \sqcup \dots \sqcup \tX_r$ is a smooth rational projective curve (each
$\tX_i\simeq\PP^1$). Let $\mathsf{wt}$ be any compatible weight function, $\XX = (X, \kA)$ be the corresponding non-commutative nodal curve,
$\widetilde\XX = (X, \kH) =  \widetilde\XX_1 \sqcup \dots \sqcup \widetilde\XX_r$ be its hereditary cover and $\YY = (X, \kB)$ be the corresponding Auslander curve.

\begin{theorem}\label{T:tilting1} Let $\widetilde\kT_i$ be the standard tilting bundle on $\widetilde\XX_i$ defined by \eqref{E:standrdtilting},
$\widetilde\kT:= \widetilde\kT_1 \oplus \dots \oplus \widetilde\kT_r$ and 
$
\kT := \widetilde\sF(\widetilde\kT) =
\left(
\begin{array}{c}
\widetilde\kT \\
\widetilde\kT \\
\end{array}
\right).
$
\smallskip
\noindent
Consider the complex  $\kX^{\bul} :=  \kT \oplus\kS[-1]$ in the derived category  $D^b\bigl(\Coh(\YY)\bigr)$, where $\kS$ is the torsion sheaf defined by the short exact sequence \eqref{E:SheafS}.  Then  $\kX^{\bul}$ is a tilting complex
in $D^b\bigl(\Coh(\YY)\bigr)$. In particular, if $\Lambda:= \bigl(\End_{D^b(\YY)}(\kX^{\bul})\bigr)^\circ$ then
the derived categories $D^b\bigl(\Coh(\YY)\bigr)$ and $D^b(\Lambda\mathsf{-mod})$ are equivalent.
\end{theorem}

\begin{proof}
The fact that $\kX^{\bul}$ generates the derived category $D^b\bigl(\Coh(\YY)\bigr)$ follows from the recollement diagram \eqref{E:RecollementAuslanderCurve} and the facts that $\widetilde\kT$ generates $D^b\bigl(\Coh(\widetilde\XX)\bigr)$ and
$\bar{A}$ generates $D^b(\bar{A}\mathsf{-mod})$. Since the functors $\sI$ and $\sL\widetilde\sF$ are fully faithful, we have:
$$
\Ext^i_{\YY}(\kS, \kS) = 0 = \Ext^i_{\YY}(\kT, \kT)
$$
for $i \ge 2$. Since the functor $\sL\widetilde\sF$ is left adjoint to $\sD\sG$ and $\sD\sG(\kS) =  0$, we have:
$$
\Ext^i_{\YY}(\kT, \kS) \cong \Hom_{D^b(\widetilde{\XX})}\bigl(\widetilde\kT, \sD\sG(\kS)[i]\bigr) = 0 \quad \mbox{\rm for all} \; i \in \ZZ.
$$
Finally, for any $i \in \ZZ$ we have:
$
\Ext^i_{\YY}(\kS, \kT) \cong \Gamma\bigl(X, \sExt^i_{\YY}(\kS, \kT)\bigr).
$
Since $\kS$ is torsion and $\kT$ is locally projective, we have: $\sHom_{\YY}(\kS, \kT) = 0$.  It follows from the exact sequence \eqref{E:SheafS} that $\sExt^{i}_{\YY}(\kS, \kT) = 0$ for $i \ge 2$. Therefore, $\Hom_{D^b(\YY)}\bigl(\kX^{\bul}, \kX^{\bul}[i]\bigr) = 0$
for $i \ne 0$. We have shown that $\kX^{\bul}$ is a tilting object in 
$D^b\bigl(\Coh(\YY)\bigr)$.  Let $\Lambda:=
 \bigl(\End_{D^b(\YY)}(\kX^{\bul})\bigr)^\circ$ then  the triangulated categories $D^b\bigl(\Coh(\YY)\bigr)$ and $D^b(\Lambda\mathsf{-mod})$ are equivalent; see \cite{Keller}.
 \end{proof}

\begin{remark}
Note that  we have: 
$
\Lambda 
  \cong
 \left(
 \begin{array}{cc}
  \Gamma & 0 \\
 W & \bar{A}
 \end{array}
 \right),
$
where $\Gamma = \bigl(\End_{\widetilde\XX}(\widetilde{\kT})\bigr)^\circ$, $\bar{A} = \bigl(\End_{\YY}(\kS)\bigr)^\circ$ and
$W := \Gamma\bigl(X, \sExt^1_{\YY}(\kS, \kT)\bigr)$ (viewed as an $\bar{A}$-$\Gamma$-bimodule). 
\end{remark}

\begin{corollary}\label{C:corollaryMain}
Let $(\tX, l, \approx)$ be an admissible datum, where $\tX$ is a disjoint union of projective lines. Then we have the following commutative diagram of categories and functors:
\begin{equation}\label{E:FunctorsII}
\begin{array}{c}
\xymatrix{
                                    & D^b\bigl(\Coh(\widetilde{\XX})\bigr) \ar[ld]_{\nu_*}  \ar@{^{(}->}[d]^-{\sL\widetilde\sF} &                          \\
D^b\bigl(\Coh(\XX)\bigr) & D^b\bigl(\Coh(\YY)\bigr) \ar@{->>}[l]_-{\sD\sG}  \ar[r]^-{\sT}                  & D^b(\Lambda\mathsf{-mod}) \\
                         & \Perf(\XX) \ar@{_{(}->}[u]_-{\sL\sF}  \ar@{_{(}->}[lu]^-{\sE} &
}
\end{array}
\end{equation}
in which $\sT$ is an  exact equivalence of triangulated categories, $\sL\sF$ and $\sL\widetilde\sF$ are fully faithful exact functors, $\sE$ is the canonical inclusion,
$\sD\sG$ is an appropriate   Verdier localization functor
and $\nu_\ast$ is induced by the forgetful functor $\Coh(\widetilde\XX) \lar \Coh(\XX)$ (normalization).
\end{corollary}

\subsection{Tame non-commutative  nodal curves and tilting}
We are especially interested in studying those finite dimensional $\kk$-algebras $\Lambda$ arising in the  diagram \eqref{E:FunctorsII} for which the derived category 
$D^b(\Lambda\mathsf{-mod})$ has tame representation type. Since $D^b(\Lambda\mathsf{-mod})$ contains the category of vector bundles $\VB(\XX)$ as a full subcategory, the non-commutative nodal curve $\XX$ has to be vector bundle tame, i.e.~of the form
$\XX(\vec{p}, \vec{q}, \approx)$, where $(\vec{p}, \vec{q}, \approx)$ is a datum from
Definition \ref{D:TameNodal}; see  Theorem \ref{T:VBTame}. 

\smallskip
\noindent
In this subsection
we are going  to elaborate one step further  an explicit description of the corresponding algebras $\Lambda(\vec{p}, \vec{q}, \approx)$.

\begin{definition}\label{D:DerivedTameAlgebras}
Let
us start with a pair of tuples
$$
\vec{p} = \bigl((p_1^+, p_1^-), \dots, (p_r^+, p_r^-)\bigr) \in \bigl(\NN^2\bigr)^r \quad \mbox{\rm and} \quad
\vec{q} =  (q_1, \dots, q_s) \in \NN^s,
$$
where  $r, s\in  \NN_0$ (either of this tuples is allowed to be empty).

\smallskip
\noindent
For any $1 \le i \le r$, let 
$\Xi_i^\pm := \bigl\{x_{i,1}^\pm, \dots, x_{i, p_i^\pm}^\pm\bigr\}$
and
\begin{equation}\label{E:AffineDynkin}
\Gamma\bigl(p_i^+, p_i^-\bigr) =
\begin{array}{c}
\xymatrix{
                 & \circ \ar[r]^-{x^+_{i,2}} &  \circ \; \; \dots & \circ \ar[r] & \circ \ar[rd]^-{x^+_{i, p_i^+}} & \\
\circ \ar[ru]^-{x^+_{i,1}}  \ar[d]_-{x^-_{i,1}}  &                            &    &    &               & \circ \\
    \circ   \ar[r]_-{x^-_{i,2}}          & \circ \ar[r] &  \circ \; \; \dots & \circ \ar[r] & \circ \ar[r] & \circ \ar[u]_-{x^-_{i, p_i^-}}
}
\end{array}
\end{equation}

\smallskip
\noindent
Next, for any $1 \le j \le s$, let $\Xi_j^\circ:= \bigl\{w_{j,1}, \dots, w_{j,q_j}\bigr\}$ and 
\begin{equation}
\Gamma\bigl(2, 2, q_j) =
\begin{array}{c}
\xymatrix{& &  &\circ  \ar@/^13pt/[rrrd]^-{v_j^+} & & & \\
\circ \ar@/^13pt/[rrru]^-{u_j^+} \ar@/_13pt/[rrrd]_-{u_j^-} \ar[r]^-{w_{j,1}} & \circ \ar[r]^-{w_{j,2}} & & \dots & \circ \ar[r] & \circ 
\ar[r]^-{w_{j,q_j}} & \circ \\
& &  &\circ  \ar@/_13pt/[rrru]_-{v_j^-} & & & \\
}
\end{array}
\end{equation}
modulo the relation $v_j^+ u_j^+ + v_j^- u_j^- +  w_{j, q_j}  \dots  
w_{j,1} = 0$.

\smallskip
\noindent
Let $\approx$ be a symmetric  relation  on the set
$$
\Xi:= \bigl((\Xi_1^+ \cup \Xi_1^-) \cup \dots \cup (\Xi_r^+ \cup \Xi_r^-)\bigr) \cup \bigl(\Xi_1^\circ \cup \dots \cup  \Xi_s^\circ \bigr)
$$
 such that for any $\xi \in \Xi$ there exists at most one $\xi' \in \Xi$ such that $\xi \approx \xi'$.
Then the  datum $(\vec{p}, \vec{q}, \approx)$ defines a finite dimensional
$\kk$-algebra $\Lambda = \Lambda(\vec{p}, \vec{q}, \approx)$ which is obtained from the disjoint union of quivers with relatioins
$\Gamma\bigl(p_i^+, p_i^-\bigr)$ and $\Gamma\bigl(2, 2, q_j)$ by the following combinatorial procedure.
\begin{itemize}
\item For any pair of tied elements $\varrho' \approx \varrho''$ of $\Xi$, we add a new vertex and two arrows 
ending in it:
\begin{equation}\label{E:gluingarrows}
    \begin{array}{cc}
    \xymatrix{
    \circ \ar[rr]^-{\varrho'} & & \circ \ar[rd]^-{\vartheta'} &  & \\
    &  &           & \bullet &  \\
    \circ \ar[rr]^-{\varrho''} & &  \circ \ar[ru]_-{\vartheta''} &  & \\
    }
    \end{array}
\end{equation}
The new arrows satisfy the following zero relations: 
$
 \vartheta' \varrho'  = 0 = \vartheta'' \varrho''.
$
\item For each reflexive element $\varrho \in \Xi$, we add
two new vertices and two arrows ending in each new vertex:
\begin{equation}\label{E:blowuparrow}
\begin{array}{cc}
    \xymatrix{
   & & \bullet \\
\circ 
\ar[r]^-{\varrho} & \circ \ar[ru]^-{\vartheta_+} \ar[rd]_-{\vartheta_-} &   \\
& & \bullet \\ 
    }
\end{array}
\end{equation}
The new arrows satisfy the following zero relations: 
$
\vartheta_\pm  \varrho = 0.
$
\end{itemize} 
\end{definition}

\begin{remark}
In the case when $s = 0$ (i.e.~when the tuple $\vec{q}$ is void) the algebra $\Lambda$ is skew-gentle  \cite{GeissDelaPena}. If additionally $\xi \not\approx \xi$ for all $\xi \in \Xi$, then the algebra $\Lambda$ is gentle \cite{AssemSkowr}. We also refer to  \cite{BurbanDrozdGentle} for a survey of results on the derived categories of gentle and skew-gentle algebras.
\end{remark}

\begin{theorem} Let $\XX = \XX(\vec{p}, \vec{q}, \approx)$ be the non-commutative nodal curve attached to an admissible datum  $(\vec{p}, \vec{q}, \approx)$ from  Definition 
\ref{D:TameNodal},  $\YY$ be the Auslander curve of $\XX$ and 
$\Lambda = \Lambda(\vec{p}, \vec{q}, \approx)$ be the finite dimensional algebra from Definition
\ref{D:DerivedTameAlgebras}. Then the following results hold:
\begin{itemize}
\item The derived categories $D^b\bigl(\Coh(\YY)\bigr)$ and $D^b(\Lambda\mathsf{-mod})$ are equivalent.
\item Moreover, $D^b\bigl(\Coh(\YY)\bigr)$ and $D^b\bigl(\Coh(\XX)\bigr)$ have tame
representation type.
\end{itemize}
\end{theorem}

\begin{proof} According to Theorem \ref{T:tilting1}, there exists a tilting complex
$\kX^{\bul} := \kT \oplus \kS[-1]$ in the derived category $D^b\bigl(\Coh(\YY)\bigr)$ such that 
\begin{equation*}
\widetilde{\Lambda} :=
 \bigl(\End_{D^b(\YY)}(\kX^{\bul})\bigr)^\circ \cong
 \left(
 \begin{array}{cc}
 \Gamma & 0 \\
 W & \bar{A}
 \end{array}
 \right),
\end{equation*}
Note, that 
$
\Gamma \cong 
\bigl(\Gamma(p_1^+, p_1^-) \times \dots \times \Gamma(p_r^+, p_r^-)\bigr) \times
\bigl(\Gamma(2,2,q_1) \times \dots \times \Gamma(2,2,q_s)\bigr).
$
Next, $\bar{A}$ is a product of several copies of the semisimple algebras $\kk$ and $\kk\times \kk$.
Namely, each pair $\omega', \omega'' \in \Pi$ of tied elements gives a factor $\kk$,  whereas each  reflexive element $\omega \in \Pi$ gives a factor $\kk \times \kk$. Taking into account  the description of the space
$W = \Gamma\bigl(X, \sExt^1_{\YY}(\kS, \kT)\bigr)$ viewed as right $\Gamma$-module given by  Lemma \ref{L:ExtLOcalComp}, we can conclude that actually 
$\widetilde{\Lambda} = \Lambda$, giving the first statement. 

Since the derived category  $D^b(\Lambda\mathsf{-mod})$ is  representation tame (it can be deduced as in \cite{Toronto}), the derived category $D^b\bigl(\Coh(\YY)\bigr)$ is representation tame too. Since $D^b\bigl(\Coh(\XX)\bigr)$ can be obtained as a Verdier localization of $D^b\bigl(\Coh(\YY)\bigr)$ (see Theorem \ref{T:NonCommNodalCurves}), one can conclude that 
$D^b\bigl(\Coh(\XX)\bigr)$ is representation tame as well.\!%
\footnote{\,Another approach to establish  the representation tameness of $D^b\bigl(\Coh(\XX)\bigr)$ is given in \cite{DrozdVoloshynDerived}.}
\end{proof}

\section{Tilting exercises with some tame non-commutative nodal curves}

\noindent
In this section we are going to study in more details several  special cases of the setting of  Corollary \ref{C:corollaryMain}.

\subsection{Elementary modifications} We are going to introduce two ``elementary modifications'', which allow to replace the algebra $\Lambda = \Lambda(\vec{p}, \vec{q}, \approx)$ by a derived-equivalent algebra.

\begin{lemma}\label{L:ModificationOne} Any fragment  of $\Lambda$ of the form \eqref{E:gluingarrows} can be replaces by the
    fragment
$$
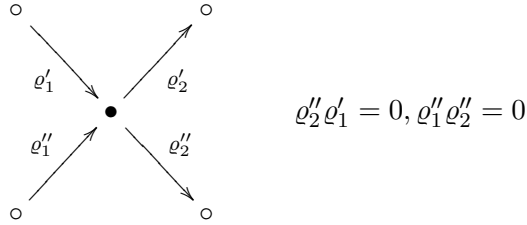
\xymatrix{
\circ \ar[rd]_-{\varrho'_1}  &  & \circ &  & \\
               & \bullet \ar[rd]^-{\varrho''_2} \ar[ru]_-{\varrho'_2} &  &   
\varrho''_2  \varrho'_1 = 0, \varrho''_1  \varrho''_2 = 0   \\           
\circ \ar[ru]^-{\varrho''_1} &  & \circ & & \\
}
$$
\end{lemma}
\begin{proof} Let $j$ be the common  target of the arrows $\vartheta'$ and $\vartheta''$, $i'$ be the source of $\vartheta'$ and $i''$ be the source of $\vartheta''$. Consider the complex 
$$
T_j := ( \dots \lar 0 \lar P_j \xrightarrow{(\vartheta', \vartheta'')} \underline{P_{i'} \oplus P_{i''}} \lar 0 \lar \dots),
$$
where the underlined term of  $T_\ast$ is located in the zero degree. Let $\Omega$ be the set of vertices of the quiver of the algebra $\Lambda$. Then 
$
T := T_j \oplus \bigl(\oplus_{i \in \Omega \setminus \{j\}} P_i)
$
is a tilting object of $D^b(\Lambda\mathsf{-mod})$. Let $\Gamma := 
\bigl(\End_{D^b(\Lambda)}(T)\bigr)^\circ$.
Then on the level of quivers and  relations we get precisely the transformation described in the statement of Lemma.
\end{proof}

\begin{example}
Let  $\Lambda$ be  the path algebra of the following quiver
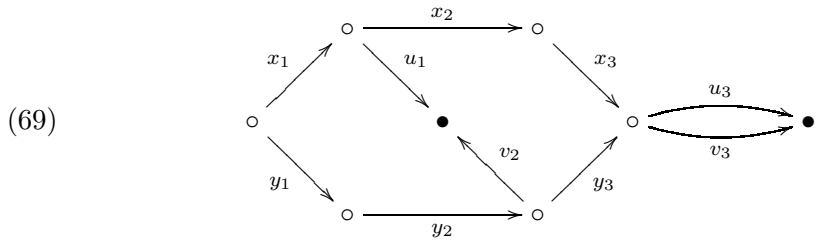
\begin{equation}
\begin{array}{c}
\xymatrix{
              & \circ \ar[rd]^-{u_1} \ar[rr]^-{x_2} &  & \circ \ar[rd]^-{x_3} & & \\
\circ \ar[ru]^-{x_1} \ar[rd]_{y_1} &  & \bullet        &          & \circ    \ar@/^/[rr]^-{u_3}  \ar@/_/[rr]_-{v_3}  & &  \bullet   &   \\
 & \circ  \ar[rr]_-{y_2} &  & \circ \ar[lu]_-{v_2}  \ar[ru]_-{y_3} & &
}
\end{array}
\end{equation}
modulo  the relations: $u_i x_i = 0$ for $i \in \{1, 3\}$ and
 $v_j y_j = 0$ for $j \in \{2, 3\}$.

\smallskip
\noindent
Making an elementary transformation at both bullets, we get a derived equivalent 
algebra $\Gamma$, which  is the path algebra of the following quiver
\begin{equation}
\begin{array}{c}
\xymatrix{
              & \circ  \ar[rr]^-{x_2} &  & \circ \ar[rd]^-{x_3^{(1)}} & & \\
\circ \ar[rr]^-{x_1^{(1)}} \ar[rd]_{y_1} &  & \bullet   \ar[lu]_-{x_1^{(2)}}  \ar[rd]_-{y_2^{(2)}}   &          & \circ    \ar@/^/[rr]^-{x_3^{(2)}}  \ar@/_/[rr]_-{y_3^{(2)}}  & &  \bullet \\
 & \circ  \ar[ru]_-{y_2^{(1)}} &  & \circ   \ar[ru]_-{y_3^{(1)}} & &
}
\end{array}
\end{equation}
modulo  the relations: $x_1^{(2)} y_2^{(1)} = y_{2}^{(2)} x_1^{(1)} = x_3^{(2)} y_3^{(1)} = y_{3}^{(2)} x_3^{(1)} = 0$. 
\end{example}

\begin{lemma}\label{L:ModificationTwo}
Any fragment  of $\Lambda$ of the form \eqref{E:blowuparrow} can be replaced by the
    fragment
$$
\xymatrix{
                      & \bullet \ar[rd]^-{\varrho^{(2)}_+} & & \\
\circ \ar[rd]_{\varrho^{(1)}_-} \ar[ru]^{\varrho^{(1)}_+} & & \circ & \varrho^{(2)}_+ \varrho^{(1)}_+ = \varrho^{(2)}_- \varrho^{(1)}_-\\
& \bullet \ar[ru]_{\varrho^{(2)}_-} & & \\
}
$$
\end{lemma}

\begin{proof} 
Let $j_\pm$ be the target of $\vartheta_\pm$ and $i$ be their common source.
Consider the complexes 
$$
T_{j_\pm} := ( \dots \lar 0 \lar P_{j_\pm} \xrightarrow{\vartheta_\pm} P_i \lar 0 \lar \dots).
$$
Again, let  $\Omega$ be the set of vertices of the quiver of the algebra $\Lambda$. Then 
$$
T := \bigl(T_{j_+} \oplus T_{j_-}\bigr) \oplus \bigl(\oplus_{i \in \Omega \setminus \{j_+, j_-\}} P_i)
$$
is a tilting object in  $D^b(\Lambda\mathsf{-mod})$. If $\Gamma := 
\bigl(\End_{D^b(\Lambda)}(T)\bigr)^\circ$,
then on the level of quivers and  relations the passage from $\Lambda$ to $\Gamma$ gives the desired elementary transformation.
\end{proof}

\begin{example}
Let
$\Lambda$  be the path algebra of the following quiver
\begin{equation}
\begin{array}{c}
\xymatrix{
              &  & \bullet & & \\
              & \circ  \ar[rr]^-{x_2} \ar[ru]^-{u_1}&  & \circ \ar[rd]^-{x_3} \ar[lu]_-{u_2}& & \\
\circ \ar[ru]^-{x_1} \ar[rrd]_{y_1} &  &        &          & \circ    \ar@/^/[rr]^-{u_3}  \ar@/_/[rr]_-{v_2}  & &  \bullet   &   \\
 & & \circ      \ar[rru]_-{y_2} \ar[rd]^-{v_1^+} \ar[ld]_-{v_1^-} & & \\
  & \bullet & &  \bullet &
}
\end{array}
\end{equation}
subject to the relations: $u_i x_i = 0$ for all $1 \le i \le 3$, $v_2 y_2 = 0$ and $v_1^\pm y_1 = 0$.

\smallskip
\noindent
Performing the elementary transformations at all bullets, we get a derived equivalent
algebra 
$\Gamma$, given as    the path algebra of the following quiver
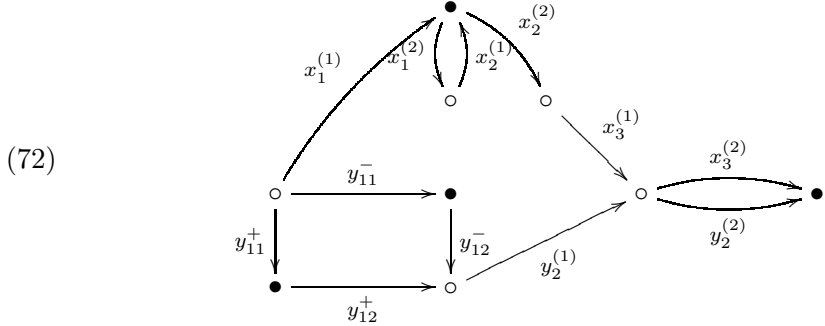
\begin{equation}
\begin{array}{c}
\xymatrix{
                                                                            &  & \bullet \ar@/_/[d]_-{x_1^{(2)}}  \ar@/^/[rd]^-{x_2^{(2)}} & & \\
                                       &  &   \circ \ar@/_/[u]_-{x_2^{(1)}} &   \circ \ar[rd]^-{x_3^{(1)}} & & \\
\circ \ar@/^/[rruu]^{x_1^{(1)}} \ar[d]_{y_{11}^+}  \ar[rr]^-{y_{11}^-}& & \bullet \ar[d]^-{y_{12}^-}  &        &           \circ    \ar@/^/[rr]^-{x_3^{(2)}}  \ar@/_/[rr]_-{y_2^{(2)}}  & &  \bullet   &   \\
\bullet \ar[rr]_-{y_{12}^+} & &\circ      \ar[rru]_-{y_2^{(1)}}   & &  & \\
}
\end{array}
\end{equation}
subject to the relations: $$x_2^{(2)} x_1^{(1)} = x_1^{(2)} x_2^{(1)} = x_3^{(2)} y_2^{(1)} = y_2^{(2)} x_3^{(1)} = 0 \quad \mbox{\rm and}\quad  y_{12}^+ y_{11}^+ = y_{12}^- y_{11}^-.$$
\end{example}

\subsection{Degenerate tubular algebra}\label{SS:DegenTub} Let $E = V\bigl(zy^2 - x^2(x-z)\bigr)  \subset \PP^2$ be a plane nodal cubic and  $G = \langle\tau \rangle \cong \ZZ_2$, where 
$E \stackrel{\tau}\lar E$ is the involution given by the rule  $(x: y: z) \mapsto 
(x: -y: z)$.  Then the  category
$\Coh^G(E)$ of $G$-equivariant coherent sheaves on $E$ is equivalent
to the category of coherent sheaves on the non-commutative nodal curve $\EE = 
\XX(\vec{p}, \vec{q} \approx)$ described
in Example \ref{Ex:NodalCubicOrbifold}.
Recall that the vector $\vec{p}$ is void,  $\vec{q} = (1)$ and $(\tilde{o}, 1) \approx
(\tilde{o}, 1)$.
 Then the corresponding algebra $\Lambda = \Lambda(\vec{p}, \vec{q}, \approx)$ is the path algebra of the following quiver
$$
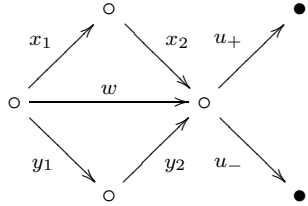
\xymatrix{
     & \circ \ar[rd]^-{x_2} & & \bullet  \\
\circ \ar[ru]^-{x_1} \ar[rr]^-{w} \ar[rd]_-{y_1} & & \circ \ar[ru]^-{u_+} \ar[rd]_{u_-} & \\
& \circ \ar[ru]_-{y_2} & & \bullet   \\
}
$$
modulo the relations $x_2 x_1 + y_2 y_1 + w = 0$ and $u_\pm w = 0$. Note that the corresponding ideal in the path algebra is not admissible and the arrow $w$ is redundant. Applying the elementary transformation from Lemma \ref{L:ModificationTwo}
to the arrow $w$, we end up with the the path algebra $T$ of the following quiver
 \begin{equation}\label{E:degtubular}
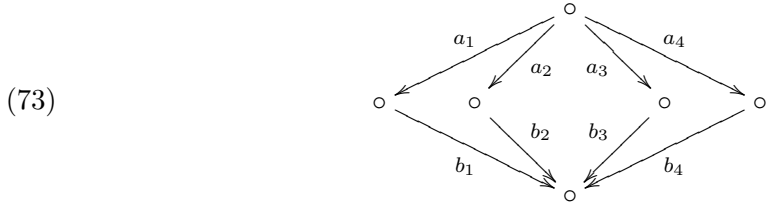
\begin{array}{c}
\xymatrix
{
        &           & \circ \ar[lld]_{a_1}
\ar[ld]^{a_2}  \ar[rd]_{a_3}  \ar[rrd]^{a_4}
      &         &       \\
\circ \ar[rrd]_{b_1}  & \circ \ar[rd]^{b_2}
&        & \circ \ar[ld]_{b_3}
& \circ \ar[lld]^{b_4}\\
        &           &\circ &         &       \\
}
\end{array}
\end{equation}
modulo the  relations $b_1 a_1 + b_2 a_2 +  b_3 a_3 = 0$ and  $b_1 a_1  = b_4 a_4$,
i.e.~the degenerate tubular algebra from Introduction. Since the derived categories 
$D^b(\Lambda\mathsf{-mod})$ and $D^b(T\mathsf{-mod})$ are equivalent, 
the commutative diagram of categories and functors \eqref{E:TiltingDegTub}
is a special case of the setting from Corollary \ref{C:corollaryMain}.

\smallskip
\noindent
Let $S$ be the path algebra of the following quiver
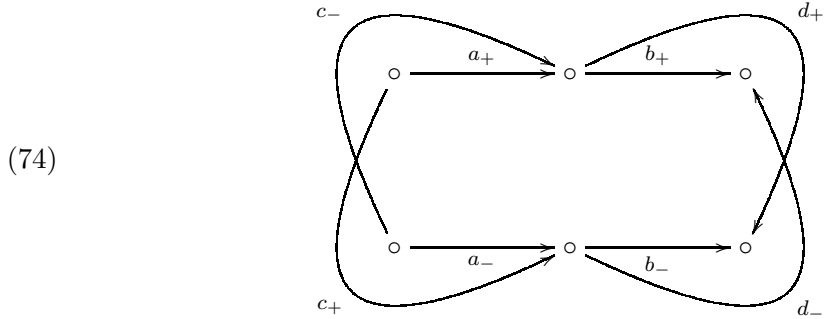
\begin{equation}\label{E:skewgentle}
\begin{array}{c}
\xymatrix{
\circ   \ar[rr]^-{a_+} \ar@/_70pt/[rrdd]_-{c_+}  & & \circ \ar[rr]^-{b_+} 
\ar@/^70pt/[rrdd]^-{d_+} & & \circ \\
& & & & \\
\circ   \ar[rr]_{a_-} \ar@/^70pt/[rruu]^-{c_-}  & & \circ \ar[rr]_{b_-} \ar@/_70pt/[rruu]_-{d_-} & & \circ
}
\end{array}
\end{equation}
modulo the following set of relations: 
$$
d_\pm a_\pm = b_\mp c_\pm \quad \mbox{\rm and}\quad b_\pm a_\pm = d_\mp a_\mp,
$$
i.e.~any two paths with the same source and target are equal.

\begin{proposition}\label{P:Geiss}
The derived categories $D^b(T\mathsf{-mod})$ and $D^b(S\mathsf{-mod})$ are equivalent.
\end{proposition}

\begin{proof}
Let $A$ be the path algebra of the affine Dynkin quiver
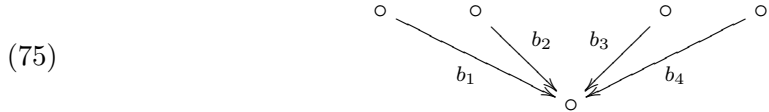
\begin{equation}\label{E:QuiverD4tilde}
\begin{array}{c}
\xymatrix
{
\circ \ar[rrd]_{b_1}  & \circ \ar[rd]^{b_2}
&        & \circ \ar[ld]_{b_3}
& \circ \ar[lld]^{b_4}\\
        &           &\circ &         &       \\
}
\end{array}
\end{equation}
and $M$ be the left $A$-module corresponding to the representation
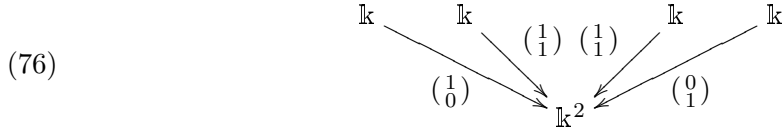
\begin{equation}\label{E:ModuleForOnePointExt}
\begin{array}{c}
\xymatrix
{
\kk \ar[rrd]_-{\left(\begin{smallmatrix} 1 \\0 \end{smallmatrix}\right)}  & \kk \ar[rd]^-{\left(\begin{smallmatrix} 1 \\1 \end{smallmatrix}\right)}
&        & \kk\ar[ld]_-{\left(\begin{smallmatrix} 1 \\1 \end{smallmatrix}\right)}
& \kk \ar[lld]^-{\left(\begin{smallmatrix} 0 \\1 \end{smallmatrix}\right)}\\
        &           &\kk^2 &         &       \\
}
\end{array}
\end{equation}
Then we have: $T \cong
\left(
\begin{array}{cc}
\kk & 0 \\
M & A
\end{array}
\right),
$
i.e.~$T$ is a so-called one-point extension of the algebra $A$ by the left $A$-module
$M$.
Next, consider the following left $A$-modules:
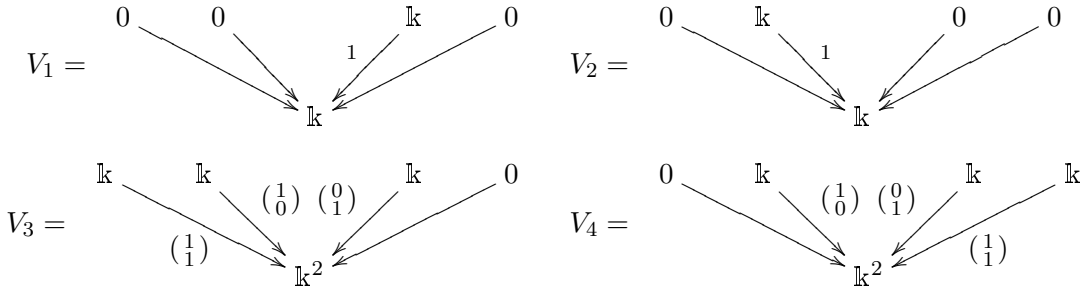
\begin{equation*}
V_1 = 
\begin{array}{c}
\xymatrix
{
0 \ar[rrd]  & 0 \ar[rd] &        & \kk \ar[ld]_-{1} & 0 \ar[lld]\\
        &           &\kk &         &       \\
}
\end{array}
\quad 
V_2 = 
\begin{array}{c}
\xymatrix
{
0 \ar[rrd]  & \kk \ar[rd]^-{1} &        & 0 \ar[ld] & 0 \ar[lld]\\
        &           &\kk &         &       \\
}
\end{array}
\end{equation*}
\begin{equation*}
V_3 = 
\begin{array}{c}
\xymatrix
{
\kk \ar[rrd]_-{\left(\begin{smallmatrix} 1 \\1 \end{smallmatrix}\right)}  & \kk \ar[rd]^-{\left(\begin{smallmatrix} 1 \\0 \end{smallmatrix}\right)}
&        & \kk\ar[ld]_-{\left(\begin{smallmatrix} 0 \\1 \end{smallmatrix}\right)}
& 0 \ar[lld]\\
        &           &\kk^2 &         &       \\
}
\end{array}
\quad
V_4 = 
\begin{array}{c}
\xymatrix
{
0 \ar[rrd]  & \kk \ar[rd]^-{\left(\begin{smallmatrix} 1 \\0 \end{smallmatrix}\right)}
&        & \kk\ar[ld]_-{\left(\begin{smallmatrix} 0 \\1 \end{smallmatrix}\right)}
& \kk \ar[lld]^-{\left(\begin{smallmatrix} 1 \\1 \end{smallmatrix}\right)}\\
        &           &\kk^2 &         &       \\
}
\end{array}
\end{equation*}
and
$$
V_5 = 
\begin{array}{c}
\xymatrix
{
0 \ar[rrd]  & \kk \ar[rd]^-{1} &        & \kk \ar[ld]_-{1} & 0 \ar[lld]\\
        &           &\kk &         &       \\
}
\end{array}
$$
Then $V := V_1 \oplus V_2 \oplus V_3 \oplus V_4 \oplus V_5$ is a tilting module
in $D^b(A\mathsf{-mod})$ and 
$$
\sT:= \mathsf{RHom}_A(T, \,-\,): D^b(A\mathsf{-mod}) \lar D^b(B\mathsf{-mod})
$$
is an equivalence of triangulated categories, where 
$B:= \bigl(\End_A(V)\bigr)^\circ$. Note that $B$ is isomophic to the path algebra
of the following quiver
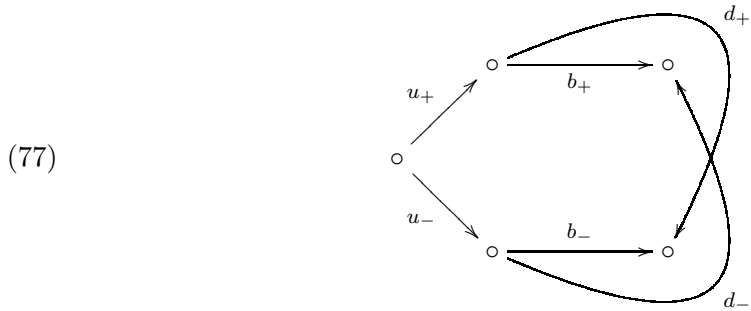
\begin{equation}\label{E:skewgentle2}
\begin{array}{c}
\xymatrix{
   & \circ \ar[rr]_-{b_+} \ar@/^70pt/[rrdd]^-{d_+} & & \circ \\
\circ \ar[ru]^-{u_+} \ar[rd]_{u_-}  & & & \\
   & \circ \ar[rr]^-{b_-} \ar@/_70pt/[rruu]_-{d_-} & & \circ
}
\end{array}
\end{equation}
modulo the relations: $b_\pm u_\pm = d_\mp u_\mp$. It is not difficult to see that 
 $\sT(M) \cong N$, where
$N$ is the following representation of the quiver \eqref{E:skewgentle2}:
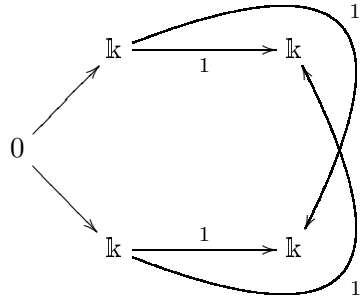
\begin{equation*}
\begin{array}{c}
\xymatrix{
   & \kk \ar[rr]_-{1} \ar@/^70pt/[rrdd]^-{1} & & \kk \\
0 \ar[ru] \ar[rd]  & & & \\
   & \kk \ar[rr]^-{1} \ar@/_70pt/[rruu]_-{1} & & \kk
}
\end{array}
\end{equation*}
Observe  that $S \cong
\left(
\begin{array}{cc}
\kk & 0 \\
N & B
\end{array}
\right).
$
The statement follows now  a result of Barot and Lenzing 
\cite[Theorem 1]{BarotLenzing} on derived equivalences of one-point extensions.
\end{proof}

\begin{remark}\label{R:DegTub} According to \cite[Theorem 8.1.10 and Exercise 8.1]{KashiwaraSchapira},
the derived category $D^b(S\mathsf{-mod})$ is equivalent to the derived category of constructible sheaves $D^b_\Sigma\bigl(\mathsf{ConSh}(\SS^2)\bigr)$ on the two-dimensional real sphere $\SS^2$ with respect to the stratification
$\Sigma$ described in the following picture:
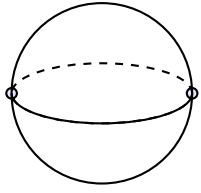
\newcommand{\sfrm}[3]{
\node[draw,solid, thick, fit=(#1-1-1)(#1-#2-#3), inner sep=0pt]{};}
\begin{center}
\begin{tikzpicture}[scale=0.40,
    thick,
    dot/.style={fill=blue!10,circle,draw, inner sep=1pt, minimum size=4pt}]
\draw (0,0) node[dot](p1){} (6,0) node[dot](p2){};
\draw (3,0) circle (3cm);
\draw[dashed] (3,0) circle [ x radius=3cm, y radius=1cm];

\draw (0,0) arc (180:360:3 and 1);
\end{tikzpicture}
\end{center}
Putting together all results obtained in this subsection, we get the following commutative diagram of triangulated categories and exact functors:
\begin{equation}\label{E:TiltingDegTub2}
\begin{array}{c}
\xymatrix{D^b(T\mathsf{-mod})
  \ar[r]^-\sT  & D^b_\Sigma\bigl(\mathsf{ConSh}(\SS^2)\bigr)  \ar@{->>}[r]^-{\sP} &  D^b\bigl(\Coh^G(E)\bigr) \\
 & \Perf^G(E) \ar@{_{(}->}[lu]^-{\sE} \ar@{^{(}->}[ru]_-{\sI} &  
}
\end{array}
\end{equation}
where $\sI$ is the canonical inclusion functor, $\sE$ is a fully faithful functor, $\sT$ is an equivalence of categories  and $\sP$ is an appropriate localization functor. It would be quite interesting to give
an interpretation of this result in terms of the homological mirror symmetry in the spirit of the approach of \cite{NadlerZaslow}. 
\end{remark}

\subsection{A purely commutative application of non-commutative nodal curves}
Again, let  $E = V\bigl(zy^2 - x^2(x-z)\bigr)  \subset \PP^2$ be a plane nodal cubic.
Consider the action of the cyclic group $G \cong \ZZ_2$ from Example \ref{Ex:FirstCyclicAction}. As an application of the technique of non-commutative nodal curves, we give a direct proof of  the following known result; see 
\cite[Example 1.4]{SST}.

\begin{proposition}\label{P:Sibilla}
Let $\widetilde{E}$ be the cycle of two projective lines. Then the derived categories
$D^b\bigl(\Coh(\widetilde{E})\bigr)$ and $D^b\bigl(\Coh^G(E)\bigr)$ are equivalent.
\end{proposition}
\begin{proof} Let $\EE$ be the non-commutative nodal curve from Example 
\ref{Ex:FirstCyclicAction} (i.e.~the categories 
$\Coh(\EE)$ and $\Coh^G(E)$ are equivalent). Let $\YY$ (respectively, 
$\widetilde{\YY}$)
be the Auslander curve of 
$\EE$  (respectively, of  $\widetilde{E}$). 
Let $\mathsf{K} := \mathsf{Ker}\Bigl(D^b\bigl(\Coh(\XX)\bigr) \stackrel{\sP}\lar  D^b\bigl(\Coh(\EE)\bigr)\Bigr)$ and
$\widetilde{\mathsf{K}} := \mathsf{Ker}\Bigl(D^b\bigl(\Coh(\widetilde{\XX})\bigr) \stackrel{\tilde\sP}\lar  D^b\bigl(\Coh(\widetilde{E})\bigr)\Bigr)$
be the kernels of the corresponding localization functors from Corollary \ref{C:corollaryMain}.
We are going to construct an equivalence of triangulated categories 
$D^b\bigl(\Coh(\XX)\bigr) \stackrel{\sE}\lar D^b\bigl(\Coh(\widetilde{\XX})\bigr)$,
which induces  a commutative diagram of categories and functors
$$
\xymatrix{
\mathsf{K} \ar[rr]^-{\sE} \ar@{_{(}->}[d] & & \widetilde{\mathsf{K}} \ar@{^{(}->}[d]\\
D^b\bigl(\Coh(\YY)\bigr) \ar[rr]^-{\sE} \ar@{->>}[d]_-{\sP} & &  D^b\bigl(\Coh(\widetilde{\YY})\bigr) \ar@{->>}[d]^-{\tilde\sP}\\
D^b\bigl(\Coh(\EE)\bigr) \ar[rr]^-{\bar{\sE}} & &  D^b\bigl(\Coh(\widetilde{E})\bigr),
}
$$
where all horizontal arrows are equivalences of triangulated categories.

\smallskip
\noindent
Let $D^b\bigl(\Coh(\YY)\bigr) \stackrel{\sT}\lar D^b(\Lambda\mathsf{-mod})$ and
$D^b\bigl(\Coh(\widetilde{\YY})\bigr) \stackrel{\widetilde{\sT}}\lar D^b(\widetilde{\Lambda}\mathsf{-mod})$) be the equivalences of triangulated categories,
where the algebras $\Lambda$ and $\widetilde{\Lambda}$ are the algebras corresponding, respectively, to $\YY$
and to $\widetilde{\YY}$ as in Corollary \ref{C:corollaryMain}. Recall that
\begin{equation}\label{E:GentleA1}
\Lambda = 
\begin{array}{c}
\xymatrix{
 & 2_- \ar[d]^-{b_-} \ar@/^/[rd]^-{c_+} & & &\\
1 \ar@/^/[ru]^-{a_+} \ar@/_/[rd]_-{a_-} & 3& 4  \ar@/^/[r]^{d_+} \ar@/_/[r]_{d_-}  
& 5 & b_\pm a_\mp = 0 \; \; \mbox{\rm and}\;\;  d_\pm c_\mp = 0\\
& 2_- \ar[u]_{b_+} \ar@/_/[ru]_-{c_-} & & & 
}
\end{array}
\end{equation}
whereas 
\begin{equation}\label{E:GentleA2}
\widetilde{\Lambda} = 
\begin{array}{c}
\xymatrix
{
 & & 3_+ & &  & \\
1 \ar@/^/[r]^-{u_+} \ar@/_/[r]_-{u_-}
 & 2  \ar[ru]^-{v_+} \ar[rd]_-{v_-} & &
 4 \ar[lu]_-{w_{+}} \ar[ld]^-{w_{-}}
& 5 \ar@/_/[l]_{z_+} \ar@/^/[l]^{z_-} & v_\pm u_\mp = 0 \;\;  \mbox{\rm and}\;\;  w_\pm z_\mp = 0.\\
& & 3_- & & & 
}
\end{array}
\end{equation}
Consider the third gentle algebra
$$
\Gamma  = 
\begin{array}{c}
\xymatrix
{
 & & 3_+ \ar[rd]^-{c_+} & &  & \\
1 \ar@/^/[r]^-{u_+} \ar@/_/[r]_-{u_-}
 & 2  \ar[ru]^-{v_+} \ar[rd]_-{v_-} & &
 4 \ar@/^/[r]^{d_+} \ar@/_/[r]_{d_-}
& 5  & v_\pm u_\mp = 0 \;\;  \mbox{\rm and}\;\;  d_\pm c_\mp = 0.\\
& & 3_- \ar[ru]_-{c_-} & & & 
}
\end{array}
$$
We construct now a pair of equivalences of triangulated categories:
$$
D^b(\Lambda\mathsf{-mod}) \stackrel{\sT_1}\lar D^b(\Gamma\mathsf{-mod})
\stackrel{\sT_2}\longleftarrow D^b(\widetilde{\Lambda}\mathsf{-mod}).
$$
\begin{itemize}
\item The first equivalence $\sT_1$ is just the elementary modification from Lemma
\ref{L:ModificationOne}, applied to the third vertex.
\item The second equivalence $\sT_2$ is given by the tilting complex
$$X^\bu:= S_1[-2] \oplus S_2[-1] \oplus P_{3_+} \oplus  P_{3_-} \oplus P_4 \oplus P_5.
$$
\end{itemize}
The image of the localizing subcategory $\mathsf{K} \subset D^b\bigl(\Coh(\YY)\bigr)$
in $D^b(\Lambda\mathsf{-mod})$ under the tilting equivalence $\sT$  is the triangulated envelope 
$\bigl\langle X_+, X_-, Y_+, Y_-\bigr\rangle$, where
$$
X_\pm = \bigl(\dots \lar  0 \lar P_5 \stackrel{d_\pm}\lar P_4 \stackrel{c_\mp}\lar \underline{P_{2_\pm}} \lar 0 \lar \dots\bigr)
$$
and
$
Y_\pm = \bigl(\dots \lar  0 \lar P_3 \stackrel{b_\mp}\lar P_{2_\mp} \stackrel{a_\pm}\lar \underline{P_{1_\pm}} \lar 0 \lar \dots\bigr).
$
One can check that $\sT_1(X_\pm) \cong S_{3_\pm}$, whereas $\sT_1(Y_\pm) \cong Z_\pm$, where
$$
Z_\pm = \bigl(\dots \lar  0 \lar P_3 \stackrel{v_\mp}\lar P_{2_\mp} \stackrel{u_\pm}\lar \underline{P_{1_\pm}} \lar 0 \lar \dots\bigr).
$$ 
In an analogous way one can check that the image of the localizing subcategory
$\widetilde{\sK}$ under the chain of equivalences of derived categories
$$
D^b\bigl(\Coh(\widetilde{\YY})\bigr) \stackrel{\widetilde\sT}\lar D^b(\widetilde{\Lambda}\mathsf{-mod}) 
\stackrel{\sT_2}\lar D^b(\Gamma\mathsf{-mod})
$$
is again the triangulated category $\bigl\langle S_{3_+}, S_{3_-}, Z_+, Z_-\bigr\rangle$. It proves the proposition.
\end{proof}

\subsection{Tilting on Zhelobenko's non-commutative cycles of projective lines}\label{SS:ZhelobenkoCurves} For $n \in \NN$, let $E = E_{2n}$ be a cycle of $2n$ projective lines. It is convenient to label
the irreducible components of $E$  
by the natural numbers $\bigl\{1, 2, \dots, 2n\bigr\}$.  Let $\widetilde{E} \stackrel{\nu}\lar E$ be the normalization of $E$, $\widetilde{\kO}:= 
\nu_\ast\bigl(\kO_{\widetilde{E}}\bigr)$ and  $\kC := \mathit{Ann}_{E}\bigl(\widetilde{\kO}/\kO\bigr) \cong \mathit{Hom}_E\bigl(\widetilde{\kO}, \kO\bigr)$ 
be the corresponding conductor ideal sheaf. In this particular case, $\kC$ is just the ideal sheaf
of the set $\bigl\{o_1, o_2, \dots, o_{2n}\bigr\}$ of the singular points of $E$. Let
$\EE = (E, \kB)$ be the Auslander curve of $E$, where 
$
\kB = 
\left(
\begin{array}{cc}
\kO & \widetilde{\kO}\\
\kC & \widetilde{\kO}
\end{array}
\right).
$
Recall that for any $1 \le k \le 2n$ we have: $\widehat{\kB}_{o_k} \cong B$, where
$B$ is the Auslander order \eqref{E:AuslanderOrderNodal}. Let $\kS_k^\pm$ be the simple torsion sheaf on $\EE$ supported at the singular point $o_k$ which corresponds
to the vertex $\pm$ of the quiver \eqref{E:AuslanderQuiver}. Note that
\begin{equation}\label{E:ExtVanishing}
\Ext_{\EE}^p\bigl(\kS_k^\pm, \kS_k^\pm\bigr) \cong
\left\{
\begin{array}{cl}
\kk & \mbox{\rm if} \; p = 0\\
0 & \mbox{\rm otherwise}. 
\end{array}
\right.
\end{equation}
Let $e := \left(
\begin{array}{cc}
1 & 0\\
0 & 1'
\end{array}
\right) \in \Gamma(E, \kB),
$ where $1 = 1' + 1'' \in \Gamma(E, \widetilde{\kO})$ is the decomposition of the identity section corresponding  to the decomposition 
$\widetilde{\kO} = \widetilde{\kO}' \oplus \widetilde{\kO}''$ of $\widetilde{\kO}$ in the direct sum of ``even'' and ``odd'' components. Let
$$
\kA :=  \mathit{Hom}_{\kB}\bigl(\kB  e\bigr) \cong
\left(
\begin{array}{cc}
\kO & \widetilde{\kO}'\\
\kC & \widetilde{\kO}'
\end{array}
\right)
$$
 and $\mathbb{A}:= (E, \kA)$ be the corresponding non-commutative nodal curve. Note that for
 any $1 \le k \le 2n$ we have: $\widehat{\kA}_{o_k} \cong A$, where
$A$ is the Zhelobenko  order \eqref{E:ZhelobenkoOrder}, so we call
$\mathbb{A}$ \emph{Zhelobenko's non-commutative cycle of projective lines}. 
 Next, we have a splitting 
$\kC = \kC' \oplus \kC''$, where $\kC' \subset \widetilde{\kO}'$ (respectively, 
$\kC'' \subset \widetilde{\kO}''$). Under these notations, the  following sequences of sheaves are exact:
$
0 \lar \kC' \lar \kO \lar \widetilde{\kO}'' \lar 0$
and $0 \lar \kC'' \lar \kO \lar \widetilde{\kO}' \lar 0.
$
Next, let 
$$
\kI := \mathsf{Im}\Bigl(\kB e \otimes_{e\kB e} e\kB \xrightarrow{\mathsf{mult}} \kB\Bigr) = 
\left(
\begin{array}{cc}
\kO & \widetilde{\kO}' \oplus \widetilde{\kO}''\\
\kC & \widetilde{\kO}' \oplus \widetilde{\kC}''
\end{array}
\right)
$$
and $\bar{\kB}:= \kB/\kI$. Then we have:
$
\bar{O}'' := \Gamma\bigl(E, \widetilde{\kO}''/\kC''\bigr) \cong 
\Gamma(E, \bar{\kB}) \cong \underbrace{\kk \times \dots \times \kk}_{\mbox{$2n$  times}}.
$

\begin{proposition} We have a recollement diagram
\begin{equation}\label{E:RecollementAuslanderZhelobenko}
\xymatrix{D^b\bigl(\bar{O}''\mathsf{-mod})\bigr) \ar[rr] && D^b\bigl(\Coh(\EE)\bigr) \ar@/^2ex/[ll] \ar@/_2ex/[ll]
 \ar[rr]
  && D^b\bigl(\Coh(\mathbb{A})\bigr). \ar@/^2ex/[ll] \ar@/_2ex/[ll]
  }
\end{equation}
In particular, there exists  an equivalence of triangulated categories:
\begin{equation}
D^b\bigl(\Coh(\mathbb{A})\bigr) \lar \bigl\langle \kS_1^+, \kS_2^-, \dots,
\kS_{2n-1}^+, \kS_{2n}^-\bigr\rangle^\perp \subset \bigl(\Coh(\EE)\bigr).
\end{equation}
\end{proposition}
\begin{proof}
It is a consequence of the corresponding local statement (see Theorem \ref{T:minorsderived}) combined with the fact (following from \eqref{E:ExtVanishing}) that the functor 
$$D^b\bigl(\bar{O}''\mathsf{-mod})\bigr) \lar
D^b_{\bar{\kB}}\bigl(\Coh(\EE)\bigr) = 
\bigl\langle \kS_1^+, \kS_2^-, \dots,
\kS_{2n-1}^+, \kS_{2n}^-\bigr\rangle
$$ is an equivalence of triangulated categories.
\end{proof}

\begin{theorem}\label{T:ZhelobenkoTilting}
Let $\Upsilon = \Upsilon_n$ be the gentle algebra given by \eqref{E:GentleGlobDimThree}. Then the derived categories
$D^b\bigl(\Coh(\mathbb{A})\bigr)$ and $D^b(\Upsilon\mathsf{-mod})$ are equivalent.
\end{theorem}

\begin{proof}
Let $D^b\bigl(\Coh(\EE)\bigr) \stackrel{\sT}\lar D^b(\Lambda\mathsf{-mod})$ be
the tilting equivalence from Corollary \ref{C:corollaryMain}. Recall from \cite[Section 5.2]{bd} that $\Lambda = \Lambda_{2n}$ is the path algebra of the following quiver
$$
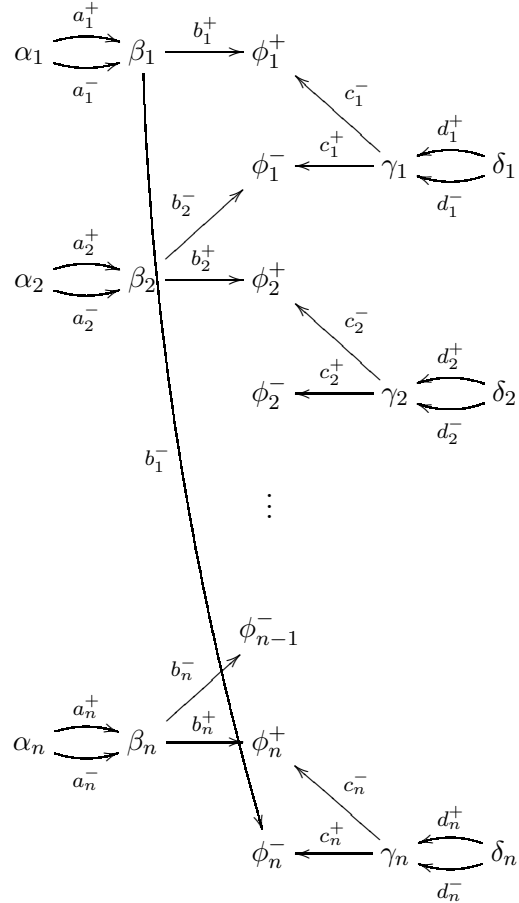
\xymatrix{
\alpha_1 \ar@/^/[r]^-{a_1^+} \ar@/_/[r]_-{a_1^-} & \beta_1 \ar@/_10pt/[rddddddd]_-{b_1^-} \ar[r]^{b_1^+}  &  \phi_1^+        &    & \\       
                                              &                  & \phi_1^- & \ar[l]_{c_1^+} \ar[lu]_-{c_1^-} \gamma_1  
                                              & \ar@/^/[l]^-{d_1^-} \ar@/_/[l]_-{d_1^+} \delta_1 \\
\alpha_2 \ar@/^/[r]^-{a_2^+} \ar@/_/[r]_-{a_2^-} & \beta_2 \ar[ru]^-{b_2^-} \ar[r]^-{b_2^+}  &  \phi_2^+        &    & \\       
                                              &                  & \phi_2^- & \ar[l]_{c_2^+} \ar[lu]_-{c_2^-} \gamma_2  
                                              & \ar@/^/[l]^-{d_2^-} \ar@/_/[l]_-{d_2^+} \delta_2 \\
                                              & & \vdots & & \\  
                                              & & \phi_{n-1}^-& & \\                                              
\alpha_n \ar@/^/[r]^-{a_n^+} \ar@/_/[r]_-{a_n^-} & \beta_n \ar[ru]^-{b_n^-} \ar[r]^-{b_n^+}  &  \phi_n^+        &    & \\       
                                              &                  & \phi_n^- & \ar[l]_{c_n^+} \ar[lu]_-{c_n^-} \gamma_n  
                                              & \ar@/^/[l]^-{d_n^-} \ar@/_/[l]_-{d_n^+} \delta_n \\
}
$$
modulo the relations $b_k^\pm a_k^\mp = 0 = c_k^\pm d_k^\mp$ for all $1 \le k \le n$. For any $1 \le k \le n$, consider the complexes
$$
A_k^\pm:= \bigl(\dots \lar 0 \lar P_{\phi_k^\pm}
\xrightarrow{c_k^\mp} P_{\gamma_k} \xrightarrow{d_k^\pm} 
\underline{P_{\delta_k}} \lar 0 \lar \dots \bigr).
$$
Then we have:
$$
\left\{
\begin{array}{ccc}
\sT\bigl(\kS_{2k-1}^+\bigr) & \cong & A_k^+ \\
\sT\bigl(\kS_{2k}^-\bigr) & \cong & A_k^-
\end{array}
\right.
\; \; 1\le k \le n.
$$
For any $1 \le k \le n$, consider the following objects of $D^b(\Lambda\mathsf{-mod})$:
$$
B_k = \bigl(\dots \lar 0 \lar P_{\phi_{k-1}^+} \oplus P_{\phi_{k}^-}
\xrightarrow{\left(b_k^- \; b_k^+\right)} P_{\beta_k}  \lar 0 \lar \dots \bigr),
$$

$$
C_k = \bigl(\dots \lar 0 \lar P_{\phi_{k-1}^+} \oplus P_{\phi_{k}^-}
\xrightarrow{\left(b_k^- a_k^- \; b_k^+ a_k^+ \right)} P_{\alpha_k}  \lar 0 \lar \dots \bigr)
$$
and
$$
H:= \bigoplus\limits_{k = 1}^n \bigl(P_{\gamma_k}\oplus P_{\delta_k} \oplus B_k \oplus C_k\bigr).
$$ 
It is not difficult to check that $\Hom_{D^b(\Lambda)}\bigl(H, H[p]\bigr) = 0$ and $H$ generates the triangulated category
$\bigl\langle A_1^+, A_1^-, \dots, A_n^+, A_n^-\bigr\rangle^\perp$. Hence $H$ is a tilting object of the latter category.
Moreover, one can show that $\bigl(\End_{D^b(\Lambda)}(H)\bigr)^\circ \cong \Upsilon$.
 Summing up, we
have a chain of equivalences of triangulated categories
$$
D^b\bigl(\Coh(\mathbb{A})\bigr) \lar \bigl\langle A_1^+, A_1^-, \dots, A_n^+, A_n^-\bigr\rangle^\perp \lar D^b(\Upsilon\mathsf{-mod}),
$$
which yields the desired statement.
\end{proof}

\smallskip
\noindent
For any $n \in \NN$, consider the \emph{graded} gentle algebra
$\Theta = \Theta_n$, given as the path algebra of the following quiver
\begin{equation}
\begin{array}{c}
\xymatrix{
\circ  \ar@/_15pt/[dd]_-{a_1}  \ar@/_3pt/[rrdd]_-{b_1} & & \circ  \ar@/_15pt/[dd]_-{a_2} \ar@/_3pt/[rrdd]_-{b_2} & & \dots  & \circ  \ar@/_15pt/[dd]_-{a_n}  \ar@/^40pt/[ddlllll]^-{b_n}\\
& & & & & \\
\circ \ar@{.>}[uu]_-{w_1} & & \circ \ar@{.>}[uu]_-{w_2} & &  \circ & \circ \ar@{.>}[uu]_-{w_n}
}
\end{array}
\end{equation}
modulo the relations $b_k w_k = 0 = w_k a_k$ for all $1 \le k \le n$, where the grading is given by the rule
$\mathsf{deg}(a_k) = \mathsf{deg}(b_k) = 0$, whereas $\mathsf{deg}(w_k) = 1$.

\begin{proposition}
The triangulated categories $D^b(\Upsilon\mathsf{-mod})$ and $D^b(\Theta)$ are equivalent, where $D^b(\Theta)$ denotes the derived category of $\Theta$ viewed as a differential graded category with trivial differential. As a consequence, the triangulated categories $D^b\bigl(\Coh(\mathbb{A})\bigr)$ and $D^b(\Theta)$ are equivalent, too.
\end{proposition}
\begin{proof} Let $\sS$ be a  Serre functor of the derived category $D^b\bigl(\Coh(\EE)\bigr)$. In \cite[Lemma 5.2]{BurbanKalck} it was observed that
$\sS(\kS_k^\pm) \cong \kS_k^\mp[2]$ for any $1 \le k \le 2n$. Moreover,
for $p \in \NN_0$ and $\varepsilon, \delta \in \left\{+, -\right\}$ we have: 
$$
\Ext_{\EE}^p\bigl(\kS_k^{\varepsilon}, \kS_k^{\delta}\bigr) = 
\left\{
\begin{array}{ccc}
\kk& \mbox{\rm if} & p = 0 \; \mbox{\rm and} \; \varepsilon = \delta \\
0 & \mbox{\rm if} & p = 2 \; \mbox{\rm and} \; \varepsilon \ne  \delta.
\end{array}
\right.
$$
In other words, for any $1 \le k \le 2n$, the pair of objects 
$S_k^+, S_k^-$ forms a generalized $2$-spherical collection. Let 
$\sT_k: D^b\bigl(\Coh(\EE)\bigr) \lar D^b\bigl(\Coh(\EE)\bigr)$ be the corresponding Seidel-Thomas twist functor. According to \cite[Proposition 2.10]{SeidelThomas} (see also 
\cite[Theorem 2]{BurbanBurban}, \cite[Remark 2.5]{Rouquier} and \cite{AnnoLogvinenko}), the functor $\sT_k$ is an auto-equivalence of $D^b\bigl(\Coh(\EE)\bigr)$.
For any $1 \le k, l \le 2n$ we have:
$$
\sT_{k}\bigl(\kS_l^\pm\bigr) \cong 
\left\{
\begin{array}{ccc}
\kS_l^\pm & \mbox{\rm if} & l \ne k \\
\kS_l^\mp[2] & \mbox{\rm if} & l =  k. \\
\end{array}
\right.
$$
It follows that the composition 
$\widetilde\sT := \sT_{1} \circ \sT_{3} \circ \dots \circ \sT_{2n-1}$ induces an equivalence of triangulated categories
$$
D^b(\Upsilon\mathsf{-mod}) \lar \bigl\langle \kS_1^+, \kS_2^-, \dots,
\kS_{2n-1}^+, \kS_{2n}^-\bigr\rangle^\perp 
\xrightarrow{\widetilde\sT}  \bigl\langle \kS_1^-, \kS_2^-, \dots,
\kS_{2n-1}^-, \kS_{2n}^-\bigr\rangle^\perp. 
$$
Let
$
\widetilde{A}_k^- := \bigl(\dots \lar 0 \lar P_{\phi_k^+}
\xrightarrow{b_k^+} P_{\beta_k} \xrightarrow{a_k^-} 
\underline{P_{\alpha_k}} \lar 0 \lar \dots \bigr).
$
Then we have: $\sT\bigl(\kS_{2k-1}^-\bigr) \cong \widetilde{A}_k^-$ for 
all $1 \le k \le n$. As a consequence, the categories $D^b(\Upsilon\mathsf{-mod})$
and $\bigl\langle A_1^-, \widetilde{A}_1^-, \dots, A_n^-, \widetilde{A}_n^- \bigr\rangle^\perp$ are equivalent.

\smallskip
\noindent
For any $1 \le k \le n$, consider the following object in $D^b(\Lambda\mathsf{-mod})$:
$$
\left\{
\begin{array}{ccl}
X_k & = &  \bigl(\dots \lar 0 \lar P_{\beta_k}
\xrightarrow{a_k^-} P_{\alpha_k}  \lar 0 \lar \dots \bigr) \\
Y_k & = & \bigl(\dots \lar 0 \lar P_{\gamma_k}
\xrightarrow{d_k^-} P_{\delta_k}  \lar 0 \lar \dots \bigr) \\
U_k & = &  \bigl(\dots \lar 0 \lar P_{\phi_k^-}
\xrightarrow{b_k^-} P_{\beta_{k+1}}  \lar 0 \lar \dots \bigr) \\
V_k & = &  \bigl(\dots \lar 0 \lar P_{\phi_k^+}
\xrightarrow{c_k^-} P_{\gamma_{k}}  \lar 0 \lar \dots \bigr). \\
\end{array}
\right.
$$
One can show that 
$$
G:= \bigoplus\limits_{k= 1}^n \bigl(X_k \oplus Y_k \oplus U_k \oplus V_k\bigr)
$$
generates the orthogonal category $\bigl\langle A_1^-, \widetilde{A}_1^-, \dots, A_n^-, \widetilde{A}_n^- \bigr\rangle^\perp$. Moreover, we have an isomorphism
of graded algebras $\Theta \cong \bigl(\Ext^*_{D^b(\Lambda)}(G)\bigr)^\circ$.

\smallskip
\noindent
A result of Bardzell (see \cite[Theorem 4.1]{Bardzell}) allows to write down a minimal resolution of $\Theta$ viewed as a module over its enveloping algebra
$\Theta^e := \Theta \otimes_\kk \Theta^\circ$. From the explicit form of this resolution one can conclude that $\mathsf{gl.dim}\,\Theta = 3$. Moreover, one can show that
in the category of \emph{graded} left $\Theta^e$-modules the following vanishing is true:
$$
\Ext^3_{\mathsf{gr}(\Theta^e)}\bigl(\Theta, \Theta(-1)\bigr) = 0.
$$
A result of Kadeishvili \cite{Kadeishvili} implies that the algebra $\Theta$ is \emph{intrinsically formal},  i.e.~that any minimal $A_\infty$-structure on $\Theta$ is equivalent to the trivial one.  According to Keller's work \cite{Keller}, the categories  $\bigl\langle A_1^-, \widetilde{A}_1^-, \dots, A_n^-, \widetilde{A}_n^- \bigr\rangle^\perp$ and $D^b(\Theta)$  are equivalent, which implies the statement.
\end{proof}

\begin{remark}
It follows from Bardzell's resolution \cite[Theorem 4.1]{Bardzell} that the third Hochschild cohomology $\mathsf{HH}^3(\Upsilon)$ is non-vanishing. Assume that 
$\Gamma$ is a finite dimensional algebra, which is derived equivalent to $\Upsilon$.
According to a result of Schr\"oer and Zimmermann \cite{SchZimm}, $\Gamma$ is a gentle algebra.
Rickard's derived Morita theorem  \cite{Rickard} implies that 
$\mathsf{HH}^3(\Gamma) \cong \mathsf{HH}^3(\Upsilon) \ne 0$ (see for instance
\cite[Section 2.4]{KellerOnHochschild} for a detailed argument). It follows from
\cite[Section 1.5]{HappelOnHochschild} that $\mathsf{gl.dim}(\Gamma) \ge 3$.
In particular, the gentle algebra $\Upsilon$  can  not be 
derived equivalent to a gentle algebra $\Lambda(\vec{p}, \approx)$ from Definition
\ref{D:DerivedTameAlgebras}. As a consequence, a Zhelobenko's non-commutative cycle of projective lines $\mathbb{A}$ is not derived equivalent to the Auslander curve of a non-commutative projective nodal curve.
\end{remark}

\section{AG-invariant of gentle algebras derived equivalent to  Auslander curves of tame nodal curves and homological mirror symmetry}

Let $\Lambda = \kk\vec{Q}/I$ be a gentle algebra \cite{AssemSkowr} (see  also 
\cite{BurbanDrozdGentle} for the definition and main properties of this  class of algebras). Let $Q_0$ be the set of vertices of $\vec{Q}$ and $Q_1$ be its set of arrows and  $s, t: Q_1 \lar Q_0$ be the maps attaching to each arrow its source and target, respectively. A path in $\vec{Q}$ is a sequence $\pi = a_m \dots a_1$ of elements of $Q_1$ such that $t(a_i) = s(a_{i+1})$ for all $1 \le i \le m-1$; $m = l(\pi)$ is the length of $\pi$. For any $\ast \in Q_0$ we have the trivial path $e_\ast$ of length zero with
$s(e_\ast) = t(e_\ast) = \ast$.
Following 
\cite{AAGeiss} we say that
\begin{itemize}
\item $\pi$ is a permitted path in $\Lambda$ if $a_{i+1} a_i \notin I$ for all $1 \le i \le l(\pi)-1$;
\item $\pi$ is a forbidden  path in $\Lambda$ if $a_{i+1} a_i \in  I$ for all $1 \le i \le l(\pi)-1$.
\end{itemize}
In this section  we assume that $\mathsf{gl.dim}(\Lambda) < \infty$. Then for any forbidden path $\pi$ in $\Lambda$ we have: $l(\pi) \le \mathsf{gl.dim}(\Lambda)$; see  \cite{Bardzell}. Now we recall the definition of the combinatorial invariant of $\Lambda$ of Avella-Alaminos and Gei\ss{} \cite{AAGeiss}.

\begin{definition}
The set $\Pi$ of \emph{permitted threads} in $\Lambda$ is the set  of all 
\begin{itemize}
\item maximal permitted paths in $\Lambda$;
\item trivial paths $e_\ast$, where $\ast \in Q_0$ is a vertex such that there exists at most one $a \in Q_1$ such that $t(a) = \ast$ and at most one
$b \in Q_1$ such that $s(b) = \ast$, and $b a \notin I$ provided both $a$ and $b$ do exist.
\end{itemize}
Similarly, the set $\Phi$ of \emph{forbidden threads} in $\Lambda$ is the set of all 
\begin{itemize}
\item maximal forbidden  paths in $\Lambda$;
\item trivial paths $e_\ast$, where $\ast \in Q_0$ is a vertex such that there exists at most one $a \in Q_1$ such that $t(a) = \ast$ and at most one
$b \in Q_1$ such that $s(b) = \ast$, and $b a \in I$ provided both $a$ and $b$ exist.
\end{itemize}
\end{definition}

\noindent
The assumption that the algebra $\Lambda$ is gentle implies that  there exists a bijection $ \Pi \stackrel{\vartheta_+}\lar \Phi$  
defined as follows.
\begin{itemize}
\item Let $\pi \in \Pi$ be a maximal permitted path. Then there exists a uniquely determined $\varphi \in \Phi$ such that $t(\pi) = t(\varphi)$ and either $l(\varphi) = 0$ or the terminating arrows of $\pi$ and $\varphi$ alter. We put $\varphi:= \vartheta_+(\pi)$. 
\item Let $\ast \in Q_0$ be such $e_\ast \in \Pi$. If there exists $b \in Q_1$ such that $t(b) = \ast$ then there exists a unique maximal forbidden path $\varphi$, whose terminating arrow is $b$. In this case we put $\varphi:= \vartheta_+(e_\ast)$. Otherwise, we set $\theta_+(e_\ast) = e_\ast$. 
\end{itemize}
Dually, we a have a bijection $\Phi \stackrel{\vartheta_-}\lar \Pi$ defined as follows. 
\begin{itemize}
\item Let $\varphi \in \Phi$ be a maximal forbidden  path. Then there exists a uniquely determined $\pi \in \Pi$ such that $s(\pi) = s(\varphi)$ and either $l(\pi) = 0$ or the starting  arrows of $\pi$ and $\varphi$ alter. We put $\pi:= \vartheta_-(\pi)$. 
\item Let  $\ast \in Q_0$ be  such that $e_\ast \in \Phi$.  If there exists $a \in Q_1$ such that $s(a) = \ast$ then there exists a unique maximal permitted  path $\pi$, whose starting  arrow is $a$. In this case we put $\pi:= \vartheta_-(e_\ast)$. Otherwise, we set $\theta_-(e_\ast) = e_\ast$. 
\end{itemize}
Consider the permutation  
\begin{equation}\label{E:keypermutation}
\vartheta:= \vartheta_+ \vartheta_-: \Phi\lar \Phi
\end{equation}
 Let $\vartheta = \vartheta_1 \circ \dots \circ \vartheta_r$ be a  cyclic decomposition of $\vartheta$. For each cycle 
$
\vartheta_i = \bigl(\varphi_1^{(i)}, \dots, \varphi_{m_i}^{(i)}\bigr)$,
$1 \le i \le r$, we put: 
$\kappa(\vartheta_i) := \bigl(m_i, l(\varphi_1^{(i)}) + \dots + l(\varphi_{m_i}^{(i)})\bigr).
$
\begin{definition} The Avella-Alaminos--Gei\ss{} invariant of $\Lambda$ (called \emph{AG-invariant} in what follows) is the function 
$\NN_0 \times \NN_0 \stackrel{\phi_\Lambda}\lar \NN_0$ defined by
the following rule:
$$
\phi_\Lambda(m,n) = \Big|\bigl\{ i \in \{1, \dots, r\}: \kappa(\vartheta_i) = (m, n)\bigr\}\Big|.
$$
\end{definition}

\noindent
It was shown in \cite[Theorem 16]{AAGeiss} that $\phi_\Lambda$ is a derived invariant of  $\Lambda$: if $\Lambda'$ is a finite dimensional algebra such that
$D^b(\Lambda\mathsf{-mod})$ and $D^b(\Lambda'\mathsf{-mod})$ are equivalent then we have: $\phi_\Lambda = \phi_{\Lambda'}$ (note that the algebra $\Lambda'$ is automatically gentle; see \cite{SchZimm}).

\begin{example}\label{Ex:AGInvariant1} Let $\Upsilon = \Upsilon_1$ be a gentle algebra from Theorem \ref{T:ZhelobenkoTilting}, i.e. the path algebra of the following quiver with relations:
\begin{equation}\label{E:GentleQ} 
\xymatrix
{
\circ  \ar@/^/[rr]^-{a_{+}} \ar@/_/[rr]_-{a_{-}} & & \circ \ar@/^/[rr]^-{b_{+}} \ar@/_/[rr]_-{b_{-}} &&
 \ar@/^/[rr]^-{c_{+}} \ar@/_/[rr]_-{c_{-}} 
 & & \circ & 
    \quad b_{\pm} a_{\mp} = 0 = c_{\pm} b_{\mp}}.
\end{equation}
Let $\pi_\pm = c_\pm b_\pm a_\pm$ and $\varphi_\pm = c_\mp b_\pm a_\mp$. Then we have: $\Pi = \left\{\pi_+, \pi_-\right\}$, $\Phi = \left\{\varphi_+, \varphi_-\right\}$ and the action of the bijections $\vartheta_\pm$  is given by the rules:
$
\vartheta_+(\pi_{\pm}) = \varphi_\pm$ and $
\vartheta_-(\varphi_\pm) = \pi_\pm.
$
From this we conclude that the AG-invariant of $\Upsilon$ is given by the formula
$$
\phi_\Upsilon(m, n) = 
\left\{
\begin{array}{cl}
2 & (m,n) = (1,3) \\
0 & \mbox{\rm otherwise}.
\end{array}
\right.
$$
\end{example}
\begin{example}\label{Ex:AGInvariant2} In the notation of Definition \ref{D:DerivedTameAlgebras},
let
$\Lambda = \Lambda\bigl((3,1), \approx\bigr)$, where $x_{1,1}^+ \approx
x_{1,2}^+$ and $x_{1,3}^+ \approx
x_{1,1}^-$. Then $\Lambda$ is the path algebra of the following quiver
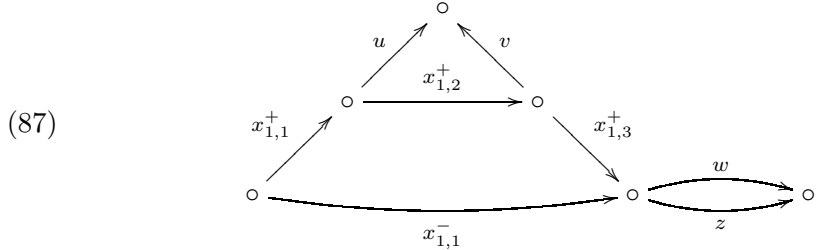
\begin{equation}
\begin{array}{c}
\xymatrix{
              &  & \circ & & \\
              & \circ  \ar[rr]^-{x_{1,2}^+} \ar[ru]^-{u}&  & \circ \ar[rd]^-{x_{1,3}^+} \ar[lu]_-{v}& & \\
\circ \ar[ru]^-{x_{1,1}^+} \ar@/_/[rrrr]_{x_{1,1}^-} &  &        &          & \circ    \ar@/^/[rr]^-{w}  \ar@/_/[rr]_-{z}  & &  \circ   &   \\
}
\end{array}
\end{equation}
subject to the relations $u x_{1,1}^+ = v x_{1,2}^+ = z x_{1,3}^+ = w x_{1,1}^- = 0$. We have:
$$
\Pi = \left\{z x_{1,3}^+, x_{1,2}^+ x_{1,1}^+, wx_{1,1}^-, v, u \right\} \quad
\mbox{\rm and} \quad
\Phi = \left\{u x_{1,1}^+, v x_{1,2}^+,  z x_{1,3}^+,  w x_{1,1}^- \right\}.
$$
Moreover, the  bijection $\Phi \stackrel{\vartheta}\lar \Phi$ is given by  the rule:
$
w x_{1,1}^- \stackrel{\vartheta}\mapsto z x_{1,3}^+ 
\stackrel{\vartheta}\mapsto u x_{1,1}^+ \stackrel{\vartheta}\mapsto
 w x_{1,1}^- $ and $v x_{1,2}^+ \stackrel{\vartheta}\mapsto  v x_{1,2}^+.
$
From this we conclude that the AG-invariant of $\Lambda$ is given by the formula
$$
\phi_{\Lambda}(m, n) 
= 
\left\{
\begin{array}{cl}
1 & (m,n) = (1,2) \; \mbox{\rm or}\, (3, 6)\\
0 & \mbox{\rm otherwise}.
\end{array}
\right.
$$
\end{example}

\noindent
Let $\vec{p} = \bigl((p_1^+, p_1^-), \dots, (p_r^+, p_r^-)\bigr) \in \bigl(\NN^2\bigr)^r$ and  $\approx$ be a relation on the set $\Xi = \Xi(\vec{p})$
without reflexive elements  as in Definition \ref{D:DerivedTameAlgebras}. For any $x \in \Xi$ we set:
$$
w(x) := 
\left\{
\begin{array}{cl}
2 \ & \mbox{\rm if} \; x \; \mbox{\rm is tied}\\
1 & \mbox{\rm otherwise}.
\end{array}
\right. 
$$
For any cyclic permutation $\varpi = (x_1, \dots, x_t)$ of elements of $\Xi$ we put $$\kappa(\varpi) := \bigl(t, w(x_1) + \dots + w(x_t)\bigr).$$
Consider two permutations $\sigma, \tau: \Xi \lar \Xi$, given by the formulae
\begin{equation}\label{E:sigma}
\sigma(x) := 
\left\{
\begin{array}{cl}
x'\ & \mbox{\rm if there  exists}\,  x'  \,\mbox{\rm such that}\, x' \approx x\\
x & \mbox{\rm otherwise}
\end{array}
\right. 
\end{equation}
and
$ 
\tau\bigl(x_{i, j}^\pm\bigr) := x_{i, j-1}^\pm$
for all  $1 \le i \le r, 1 \le j \le p_i^\pm$, where $x_{i, 0}^\pm := x_{i, p_i^\pm}^\pm.
$ Finally, let $\varrho= \sigma \circ \tau$ and 
$
\varrho= \varrho_1 \circ \dots \circ \varrho_c
$
be its cyclic decomposition.

\begin{lemma}\label{L:AGinvariant} Let   $ \Lambda = \Lambda(\vec{p}, \approx)$ be a gentle algebra from Definition \ref{D:DerivedTameAlgebras}.  Then the  AG-invariant of $\Lambda$ is given by the following formula:
$$
\phi_\Lambda(m,n) = \Big|\bigl\{ i \in \{1, \dots, c\}: \kappa(\varrho_i) = (m, n)\bigr\}\Big|.
$$
\end{lemma}
\begin{proof}
Since $\mathsf{gl.dim}(\Lambda) = 2$, all forbidden threads in $\Lambda$ have length at most two. Moreover, we have a bijection $\Xi \lar \Phi$,
sending an element $x\in \Xi$ to itself if $x$ is untied and to the unique forbidden path of length two containing $x$ if $x$ is tied. It is a straightforward verification that the permutation $\varrho$ defined above gets identified 
with the permutation  $\vartheta$ from the definition of the AG-invariant.
\end{proof}

\smallskip
\noindent
Let $\Sigma$ be a compact  oriented surface with non-empty boundary. According to the classification of surfaces, $\Sigma$ is determined (up to a homeomorphism)  by its genus $g \in \NN_0$ and the number $b \in \NN$ of the boundary components. In what follows, we shall denote such surface by $\Sigma_{g, b}$.

\smallskip
\noindent
Next, let  $\mathfrak{S} \subset \partial \Sigma$ be a finite subset such that
$\mathfrak{S}_i := \mathfrak{S} \cap \partial_i \Sigma \ne \emptyset$ for all $1 \le i \le b$, where $\partial_i \Sigma$ is the $i$-th boundary component. Let $P_\Sigma$ be the projectivized tangent bundle of $\Sigma$ and $\eta \in \Gamma(\Sigma, P_\Sigma)$ be a line field (it follows from the assumption $b \ge 1$ that $\Gamma(\Sigma, P_\Sigma)\ne \emptyset$). 

\smallskip
\noindent
In \cite[Chapter 3]{HaidenKatzarkovKontsevich}, Haiden, Katzarkov and Kontsevich attached to 
the datum $(\Sigma, \mathfrak{S}, \eta)$ a certain triangulated category 
$\mathsf{WFuk}(\Sigma, \mathfrak{S}, \eta)$ called \emph{partially wrapped Fukaya category}. Here, we follow the terminology of the papers of Lekili
and Polishchuk \cite{LekiliPolishchuk, LekiliPolishchukII} (in the original version of  \cite{HaidenKatzarkovKontsevich} this category  was called topological Fukaya category). According to  \cite[Lemma 3.3]{HaidenKatzarkovKontsevich}, for any datum $(\Sigma, \mathfrak{S}, \eta)$ there exists a
homologically smooth   \emph{graded} gentle algebra $\Lambda$ and an equivalence of triangulated categories
$
\mathsf{WFuk}(\Sigma, \mathfrak{S}, \eta) \simeq D^b_{dg}(\Lambda),
$
where $D^b_{dg}(\Lambda)$ is the derived category of $\Lambda$ viewed as a graded dg-algebra with trivial differential. Note that if the degrees of all arrows of $\Lambda$ are zero then $D^b_{dg}(\Lambda)$ is  equivalent to the usual derived category $D^b(\Lambda\mathsf{-mod})$ of the abelian category  of finite dimensional representations of $\Lambda$; see \cite{Keller}.

In  \cite{LekiliPolishchukII}, Lekili and Polishchuk proved that conversely,  for any homologically smooth  graded gentle algebra $\Lambda$ there exists a datum $(\Sigma, \mathfrak{S}, \eta)$ as above such that
$\mathsf{WFuk}(\Sigma, \mathfrak{S}, \eta)$ is equivalent to $D^b_{dg}(\Lambda)$. We are mostly interested in the case when the grading of a gentle algebra $\Lambda = \kk \vec{Q}/I$ is trivial, i.e. we have:
$$D^b(\Lambda\mathsf{-mod}) \simeq 
\mathsf{WFuk}(\Sigma, \mathfrak{S}, \eta).$$
Let $\chi(\Sigma)$ be the Euler characteristic of $\Sigma$ and $b$ be the number of its boundary components. Recall at this place that $
\chi(\Sigma_{g,b}) = 2(1-g) - b.$
According to 
\cite[Theorem 3.2.2]{LekiliPolishchukII}, we have 
 \begin{equation}\label{E:NumberBoundaryComp}
\chi(\Sigma) = |Q_0|-|Q_1| \quad \mbox{\rm and} \quad
b = \sum\limits_{(m, n) \in \NN_0^2} \phi_\Lambda(m, n).
\end{equation}

\smallskip
\noindent
Formula (\ref{E:NumberBoundaryComp}) tells that
the permutation $\vartheta$ given by (\ref{E:keypermutation}) splits into a product of precisely $b$ cycles: $\vartheta = \vartheta_1 \circ \dots \circ \vartheta_b$. Lekili and Polishchuk also prove that the order of the boundary components
$\partial_1\Sigma, \dots, \partial_b\Sigma$  can be chosen  in such a way that 
$
\big|\mathfrak{S}_i\big| = m_i$ and $
w_\eta^{(i)} = m_i-n_i
$
for all $1 \le i \le b$,
where $(m_i, n_i) = \kappa(\vartheta_i)$  and $w_\eta^{(i)}$ is the winding number of the line field $\eta$ along 
$\partial_i\Sigma$; see  \cite[Theorem 3.2.2]{LekiliPolishchukII}.
Since $\phi_\Lambda$  is a derived invariant, it implies   that the \emph{marked surface} $(\Sigma, 
\mathfrak{S})$ is a derived invariant of $\Lambda$, too. See also a work of Opper, Plamondon and Schroll \cite{OPS} for an alternative approach to attach a marked surface  to a gentle algebra of possibly infinite global dimension.

\begin{remark}
Let $\Lambda = \Lambda(\vec{p}, \approx) = \kk\vec{Q}/I$. It is easy to see that $|Q_1|-|Q_0|$ is equal to the number of two-cycles of the permutation $\sigma$, defined by (\ref{E:sigma}). This  gives  a closed formula for the Euler characteristic of the surface of $\Lambda$. 
\end{remark}

\begin{example}
Let $(\Sigma, \mathfrak{S})$ be the marked surface of  the gentle algebra
$\Upsilon$  from Example \ref{Ex:AGInvariant1}. Then $\Sigma = \Sigma_{1, 2}$ is a torus with two boundary components. Moreover, $|\mathfrak{S}| = 2$ and 
on each boundary component of $\Sigma$ lies one marked point. If $\eta$ is a line field on $\Sigma$ such that 
$D^b(\Upsilon\mathsf{-mod}) \simeq 
\mathsf{WFuk}(\Sigma, \mathfrak{S}, \eta)$ then 
$w_\eta^{(i)} = -2$ for $i = 1,2$. We have in this case:
$$
\mathsf{WFuk}(\Sigma, \mathfrak{S}, \eta) \simeq D^b\bigl(\Coh(\mathbbm{A})\bigr),
$$
where $\mathbbm{A}$ is the Zhelobenko nodal curve, whose central curve is a cycle of two projective lines; see Subsection \ref{SS:ZhelobenkoCurves}.
\end{example}

\begin{example} Let $(\Sigma, \mathfrak{S})$ be the marked surface of  the gentle algebra $\Lambda$ from Example \ref{Ex:AGInvariant2}. Again, we have:  $\Sigma = \Sigma_{1, 2}$. However, this time, 
there is only one marked point lying 
 on $\partial_{1}\Sigma$   and $w_{\eta}^{(1)} = -1$, whereas  there are  three marked points on $\partial_{2}\Sigma$ and $w_{\eta}^{(2)} = -3$. According to \cite[Corollary 3.25]{LekiliPolishchukII}, in this case the AG-invariant $\phi_\Lambda$ is a full derived invariant of $\Lambda$. Equivalently, for any line field $\tilde\eta \in \Gamma(\Sigma, P_\Sigma)$ such that $w_{\tilde\eta}^{(1)} = -1$ and $w_{\tilde\eta}^{(2)} = -3$ we have: $\mathsf{WFuk}(\Sigma, \mathfrak{S}, \eta) \simeq \mathsf{WFuk}(\Sigma, \mathfrak{S}, \tilde\eta)$.
\end{example}

\begin{remark}
Let $\XX =\XX(\vec{p}, \approx)$ be a tame non-commutative nodal projective curve of gentle type, $\YY$ be the corresponding Auslander curve and $\Lambda = 
\Lambda(\vec{p}, \approx)$ be the corresponding gentle algebra. According to
\cite[Theorem 3.2.2]{LekiliPolishchukII}, there exists a uniquely determined oriented compact marked surface
$(\Sigma, \mathfrak{S})$ with non-empty boundary and a line field $\eta$ on $\Sigma$ such 
that the triangulated categories $D^b(\Lambda\mathsf{-mod})$ and $\mathsf{WFuk}(\Sigma, \mathfrak{S}, \eta)$ are equivalent. In combination 
with Theorem \ref{T:tilting1}, we get an equivalence of triangulated categories
\begin{equation}
D^b\bigl(\Coh(\YY)\bigr) \simeq \mathsf{WFuk}(\Sigma, \mathfrak{S}, \eta),
\end{equation}
what gives  a version of the homological mirror symmetry for  oriented compact surfaces with non-empty boundary, generalizing   \cite[Theorem A]{LekiliPolishchuk}. If $\XX =\XX(\vec{p}, \approx)$ is a stacky chain or cycle of projective lines (see Example \ref{Ex:StackyCycles}) then all winding numbers of the line field $\eta$ along of  all boundary components of $\Sigma$ are even; see \cite[Section 4]{LekiliPolishchukII}. However, in the case of the datum from Example  \ref{Ex:AGInvariant2}, the  winding numbers of the line field  $\eta$ along both boundary components are odd. This implies that the class of stacky chains and stacky cycles of projective lines  does not exhaust all  derived equivalence classes of tame nodal curves of gentle type. 
 \end{remark}

\end{document}